\documentclass[11pt,a4paper]{article}

\usepackage[utf8]{inputenc}\usepackage[english]{babel}

\usepackage{amsmath}
\usepackage{amsfonts}
\usepackage{amssymb}

\usepackage{graphicx}
\usepackage[left=2cm,right=2cm,top=2cm,bottom=2cm]{geometry}
\usepackage{authblk}  % Authors
\usepackage{xspace}  % White space in newcommand
\usepackage{hyperref}  %references
\usepackage[all,cmtip]{xy}

\usepackage{mathtools} 

\usepackage[style=numeric,maxbibnames=99, backend=biber]{biblatex}
\renewbibmacro{in:}{}
\usepackage{csquotes}
\DeclareFieldFormat*{title}{\mkbibemph{#1}}
\DeclareFieldFormat[book]{date}{\mkbibparens{#1}}
\DeclareFieldFormat[thesis, phdthesis, inproceedings]{date}{\mkbibparens{#1}} 
\DeclareFieldFormat{journaltitle}{#1}

\usepackage{tikz}
\usetikzlibrary{positioning,fit}

\usepackage{tikz-cd} 
\newcommand\mlnode[1]{\fbox{{\renewcommand{\arraystretch}{1.1}\begin{tabular}{@{}c@{}}#1\end{tabular}}}}

\usepackage{fancyhdr} 
\fancyheadoffset{0cm}

\fancypagestyle{plain}{%
   \fancyhf{}%
   \fancyhead[lef,rof]{\thepage}%
}

\usepackage{xcolor}
\usepackage{float}

\title{Classification of classical twists of the standard Lie bialgebra structure on a loop algebra}
\date{\today}
\author[1]{Raschid Abedin}
\author[2]{Stepan Maximov}
\affil[1]{Institut für Mathematik,
Universität Paderborn, 33098 Paderborn,
Germany}
\affil[2]{Department of Mathematical Sciences, 
Chalmers University of Technology  and \protect\\ the University of Gothenburg, 412 96 Gothenburg, Sweden}

\newcommand{\ZZ}{\mathbb{Z}}

\newcommand{\RR}{\mathbb{R}}
\newcommand{\CC}{\mathbb{C}}

\newcommand{\g}{\mathfrak{g}}
\newcommand{\n}{\mathfrak{n}}
\newcommand{\h}{\mathfrak{h}}
\newcommand{\p}{\mathfrak{p}}

\newcommand{\B}{\mathfrak{B}}
\newcommand{\N}{\mathfrak{N}}
\newcommand{\Loop}{\mathfrak{L}}

\newcommand{\K}{\kappa}

\newcommand{\add}{\dotplus}
\newcommand{\ot}{\otimes}

\newcommand{\set}[1]{\left\{ #1 \right\}}
\newcommand{\seq}[1]{\left( #1 \right)}

\newcommand{\dual}[1]{#1^{\vee}}
\DeclareMathOperator{\adOp}{ad}
\newcommand{\ad}{\mathord{\adOp}}
\DeclareMathOperator{\imOp}{im}
\newcommand{\im}{\mathord{\imOp}}
\DeclareMathOperator{\diagOp}{diag}
\newcommand{\diag}{\mathord{\diagOp}}
\DeclareMathOperator{\ordOp}{ord}
\newcommand{\ord}{\mathord{\ordOp}}
\DeclareMathOperator{\hgtOp}{ht}
\newcommand{\hgt}{\mathord{\hgtOp}}
\DeclareMathOperator{\idOp}{id}
\newcommand{\id}{\mathord{\idOp}}
\DeclareMathOperator{\spanOp}{span}
\newcommand{\Span}{\mathord{\spanOp}}

\newcommand{\rmatr}{\ensuremath{r}\text{-matrix}\xspace}
\newcommand{\rmatrs}{\ensuremath{r}\text{-matrices}\xspace}

\newcommand\restr[2]{{%
  \left.\kern-\nulldelimiterspace%
  #1%
  \vphantom{\big|}%
  \right|_{#2}%
  }}

\usepackage{amsthm}

\numberwithin{equation}{section}
\renewcommand*{\theequation}{%
  \ifnum\value{section}>0 %
    \thesection.%
  \fi
    \arabic{equation}%
}

\makeatletter
\def\th@plain{%
  \thm@notefont{}% same as heading font
  \itshape % body font
}
\def\th@definition{%
  \thm@notefont{}% same as heading font
  \normalfont % body font
}
\makeatother

\newtheorem{theoremcounter}{theoremcounter}[section]
\newtheorem{maintheoremcounter}{maintheoremcounter}

\theoremstyle{plain}

\newtheorem{lemma}[theoremcounter]{Lemma}

\newtheorem{theorem}[theoremcounter]{Theorem}

\newtheorem{maintheorem}[maintheoremcounter]{Theorem}

\newtheoremstyle{example_style}
  {\topsep} % Space above
  {\topsep} % Space below
  {} % Body font
  {} % Indent amount
  {\itshape} % Theorem head font
  {.} % Punctuation after theorem head
  {.5em} % Space after theorem head
  {} % Theorem head spec (can be left empty, meaning `normal')

\theoremstyle{example_style}

\theoremstyle{definition}

\theoremstyle{example_style}
\newtheorem{remarkx}[theoremcounter]{Remark}
\newenvironment{remark}
  {\pushQED{\qed}\remarkx}
 {\popQED\endremarkx}

\begin{document}
\clearpage
\maketitle
\thispagestyle{empty}

\begin{abstract}
The standard Lie bialgebra structure on an affine Kac-Moody algebra induces a Lie bialgebra structure on the underlying loop algebra and its parabolic subalgebras. 
In this paper we classify all classical twists of the induced Lie bialgebra structures in terms of Belavin-Drinfeld quadruples up to a natural notion of equivalence.
To obtain this classification we first show that the induced bialgebra structures are defined by certain solutions of the classical Yang-Baxter equation (CYBE) with two parameters. Then, using the algebro-geometric theory of CYBE, based on torsion free coherent sheaves, we reduce the problem to the well-known classification of trigonometric solutions given by Belavin and Drinfeld.
The classification of twists in the case of parabolic subalgebras allows us to answer recently posed open questions regarding the so-called quasi-trigonometric solutions of CYBE.

\iffalse
    In this paper,
        we investigate the twisted versions of the standard Lie bialgebra structure on loop algebras, which we will call twisted standard structures. These turn out to be pseudoquasitriangular defined by certain trigonometric \rmatrs, which represent all equivalence classes of trigonometric solutions in the sense of the classification given in \cite{belavin_drinfeld_solutions_of_CYBE_paper}. This can be viewed as the natural bialgebraic framework of trigonometric \rmatrs. Applying the algebro-geometric results in \cite{burban_galinat, TBA}, we will reduce the classification of twisted standard structures up to certain regular equivalences to the classification of trigonometric \rmatrs, which is given up to a holomorphic analog of these equivalences. The standard structure restricts to a bialgebra structure on each parabolic subalgebra. We will apply the classification of twisted standard structures to obtain a classification of the twisted versions of this Lie bialgebra, generalising the results in \cite{KPSST, pop_stolin}.
\fi
\end{abstract}

\section{Introduction}
A Lie bialgebra is a pair \( (L, \delta) \)
consisting of a Lie algebra \( L \) and a linear map
\( \delta \colon L \longrightarrow L \ot L \), called Lie cobracket, inducing a compatible Lie
algebra structure on the dual space \( \dual{L} \).
This notion originated in
\cite{drinfeld_hamiltonian_structures}
as the infinitesimal counterpart of a Poisson Lie group. Shortly after, in \cite{drinfeld_hopf_algebras,drinfeld_quantum_groups},
Lie bialgebras were described as quasi-classical limits of certain quantum groups and received a fundamental role in the quantum group theory.

Having a Lie bialgebra structure 
\( \delta \) on a Lie algebra \( L \)
we can obtain new Lie bialgebra structures
using a procedure called twisting. More precisely, let
\(t \) be a skew-symmetric tensor in \( L \ot L\) satisfying 
\begin{equation*}\label{eq: twisting condition on t}
  \textnormal{CYB}(t)
  = \textnormal{Alt}((\delta \otimes 1)t),
\end{equation*}
where
\begin{equation*}\label{eq: CYB and Alt}
\begin{split}
    \textnormal{CYB}(t)
  &\coloneqq
  [t^{12}, t^{13}] +
  [t^{12}, t^{23}] +
  [t^{13}, t^{23}], \\
  \textnormal{Alt}(x_1 \ot x_2 \ot x_3)
  &\coloneqq x_1 \ot x_2 \ot x_3 + 
    x_2 \ot x_3 \ot x_1 + 
    x_3 \ot x_1 \ot x_2
\end{split}
\end{equation*}
and, for example, \([(a\ot b)^{12},(c \ot d)^{23}] \coloneqq a \ot [b,c] \ot d\).
Then the linear map
\( \delta_t \coloneqq \delta + dt \)
is a Lie bialgebra structure on \(L\). Such a tensor \( t \) is called a classical twist.

The most important example of a Lie bialgebra structure is the standard structure \(\delta\) on
a symmetrizable Kac-Moody algebra \(\mathfrak{K} \coloneqq \mathfrak{K}(A)\) introduced in \cite{drinfeld_hopf_algebras}. In the case when the Cartan matrix \(A\) is of finite type or, equivalently, when
\( \mathfrak{K} \) is a finite-dimensional semi-simple Lie algebra, the standard structure \( \delta \) and all its twisted versions 
\( \delta_t \) are known to be quasi-triangular,
i.e. they are of the form \( dr \) for some
\( r \in \mathfrak{K}\ot \mathfrak{K}\)
satisfying the classical Yang-Baxter equation \(
\textnormal{CYB}(r) = 0\).

When the matrix \( A \) is of affine type,
the standard structure on \( \mathfrak{K} \)
induces a Lie bialgebra structure on \(  [\mathfrak{K}, \mathfrak{K}]/Z(\mathfrak{K}) \), where \(Z(\mathfrak{K})\) is the center of \(\mathfrak{K}\), which we will also call standard. The latter Lie algebra is known (see \cite{kac_book}) to be isomorphic to the loop algebra \( \Loop^\sigma \) over a simple finite-dimensional Lie algebra \( \g \) corresponding to an automorphism
\( \sigma \in \nolinebreak \textnormal{Aut}_{\CC - \textnormal{LieAlg}}(\g) \) of finite order \( m \). It has the following explicit description
\begin{align*}
  \Loop^\sigma = \set{f \in \g[z,z^{-1}] \mid f(\varepsilon_\sigma z) = \sigma(f(z))},
  \qquad
  \varepsilon_\sigma \coloneqq \exp(2\pi i / m).
\end{align*}
We denote the induced standard Lie bialgebra structure on
\( \Loop^\sigma \) by \( \delta^\sigma_0 \) and
call its twists \( \delta_t^\sigma \coloneqq \delta_0^\sigma + dt \)
twisted standard structures. These Lie bialgebra structures are not quasi-triangular, but pseudoquasitriangular as is shown in
Theorem \ref{thm: delta = [, r]}, i.e. they are defined by
meromorphic functions \( r \colon \CC^2 \longrightarrow \g \ot \g \), also known as \rmatrs,
satisfying
the two-parametric classical Yang-Baxter equation (CYBE)
\begin{equation*}
    \textnormal{CYB}(r)(x_1,x_2,x_3) \coloneqq [r^{12}(x_1, x_2), r^{13}(x_1, x_3)] +
    [r^{12}(x_1, x_2), r^{23}(x_2, x_3)] +
    [r^{13}(x_1, x_3), r^{23}(x_2, x_3)] = 0.
\end{equation*}
For example,
%These \rmatrs turn out to be trigonometric in the sense of the Belavin-Drinfeld classification.
the trigonometric \rmatr
given by a Belavin-Drinfeld (BD) quadruple \( Q \), corresponding to an outer automorphism \(\nu \in \textnormal{Aut}_{\CC - \textnormal{LieAlg}}(\g)\) (see \cite{belavin_drinfeld_solutions_of_CYBE_paper}),
gives rise to a twisted standard structure \(\delta^\sigma_Q\) on \( \Loop^\sigma \) for any finite order automorophism \( \sigma \) whose coset is conjugate to \(\nu \textnormal{Inn}_{\CC - \textnormal{LieAlg}}(\g)\). 
It turns out that any \rmatr defining a twisted standard bialgebra structure on \( \Loop^\sigma \) is globally holomorphically equivalent to a trigonometric solution in the sense of
the Belavin-Drinfeld classification (see Theorem \ref{thm: r(x,y)=X(u-v)}). We refer to such \rmatrs as
\( \sigma \)-trigonometric.

We call two twisted standard structures \( \delta^\sigma_t \) and \( \delta^\sigma_s \) (regularly) equivalent
if there is a function
\begin{equation*}
    \phi \in \textnormal{Aut}_{\CC[z^m,z^{-m}] - \textnormal{LieAlg}} (\Loop^\sigma)
    =
    \set{f \colon \CC^* \longrightarrow \textnormal{Aut}_{\CC - \textnormal{LieAlg}}(\g) \mid f \text{ is regular and } f(\varepsilon_\sigma z) = \sigma f(z)\sigma^{-1}},
\end{equation*}
called a regular equivalence, such that
\( \delta^\sigma_t \phi = (\phi \ot \phi) \delta^\sigma_s \).
The main result of this paper is the classification of twisted standard structures up to regular equivalence.
The classification is obtained by reducing our problem to the classification of trigonometric \rmatrs up to holomorphic equivalence given in 
\cite{belavin_drinfeld_solutions_of_CYBE_paper}.
To deal with the difference between the notions of equivalence we use the geometric formalism of CYBE presented in \cite{burban_galinat}.
More precisely, one of the key results in \cite{burban_galinat} is that certain coherent sheaves of Lie algebras on Weierstraß cubic curves give rise to so-called geometric \rmatrs, satisfying a geometric version of CYBE. In Section \ref{sec: Proof} we 
prove the following extension property:
%
%to the classification of trigonometric \rmatrs up to holomorphic equivalence in \cite{belavin_drinfeld_solutions_of_CYBE_paper}. To do so, we will apply the geometric formalism of the CYBE from \cite{burban_galinat}.
%Roughly speaking, it is shown there that certain coherent sheaves of Lie algebras on Weierstraß cubic curves give rise to so called geometric \rmatrs satisfying a geometric version of the CYBE. In this framework, we will derive the following result (see Section \ref{sec: Proof} and more specifically Theorem \ref{thm: lifting equivalences}).
\begin{maintheorem}\label{thm: intro theorem extending equivalence}
A formal equivalence of geometric \rmatrs at the smooth point at infinity of the Weierstraß cubic curve gives rise to an isomorphism of the corresponding sheaves of Lie algebras.
\end{maintheorem}
It is shown in \cite{TBA} that all \(\sigma\)-trigonometric \rmatrs arise as geometric \rmatrs from coherent sheaves of Lie algebras on the nodal Weierstraß cubic with section \(\Loop^\sigma\) on the set of smooth points.
Since holomorphic equivalences are formal, this result and Theorem \ref{thm: intro theorem extending equivalence} give the desired classification:
\begin{maintheorem}\label{thm: theorem A}
For any twisted standard structure \( \delta_t^\sigma \) there is a 
regular equivalence 
\( \phi \) of \( \Loop^\sigma \) and
a BD quadruple
\( Q = \seq{\Gamma_1, \Gamma_2, \gamma, t_\h} \)
such that
\begin{align*}
\delta^\sigma_t \phi 
= 
(\phi \ot \phi) \delta^\sigma_Q.
\end{align*}
Furthermore, if
\( Q' = (\Gamma'_1, \Gamma'_2, \gamma', t'_\h) \)
is another BD
quadruple, then
the twisted bialgebra structures
\( \delta^\sigma_Q \) and
\( \delta^\sigma_{Q'} \)
are regularly equivalent if and only if
there is
an automorphism \( \vartheta \)
of the Dynkin diagram
of \( \Loop^\sigma \) 
such that
\( \vartheta(\Gamma_i) = \Gamma'_i \)
for \( i = 1,2 \),
\( \vartheta \gamma \vartheta^{-1} = \gamma' \) and
\( (\vartheta \ot \vartheta)t_\h = t'_\h \).
\end{maintheorem}

Let \( \Pi^\sigma \) be the set of simple roots of \( \Loop^\sigma \),  \( S \subsetneq \Pi \) and
\( \p^S_+ \subseteq \Loop^\sigma \) be the corresponding parabolic subalgebra (see Section \ref{sec: Loop algebras}).
The standard Lie bialgebra structure \( \delta^\sigma_0 \) on \( \Loop^\sigma \) restricts
to a Lie bialgebra structure on the parabolic subalgebra \( \p^S_+ \). We refer to this Lie bialgebra structure as the restricted standard structure.

In the special case \( \sigma = \id \) and 
\( S = \Pi^{\id} \setminus \{ \widetilde{\alpha}_0 \} \), where
\( \widetilde{\alpha}_0 \) is the affine root of
\( \Loop^{\id} = \g[z, z^{-1}] \),
the classical twists of the restricted standard structure are in one-to-one correspondence with so-called
quasi-trigonometric solutions of CYBE. 
Such \rmatrs were studied and classified in terms of BD quadruples
%up to polynomial equivalence
in \cite{KPSST, pop_stolin}.
We use this classification in Section \ref{sec: sl(n,C) case} to demonstrate the first part of Theorem
\ref{thm: theorem A} 
in the special case \( \g = \mathfrak{sl}(n, \CC ) \).
%The classification of quasi-trigonometric \rmatrs in \cite{KPSST}
%for \( \g = \mathfrak{sl}(n, \CC) \)
%to prove the first part of Theorem \ref{thm: theorem A} (see Section \ref{sec: sl(n,C) case}). 
This connection to quasi-trigonometric solutions serves as a motivation for our study of restricted standard structures.
\iffalse
Choosing \( \sigma = \id \) and letting
\( S \) to be the set of all simple roots of \( \Loop^{\id} = \g[z,z^{-1}] \) except the affine root, the corresponding parabolic subalgebra \( \p^S_+ \) becomes \( \g[z] \). In this case
the twisted versions of \( \delta^{\id}_0 |_{\g[z]} \) are defined by so-called quasi-trigonometric \rmatrs, which were studied and classified
in \cite{KPSST, pop_stolin}. 

In this case
the classical twists of \( \delta^{\id}_0 |_{\g[z]} \) are in one-to-one correspondence with so called quasi-trigonometric \rmatrs, which were studied and classified up to polynomial equivalence in \cite{KPSST, pop_stolin}. 

In the important case \( \g = \mathfrak{sl}(n, \CC)\) we can even reduce the classification of twisted standard structures on \(\Loop^{\id}\) to that classification (see Section \ref{sec: sl(n,C) case}).
\fi

We discover that Theorem \nolinebreak \ref{thm: theorem A} also gives a full classification of classical twists of restricted Lie bialgera structures.
More formally, for any classical twist 
\( t \in \p^S_+ \ot \p^S_+ \) the structure of \(\Loop^\sigma\) guarantees that the regular equivalence between \(\delta^\sigma_t\) and some \(\delta^\sigma_Q\), given by Theorem \ref{thm: theorem A}, can be chosen to fix the parabolic subalgebra \(\p^S_+\). The following theorem summarizes this observation.

\begin{maintheorem}\label{thm: theorem B}
For any classical twist
\( t \in \p_{+}^S \ot \, \p_{+}^S \)
of the standard Lie bialgebra structure
\( \delta_0^\sigma \) on \( \Loop^\sigma \)
there exists a regular equivalence \( \phi \)
that restricts to an automorphism of
\( \p_+^S \)
and a BD quadruple \( Q = \seq{\Gamma_1, \Gamma_2, \gamma, t_\h} \) such that
\begin{equation*}
\Gamma_1 \subseteq S
\ \text{ and } \
\delta_t^\sigma \phi = (\phi \ot \phi) \delta^\sigma_{Q}.
\end{equation*}
Let \( Q' = (\Gamma'_1, \Gamma'_2, \gamma', t'_\h) \),
\( \Gamma'_1 \subseteq S \), be another BD
quadruple.
A regular equivalence between twisted standard structures
\( \delta_Q^\sigma \) 
and \( \delta_{Q'}^\sigma \) restricts to an automorphism of \( \p^S_+ \)
if and only if the induced Dynkin diagram automorphism \( \vartheta \)
preserves \( S \), i.e.
\( \vartheta(S) = S \).
\end{maintheorem}

%Choosing \( \sigma = \id \) and letting
%\( S \) to be the set of all simple roots of \( %\Loop^{\id} = \g[z,z^{-1}] \) except the affine root, %the corresponding parabolic subalgebra \( \p^S_+ \) %becomes \( \g[z] \).
%In this case, classical twists of the restricted %bialgebra structure are in one-to-one correspondence %with so called quasi-trigonometric \rmatrs, introduced %in \cite{KPSST} and classified in \cite{pop_stolin, %KPSST} up to polynomial equivalence.
%The regular equivalences in Theorem \ref{thm: theorem %B} are precisely the polynomial equivalences between %quasi-trigonometric \rmatrs. In other words, 

The theorems stated above provide us with a list of interesting consequences:
\begin{itemize}
  \item Letting \( \sigma = \id \) and \( S = \Pi^{\id} \setminus \{ \widetilde{\alpha}_0 \} \) in Theorem \ref{thm: theorem B} we obtain an alternative proof of
  the classification of all quasi-trigonometric solutions \cite{KPSST, pop_stolin};
  \item The necessary and sufficient condition to have an \( \id \)-trigonometric \rmatr which is not regularly equivalent to a quasi-trigonometric one is the existence of a BD qudruple \( (\Gamma_1, \Gamma_2, \gamma, t_\h ) \) such that for any automorphism \( \vartheta \) of the extended Dynkin diagram of \( \g \) we have
   \( \widetilde{\alpha}_0 \in \vartheta(\Gamma_1) \).
  Analyzing Dynkin diagrams, we conclude that
  any \( \id \)-trigonometric \rmatr is regularly equivalent to a quasi-trigonometric one
  if and only if
  \( \g \) is of type
  \( A_n, C_n, B_{2-4} \) or \( D_{4-10} \);
  \item We have mentioned that any \( \sigma \)-trigonometric \rmatr is holomorphically equivalent to a trigonometric one in the sense of the Belavin-Drinfeld classification.
  Combining the structure theory of \( \Loop^\sigma \) (Section \ref{sec: Loop algebras}) and Theorem \ref{thm: theorem A} we can improve that result and get more control over that equivalence.
  More precisely,
  let \( \nu \) be an outer automorphism of \( \g \), \( \sigma \) be a finite order automorphism of \( \g \) whose coset is conjugate to \( \nu \textnormal{Inn}_{\CC - \textnormal{LieAlg}}(\g) \) and
  \( r_t^\sigma \) be 
  the \( \sigma \)-trigonometric \rmatr
  defining a twisted standard Lie bialgebra
  structure \( \delta_t^\sigma \) on 
  \( \Loop^\sigma \).
  Applying to \( r_t^\sigma\)
  the regular equivalence, given by Theorem \ref{thm: theorem A},
  and regrading to the principle grading, i.e. grading corresponding to the Coxeter automorphism \( \sigma_{(\mathbf{1};|\nu|)} \),
  we obtain a trigonometric \rmatr \( X \) depending on the quotient of its parameters:
  \begin{equation*}
      r^\sigma_t(x,y) \xmapsto{\text{regular eq.}} r^\sigma_Q(x,y) \xmapsto{\text{regrading}} r_Q^{\sigma_{(\mathbf{1};|\nu|)}}(x,y) = X(x/y);
  \end{equation*}
  \item We answer questions one and two posed at the end of \cite{burban_galinat_stolin}
  concerning an explicit formula for the
  quasi-trigonometric solution given
  by a BD quadruple \( Q \) and its
  connection with the trigonometric solution described by the same quadruple \( Q \) (see \cite{belavin_drinfeld_solutions_of_CYBE_paper}). 
\end{itemize}

In \cite{montaner_stolin_zelmanov} Montaner, Stolin and Zelmanov classified all Lie bialgebra structures on \( \g[z] \) by classifying classical twists within each of four possible Drinfeld double algebras.
A main point in their argument is the aforementioned classification of quasi-trigonometric solutions
\cite{pop_stolin}
or, equivalentely, the classification of classical twists within one of the doubles. 
From this perspective, our work
%Therefore, our work can be also seen as 
is a natural step towards the classification of all Lie bialgebra structures on \( \Loop^{\id} = \g[z,z^{-1}] \) or, more generally, \(\Loop^\sigma\).

\vspace{0.5cm}
\noindent \emph{Acknowledgements}.
The authors are thankful to I. Burban and A. Stolin for introduction to the topic and fruitful discussions
as well as to E. Karolinsky
for useful comments.
R.A. also acknowledges the support from the DFG project Bu-1866/5-1.

\section{Preliminaries}
In this section we give a brief review of the theory of Lie bialgebras and loop algebras as well as set up notation and terminology used throughout the paper.
Most of the presented results
on Lie bialgebras can
be found in 
\cite{chari_pressley, etingof_schiffmann} and \cite{kosmann-schwarzbach}.
A detailed exposition of the theory of loop
algebras can be found in 
\cite{carter, kac_book} and
\cite[Section X.5]{helgason}.

\subsection{Lie bialgebras, Manin triples and twisting}
A \emph{Lie coalgebra} is a pair
\( \seq{L, \delta} \) 
consisting of a vector space \( L \)
over a field \( k \) of characteristic zero
and a linear map
\( \delta \colon \nobreak L \longrightarrow \nobreak L \ot L \),
called \emph{Lie cobracket},
such that for all \( x \in L \)
\begin{equation}\label{eq: Lie coalgebra conditions}
  \delta(x) + \tau\delta(x) = 0 \
  \textnormal{ and } \
  \textnormal{Alt}((\delta \ot 1)\delta(x)) = 0,
\end{equation}
where \( \tau(x_1 \ot x_2) \coloneqq x_2 \ot x_1 \) and 
\( 
  \textnormal{Alt}(x_1 \ot x_2 \ot x_3)
  \coloneqq x_1 \ot x_2 \ot x_3 + 
    x_2 \ot x_3 \ot x_1 + 
    x_3 \ot x_1 \ot x_2
\).
These conditions guarantee that
the restriction of the dual map
\( 
  \dual{\delta} \colon \dual{\seq{L \ot L}}
    \longrightarrow \dual{L}
\)
to \( \dual{L} \ot \dual{L} \)
defines a Lie algebra structure.
A \emph{morphism} between two Lie coalgebras 
\( \seq{L, \delta} \) and
\( \seq{L', \delta'} \) 
is a linear map
\( \phi \colon L \longrightarrow L' \)
such that
\begin{equation}\label{eq: coalgebra morphism}
  (\phi \ot \phi)\delta = \delta' \phi.
\end{equation}
A \emph{Lie bialgebra} is a triple\footnote{For convenience, the notation \( [-,-] \) for the Lie bracket on \( L \) will be omitted from the triple.}  
\( \seq{L, [-,-], \delta} \)
such that
\( \seq{L, [-,-]} \) is a Lie algebra,
\( \seq{L, \delta} \) is a Lie coalgebra
and the following
compatibility condition holds
\begin{equation}\label{eq: 1-cocycle conditoin}
  \delta\seq{[x,y]} = x \cdot \delta(y) -
                      y \cdot \delta(x)
                      \qquad
               \forall x,y \in L,
\end{equation}
where 
\( 
  x \cdot (y_1 \ot y_2) \coloneqq [x, y_1] \ot y_2 +
                           y_1 \ot [x, y_2]
\).
In other words,
\( \delta \) is a 1-cocycle of \( L \) with
values in \( L \ot L \).
A linear map between two Lie bialgebras is
a \emph{Lie bialgebra morphism} if it is
a morphism of both Lie algebra
and Lie coalgebra structures.

Lie bialgebras are closely related to \emph{Manin triples}, i.e.
triples \( \seq{L, L_+, L_-} \), where \( L \) is a Lie
algebra equipped with an invariant non-degenerate symmetric
bilinear form \( B \) and \( L_\pm \) are isotropic subalgebras of \( L \) 
with respect to that form, such that
\( L = L_+ \add L_- \).%
\footnote{We write \( A \add B \) (or
\( A \oplus B \))
meaning the direct
sum of \( A \) and \( B \) as vector spaces (modules),
but not as Lie algebras.
The latter is denoted by \( A \times B \).}
The definition immediately implies that \( L_\pm \)
are Lagrangian subalgebras of \( L \) which are paired non-degenerately by \( B \).
We say that two Manin triples \( \seq{L, L_+, L_-} \)
and \( \seq{L', L'_+, L'_-} \) are \emph{isomorphic} if there is
a Lie algebra isomorphism 
\( \phi \colon L \longrightarrow L' \) such that
\begin{equation}
\phi(L_\pm) = L'_\pm 
\
\text{ and }
\
B(x,y) = B(\phi(x), \phi(y))
\
\text{ for all } x,y \in L.
\end{equation}

Every Lie bialgebra \( \seq{L, \delta} \)
gives rise to the Manin triple
\( \seq{L \add \dual{L}, L, \dual{L}} \)
with the canonical bilinear form \( B \)
given by
\begin{equation}\label{eq: canonical bilinear form on D}
  B(x + f, y + g) \coloneqq f(y) + g(x)
  \qquad
  \forall x, y \in L, 
  \ 
  \forall f, g \in \dual{L},
\end{equation}
and the Lie algebra structure on 
\( L \add \dual{L} \)
defined by
\begin{equation}\label{eq: unique [], classical double}
  [x, f] \coloneqq \ad_x^* f + (f \ot 1)(\delta(x))
  \qquad
  \forall x \in L, 
  \
  \forall f \in \dual{L},
\end{equation} 
where \( \ad^*_x \coloneqq -\dual{\ad_x} \) is the coadjoint action. 
\begin{remark}
The Lie algebra structure
\eqref{eq: unique [], classical double}
is the unique Lie algebra structure on
\( L \add \dual{L} \) making the canonical form
\( B \) invariant and \( L \), \( \dual{L} \) into Lagrangian subalgebras. 
The space \( L \add \dual{L} \) equipped with
this particular Lie algebra structure is called
the \emph{classical double} of \( \seq{L, \delta} \).
\end{remark}
The converse statement is not true, i.e.
not every Manin triple 
\( \seq{L, L_+, L_-} \) induces
a Lie bialgebra structure on \( L_+ \).
However, this is the case when the dual map
\( 
  \dual{[-,-]} \colon \dual{L_-} \longrightarrow 
               \dual{\seq{L_- \ot L_-}}
\)
of the Lie bracket on \( L_- \) restricts to a map
\( \delta \colon L_+ \longrightarrow L_+ \ot L_+ \), where we use the injection \(L_+ \longrightarrow \dual{L}_-\) induced by \(B\).
This condition can be equivalently formulated in the
following way: there is a linear map
\( \delta \colon L_+ \to L_+ \ot L_+ \) such that
\begin{equation}\label{eq: M. triple -> Bialgebra if exists delta}
  B(\delta(x), y \ot z) = B(x, [y,z])
  \qquad 
  \forall x \in L_+, \ \forall y,z \in L_-.
\end{equation}
When this condition is satisfied, we say that
the Manin triple
\( \seq{L, L_+, L_-} \) 
\emph{defines} the Lie bialgebra
\( \seq{L_+, \delta} \).

\begin{remark}\label{rem: isomorphic manin tiples define isomorphic bialgebras}
Let \( \phi \) be an isomorphism between
two Manin triples
\( M = \seq{L, L_+, L_-}  \) and 
\( M' = \seq{L', L'_+, L'_-}  \).
If \( M \) defines a Lie bialgebra structure
\( (L_+, \delta) \), then \( M' \) also
defines a Lie bialgebra
structure \( (L'_+, \delta') \) and \(\phi|_{L_+} \colon (L_+, \delta) \longrightarrow (L'_+, \delta')\) is a Lie bialgebra isomorphism.
\end{remark}

\begin{remark}
  Generally, there may exist many non-isomorphic Manin triples
  defining the same Lie bialgebra structure.
  However, in the finite-dimensional case the condition
  \eqref{eq: M. triple -> Bialgebra if exists delta}
  holds automatically and the correspondence
  between Manin triples and Lie bialgebras described above
  is one-to-one.
\end{remark}

Having a Lie bialgebra structure, we can produce a new bialgebra structure by means of a procedure called \emph{twisting}.
Let \( \seq{L, \delta} \) be a Lie bialgebra
and \( t \in L \ot L \) be a skew-symmetric tensor
satisfying the identity
\begin{equation}
  \textnormal{CYB}(t)
  = 
  \textnormal{Alt}\seq{(\delta \ot 1)t},
\end{equation}
where 
\(
  \textnormal{CYB}(t)
  \coloneqq
  [t^{12}, t^{13}] +
  [t^{12}, t^{23}] +
  [t^{13}, t^{23}]
\).
Then the linear map 
\( \delta_t \coloneqq \delta + dt \), where
\( dt(x) \coloneqq x\cdot t \) for all
\( x \in L \),
\iffalse
\begin{equation}\label{eq: twisted delta}
  \delta_t \coloneqq \delta + dt \ \text{ where } \ dt(x) \coloneqq x\cdot t  \ \text{ for all } x \in L,
\end{equation}
\fi
defines a new Lie bialgebra structure on \( L \).
The skew-symmetric tensor \( t \) is called
a \emph{classical twist} of \( \delta \).

It was implicitly shown in
\cite{KPSST, pop_stolin, stolin_sln, stolin_maximal_orders}
that the problem of classification of classical twists
of some particular Lie bialgebra structures
can be reduced to the classification
of Lagrangian Lie subalgebras. 
In the following theorem we summarize and generalize
these ideas.

\begin{theorem}\label{thm: W <-> t <-> T correspondence}
  Let \( \seq{L_+, \delta} \) be a Lie bialgebra defined
  by the Manin triple \( \seq{L, L_+, L_-} \).
  Then there are the following one-to-one correspondences:
  
 \begin{tikzpicture}[every node/.style={inner sep=.15cm,outer sep=0cm}]
  \node[] (A) {\mlnode{
    Classical twists of \( \delta \), i.e. \\
    skew-symmetric tensors \( t \in L_+ \ot L_+ \) \\
    satisfying \( 
    \textnormal{CYB}(t)
    =
    \textnormal{Alt}\seq{(\delta \ot 1)t} \)
  }};
  \node[below left=1cm and -3cm of A] (B) {\mlnode{
    Lagrangian Lie subalgebras \( L_t \subseteq L \) \\
    complementary to \( L_+ \) and \\
    commensurable with \( L_- \), i.e. \\
    \( \dim(L_t + L_-) / (L_t \cap L_-) < \infty \)
  }};
    \node[below right=1cm and -3cm of A] (C) {\mlnode{
    Linear maps 
    \( T \colon L_- \longrightarrow L_+ \) such that \\
    \(\dim(\im(T)) < \infty \) and for all \( w_1, w_2, w_3 \in L_- \) \\
    holds \( B(Tw_1,w_2)+ B(w_1,Tw_2) = 0\) and \ \  \\
    \( B\seq{[Tw_1 - w_1, Tw_2 - w_2], Tw_3 - w_3} = 0 \)
  }};
%  \node[below right=1cm and -3cm of A] (C) {\mlnode{
%    Linear maps 
%    \( T \in \textnormal{Hom}_{\textnormal{fin}}(L_-, L_+) \)\\
%    such that \( T + T^* = 0 \) and \\
%    \( B\seq{[Tw_1 + w_1, Tw_2 + w_2], Tw_3 + w_3} = 0 \) \\
%    for all \( w_1, w_2, w_3 \in L_- \)
%  }};
  
\draw[<->] (A) edge (B) (B) edge (C) (C) edge (A);
\end{tikzpicture} 
\end{theorem}

\begin{proof}
  Let \( t = x_i \ot y^i\in L_+ \ot L_+ \)
  be a classical twist%
  \footnote{We use the Einstein summation convention:
  \( x_i \ot y^i = \sum_{i} x_i \ot y^i \).} 
  of \( \delta \).
  Define the linear map \( T \colon L_- \longrightarrow L_+ \) and the subspace
  \( L_t \subseteq L \) by
  \begin{equation}\label{eq: t -> T}
    T \coloneqq B(y^i, -)x_i \ \
    \textnormal{and} \ \
    L_t \coloneqq \set{Tw - w \mid w \in L_-}.
  \end{equation}
  We now show that they meet the requirements of the theorem.
  The conditions \( \dim(\im(T)) < \infty \) and
  \( L_+ \dotplus L_t = L \) hold by definition.
  For all \( w_1, w_2 \in L_- \) we have
  \begin{equation}\label{eq: skew sym t and T <-> Lagrangian}
  \begin{split}
    B\seq{Tw_1 - w_1, Tw_2 - w_2}
    &= -  B(Tw_1,w_2)- B(w_1,Tw_2) \\
    &= -B(y^i, w_1)B(x_i, w_2) - B(y^i, w_2)B(x_i, w_1).
  \end{split}
  \end{equation}
  Therefore, the skew-symmetry of \( t \) is
  equivalent to the skew-symmetry of \( T \) and to
  \( L_t \) being a Lagrangian subspace.
  To prove the commensurability of \( L_t \) and \( L_- \)
  we note that \( \ker(T) = L_t \cap L_- \) and hence
  \begin{equation}\label{eq: commensurability Lt and L-}
    \dim\seq{L_- / \seq{L_t \cap L_-}} = 
    \dim\seq{\textnormal{im}(T)}.
  \end{equation}
  This shows that \( L_t \cap L_- \) has finite codimension
  inside \( L_- \). The commensurability now follows from the
  fact that \( L_- \) has codimension 
  at most \( \dim(\im(T)) \) inside \( L_t + L_- \).
  Finally, the last condition follows from the identity
  \begin{equation}\label{eq: CYBE - Alt = B([],)}
    B\seq{%
    w_1 \ot w_2 \ot w_3, 
    \textnormal{CYB}(t) - 
    \textnormal{Alt}\seq{\seq{\delta \ot 1}t}
    }
    = - B\seq{[Tw_1 - w_1, Tw_2 - w_2], Tw_3 - w_3},
\end{equation}
where \( w_1, w_2, w_3 \in L_- \).
This identity
is obtained by repeating the argument in the proof of
\cite[Theorem 7]{karolinsky_Stolin} within
our framework.
\iffalse
Let us compute the left-hand side of
\eqref{eq: CYBE - Alt = B([],)}.
We have
\begin{equation}
\begin{split}
  B(
  w_1 \ot w_2 \ot w_3,
  [t^{12}, t^{13}]
  ) 
  &=
  B(w_1, [x_i, x_j])
  B(w_2, y^i)
  B(w_3, y^j) \\
  &=
  B(w_1, [B(w_2, y^i) x_i, B(w_3, y^j) x_j]) \\
  &=
  B(w_1, [Tw_2, Tw_3]),
\end{split}
\end{equation}
and by symmetry we obtain
\begin{equation}
\begin{split}
  B(
  w_1 \ot w_2 \ot w_3,
  [t^{12}, t^{23}]
  ) 
  &=
  B(w_2, [Tw_3, Tw_1]), \\
  B(
  w_1 \ot w_2 \ot w_3,
  [t^{13}, t^{23}]
  )
  &=
  B(w_3, [Tw_1, Tw_2]).
\end{split}
\end{equation}
Similarly we find that
\begin{equation}
\begin{split}
  B(
  w_1 \ot w_2 \ot w_3,
  (\delta \ot 1)t
  ) 
  &=
  B(Tw_3, [w_1, w_2]), \\
  B(
  w_2 \ot w_3 \ot w_1,
  (\delta \ot 1)t
  ) 
  &=
  B(Tw_1, [w_2, w_3]), \\
  B(
  w_3 \ot w_1 \ot w_2,
  (\delta \ot 1)t
  ) 
  &=
  B(Tw_2, [w_3, w_1]). 
\end{split}
\end{equation} 
Expanding the expression on the right hand-side
of \eqref{eq: CYBE - Alt = B([],)}
we get the same summands, proving the desired identity.
\fi

Conversely, given a Lagrangian Lie subalgebra 
\( L' \subseteq L \),
satisfying the conditions of the theorem,
we define the linear map \( T \colon L_- \longrightarrow L_+ \) 
in the following way:
any \( w \in L_- \) can be uniquely written as
\( w_+ + \, w' \), for some \( w_+ \in L_+ \) and
\( w' \in L' \); We let \( T(w) \coloneqq w_+ \).
Then \( L' = \set{Tw - w \mid w \in L_-} \) and
the commensurability of \( L' \) and \( L_- \) implies
that the rank of \( T \) is finite.
The other two conditions on \( T \) hold 
because of the relations
\eqref{eq: skew sym t and T <-> Lagrangian} and
the Lagrangian property of \( L' \).
To construct the classical twist \( t \in L_+ \ot L_+ \)
we note that \( B \) gives a non-degenerate pairing between the finite-dimensional spaces
\( L_- / \ker(T) \) and \( \im(T) \).
Let \( \set{Tw_i}_{i=1}^n \) be a basis for \( \im(T) \) and
\( \set{v^i + \ker(T)}_{i=1}^n \) be its dual basis for
\( L_- / \ker(T) \). Then
\begin{equation}
  B(w_k, -Tv^i)Tw_i = B(Tw_k, v^i)Tw_i = Tw_k,
\end{equation}
for all \( k \in \{1,\ldots, n\} \). 
Since \(T\) is completely determined by its action on \( \{w_i\}^n_{i = 1}\), we have
the equality \(T = -B(Tv^i,-)Tw_i\).
We define \( t \coloneqq -Tw_i \ot Tv^i \). 
The identities 
\eqref{eq: CYBE - Alt = B([],)} and
\eqref{eq: skew sym t and T <-> Lagrangian}
guarantee that \( t \) meets the desired requirements and \(L' = L_t\).
\end{proof}

\begin{remark}\label{rem: manin defining twisted bialgebra}
 It follows that 
 if \( \seq{L_+, \delta} \) is a Lie bialgebra defined by
 the Manin triple \( \seq{L, L_+, L_-} \) and
 \( t \) is a classical twist of \( \delta \), then
 the twisted Lie bialgebra \( \seq{L_+, \delta + dt} \)
 is defined by the Manin triple
 \( \seq{L, L_+, L_t} \). Equivalently,
 \begin{equation}
   B\seq{\delta(x) + x \cdot t, \seq{Tw_1 - w_1} \ot \seq{Tw_2 - w_2}}
   =
   B\seq{x, [Tw_1 - w_1, Tw_2 - w_2]},
 \end{equation}
 for all \( x \in L_+ \) and \( w_1, w_2 \in L_- \). 
\end{remark}

\subsection{Loop algebras}\label{sec: Loop algebras}
Let \( \g \) be a fixed
finite-dimensional simple Lie algebra 
over \( \CC \) and
\( \sigma \)
be an automorphism of \( \g \)
of finite order \( |\sigma| \in \ZZ_+ \).
The eigenvalues of \( \sigma \) are
\( \varepsilon^k_{\sigma} \coloneqq e^{2 \pi i k /|\sigma|} \),
\( k \in \ZZ \),
and we have the following
\( \ZZ / |\sigma| \ZZ \)-gradation of \( \g \)
\begin{equation}
  \g = \bigoplus_{k=0}^{|\sigma|-1} \g^\sigma_{k},
\end{equation}
where \( \g_{k}^\sigma \)
is the eigenspace of \( \sigma \)
corresponding to the eigenvalue
\( \varepsilon^k_{\sigma} \).
The tensor product
\begin{equation}\label{eq: Loop regular funcs}
  \Loop \coloneqq \g \ot \CC[z,z^{-1}] =
  \bigoplus_{k \in \ZZ} z^k \g
  = \set{f \colon \CC^* \longrightarrow \g \mid f \text{ is regular} },
\end{equation}
equipped with the bracket described by
\( [z^i x, z^j y] \coloneqq z^{i+j}[x,y] \), 
for all \( x,y \in \g \) and \( i,j \in \ZZ \), is a
\( \ZZ \)-graded Lie algebra over \( \CC \).
\emph{The loop algebra} \( \Loop^\sigma \)
over \( \g \) is the \( \ZZ \)-graded
Lie subalgebra of \( \Loop \)
defined by
\begin{equation}\label{eq: Loop sigma regular funcs}
\Loop^\sigma \coloneqq \bigoplus_{k \in \ZZ} z^k \g^\sigma_k
=
\set{f \in \Loop \mid \sigma(f(z)) = f(\varepsilon_\sigma z)},
\end{equation}
where \( \g^\sigma_{k + \ell |\sigma|} = \g^\sigma_{k} \) for all \( \ell \in \ZZ \).
It possesses
an invariant non-degenerate symmetric bilinear form 
\( B \), which is given by
\begin{equation}\label{eq: form on L sigma}
  B(f,g)
  \coloneqq
  \textnormal{res}_{z = 0} \left[ \frac{1}{z} \K(f(z), g(z)) 
  \right]
  \qquad
  \forall f,g \in \Loop,
\end{equation}
where \( \K \) stands for the Killing form on \( \g \).

\begin{remark}
We can extend \( \sigma \) to an automorphism on \( \Loop \)
by
\( \sigma(z^k x) \coloneqq (z /\varepsilon_{\sigma})^k \sigma(x) \). Then 
\( \Loop^\sigma \) can be viewed 
as the Lie subalgebra of \( \Loop \) consisting of
fixed points
of the extended action of \( \sigma \) on \( \Loop \).
In particular, we have the identity
\( \Loop = \Loop^{\textnormal{id}} \).
This motivates our choice of notation.
\end{remark}

\subsubsection{Structure theory (outer automorphism case)}
The classification of all finite order automorphisms
of \( \g \), explained in
\cite{kac_finite_order_automorphisms, helgason, kac_book},
gives the following relation at the level of loop algebras:
for any finite order automorphism
\( \sigma \)
there is an automorphism \( \nu \) of \( \g \),
induced by an automorphism of the corresponding
Dynkin diagram, such that
\( \Loop^\sigma \cong \Loop^\nu \).
Therefore, we first describe
the structure of \( \Loop^\nu \) and then
explain how regrading of \( \Loop^\nu \) carries
over the structure theory to \(\Loop^\sigma\).

Let \( \g = \n'_- \add \h' \add \n'_+ \) 
be a triangular decomposition of \( \g \)
and \( \widetilde{\nu} \) be an automorphism of the
corresponding Dynkin diagram.
The induced outer automorphism \( \nu \)
of \( \g \)
is described explicitly by
\begin{equation}
  \nu(x^{\pm}_i) = x^{\pm}_{\widetilde{\nu}(i)}, \qquad
  \nu(h^{~}_i) = h^{~}_{\widetilde{\nu}(i)},
\end{equation}
where \( \{ x^-_i, h^{~}_i, x^+_i \} \)
is a fixed set of standard
Chevalley generators for \( \g \).
The order of such an automorphism is necessarily 1, 2 or 3.
The subalgebra \( \g^\nu_0 \) turns
out to be simple with the following triangular decomposition
\begin{equation}\label{eq: triangular decomposition for g0}
  \g^\nu_0 = 
  \underbrace{(\g^\nu_0 \cap \n'_- )}_{\eqqcolon \, \n_-}
  \add
  \underbrace{(\g^\nu_0 \cap \h')}_{\eqqcolon \, \h}
  \add
  \underbrace{(\g^\nu_0 \cap \n'_+)}_{\eqqcolon \, \n_+}.
\end{equation}
Moreover, when \( |\nu|  = 2 \text{ or } 3 \) the subspace
\( \g^\nu_1 \) is an irreducible
\( \g^\nu_0 \)-module. In the case \( |\nu| = 3 \) it is 
isomorphic (as a module) to \( \g^\nu_2 = \g^\nu_{-1} \).

\begin{remark}\label{rem: L^nu is determined by ord(nu)}
For any automorphism \( \rho \) of \( \g \) we have
a natural \(\ZZ\)-graded Lie algebra isomorphism
\( \Loop^\sigma \cong \Loop^{\rho \sigma \rho^{-1}} \)
given by \( z^k x \longmapsto z^k \rho(x) \).
Since the automorphism \( \nu \) is defined by its order
up to conjugation, this result implies that \( \Loop^\nu \)
is also determined by the order of the
automorphism \( \nu \).
\end{remark}

A pair \( \seq{\alpha, k} \),
where \( \alpha \in \dual{\h}  \)
and \( k \in \ZZ \), is called a \emph{root}
if the joint eigenspace
\begin{equation}\label{eq: eqigenspace g nu alpha}
  \g^\nu_{\seq{\alpha, k}}
  =
  \set{x \in \g^\nu_k \mid [h, x] = \alpha(h) x \ \forall h \in \h}
\end{equation}
is non-zero.
Let \( \Phi \) be the set of all roots and
\( \Phi_k \) be the set of roots of the form
\( \seq{\alpha, k} \).
The triangular decomposition
\eqref{eq: triangular decomposition for g0} of
\( \g^\nu_0 \)
gives rise to the polarization
\( \Phi_0 = \Phi_0^- \cup \set{(0,0)} \cup \Phi_0^+  \).
For convenience we introduce two more subsets of roots:
\begin{equation}
\begin{split}
  \Phi^+ &\coloneqq  \Phi_0^+ \cup \set{(
  \alpha, k) \in \Phi \mid k > 0}, \\
  \Phi^- &\coloneqq \Phi_0^- \cup \set{(
  \alpha, k) \in \Phi \mid k < 0}.
\end{split}
\end{equation}
The elements of \( \Phi^+ \) and \( \Phi^- \)
are called \emph{positive} and \emph{negative roots} respectively.
It is clear that
\( \Phi = \Phi^- \cup \{(0,0)\} \cup  \Phi^+ \) and
\( -\Phi^+ = \Phi^- \).
%If \( \seq{\alpha, k} \) is a root, we write
%\( \Loop_{\seq{\alpha, k}} = z^k \g_{\seq{\alpha, k}} \)
%for the corresponding root space.
%Therefore, we have the following root space decomposition
Denoting \( z^k \g^\nu_{\seq{\alpha, k}} \) by
\( \Loop^\nu_{\seq{\alpha, k}} \) we get the root space
decomposition
\begin{equation}\label{eq: L nu decomposition}
\Loop^\nu = \bigoplus_{\seq{\alpha,k} \in \Phi}\Loop^\nu_{\seq{\alpha, k}},
\end{equation}
where 
\( \dim (\Loop^\nu_{\seq{\alpha, k}}) = 1 \)
if \( \alpha \neq 0 \)
and 
\( \Loop^\nu_{(0,0)} = \h \). 
The form \( B \) pairs the spaces
\( \Loop^\nu_{\seq{\alpha, k_1}} \) and
\( \Loop^\nu_{\seq{\beta, k_2}} \)
non-degenerately
if 
\(  (\alpha, k_1) + (\beta, k_2) = (0,0) \);
otherwise
\( B(\Loop^\nu_{\seq{\alpha, k_1}}, \Loop^\nu_{\seq{\beta, k_2}}) = 0 \). 
Defining 
\begin{equation}
\mathfrak{N}_\pm \coloneqq \bigoplus_{\seq{\alpha,k} \in \Phi^\pm} \Loop^\nu_{\seq{\alpha, k}},
\end{equation} 
we obtain analogues of a triangular decomposition
and Borel subalgebras for \( \Loop^\nu \),
namely
\begin{equation}
\Loop^\nu =  \mathfrak{N}_- \add \h \add \mathfrak{N}_+
\ \text{ and } \
\B_\pm \coloneqq \h \add \mathfrak{N}_\pm.
\end{equation}

Let \( \{ \alpha_1, \ldots, \alpha_n \} \) be a set of simple
roots of \( \g^\nu_0 \) with respect to
\eqref{eq: triangular decomposition for g0}
and \( \alpha_0 \) be the corresponding minimal root.
For any root \( \alpha \) we
write \( \dual{\alpha} \)
for the unique element in \( \h \)
such that
\( B(\dual{\alpha}, -) = \alpha(-) \).
The set
\begin{equation}\label{eq: simple root system nu}
  \Pi \coloneqq \{ \underbrace{\seq{\alpha_0, 1}}_{\eqqcolon \, \widetilde{\alpha}_0}, \underbrace{\seq{\alpha_1, 0}}_{\eqqcolon \, \widetilde{\alpha}_1}, \ldots,
  \underbrace{\seq{\alpha_n, 0}}_{\eqqcolon \, \widetilde{\alpha}_n} \}.
\end{equation}
is called
\emph{the simple root system}
of \( \Loop^\nu \). It satisfies
the following properties:
\begin{enumerate}
\item Any \( (\alpha,k) \in \Phi \) can be uniquely written
in the form
\( (\alpha, k) = \sum_{i=0}^n c_i \widetilde{\alpha}_i \), where \( c_i \in \ZZ \).
If the root \( (\alpha, k) \) is positive (negative), then
the coefficients \( c_i \) in its decomposition
are all non-negative (non-positive);
\item The matrix
\( A \coloneqq (a_{ij}) \), where
\begin{equation}\label{eq: Cartan matrix of Lnu}
  a_{ij} \coloneqq 2\frac{B(\dual{\alpha_i}, \dual{\alpha_j})}{B(\dual{\alpha_j} , \dual{\alpha_j})} \in \ZZ
  \qquad
  i,j \in \{ 0, 1, \ldots, n \},
\end{equation}
is a generalized Cartan matrix of affine type.
We call it \emph{the affine matrix associated to} \( \Loop^\nu \).
The Dynkin diagram corresponding to \( A \) is called \emph{the Dynkin diagram of} \( \Loop^\nu \).
\end{enumerate}
Let \( \Lambda_0 \coloneqq \{X^-_i, H^{~}_i, X^+_i \}_{i = 1}^n \) be the set of
standard Chevalley generators for \( \g^\nu_0 \) with respect
to the choice of simple roots we made earlier.
Take two elements
\( X_0^{\pm} \in \Loop_{(\pm \alpha_0, \pm 1)} \) such that
\begin{equation}
\left[ X_0^+, X_0^- \right]
=
\frac{\dual{\alpha_0}}{B(\dual{\alpha_0}, \dual{\alpha_0})}
\eqqcolon H^{~}_0.
\end{equation}
By
\cite[Lemma X.5.8]{helgason} the set
\( \Lambda \coloneqq \Lambda_0 \cup \{ X_0^-, H^{~}_0, X_0^+ \} \)
generates the whole Lie algebra \( \Loop^\nu \).
For any \( S \subsetneq \Pi \) we denote by
\( \mathfrak{S}^S \) the semi-simple subalgebra of
\( \Loop^\nu \) generated by 
\( \{X^-_i, H^{~}_i, X^+_i \}_{\widetilde{\alpha}_i \in S} \)
with the induced triangular decomposition
\( \mathfrak{S}^S = \N^S_- \add \h^S \add \N^S_+  \).
The subalgebras 
\( \mathfrak{p}^S_\pm \coloneqq \B_\pm \add \N^S_\mp \ \) are
the analogues for the parabolic subalgebras
in the theory of semi-simple Lie algebras.

\subsubsection{Classification of finite order automorphisms and regrading}
We now explain the regrading procedure that makes it possible
to transfer all the preceding results of this section to
\( \Loop^\sigma \)
for an arbitrary finite order automorphism \( \sigma \).
Let \( s = (s_0, s_1, \ldots, s_n) \) be a sequence
of non-negative integers with at least one non-zero element.
Using the properties of the simple
root system \eqref{eq: simple root system nu} we can write
\begin{equation}\label{eq:defa_i}
(0,|\nu|)
= |\nu|
\sum_{i=0}^{n} a_i \widetilde{\alpha}_i
%|\nu| a_0 (\alpha_0, 1) + |\nu|\sum_{i=1}^n a_i (\alpha_i, 0)
\end{equation}
for some unique positive integers \( a_i \). 
We define a positive integer \( m \coloneqq |\nu| \sum_{i = 0}^n a_i s_i \).
The following results were proven in
\cite[Theorem X.5.15]{helgason}:
\begin{enumerate}
\item The set \( \{ X_j^+(1) \}_{j=0}^n \) generates the
Lie algebra \( \g \) and the relations
\begin{equation}
\sigma_{(s; |\nu|)}(X_j^+(1)) \coloneqq e^{2  \pi i s_j / m} X_j^+(1)
\qquad 0 \leq j \leq n
\end{equation}
define a unique automorphism
\( \sigma_{(s; |\nu|)} \) of \( \g \)
of order \( m \) such that
\( \Loop^{\nu} \cong \Loop^{\sigma_{(s; |\nu|)}} \).
In particular, \( \nu = \sigma_{((1,0,\ldots, 0); |\nu|)} \);
\item Up to conjugation any finite order automorphism \( \sigma \)
of \( \g \)
arise in this way.
\end{enumerate}
It follows immediately that
for any finite order automorphism \( \sigma \) of \( \g \)
there is an automorphism \( \sigma_{(s; |\nu|)} \) and
an outer automorphism \( \nu \) of \( \g \)
such that
\begin{equation}\label{eq: L sigma = L sigma(s,k) = L nu}
\Loop^\nu \xrightarrow[G^s]{\sim} \Loop^{\sigma_{(s; |\nu|)}} \xrightarrow[\phantom{G^s}]{\sim} \Loop^\sigma,
\end{equation}
where the second isomorphism, given by conjugation, is described in
Remark
\ref{rem: L^nu is determined by ord(nu)}.
The automorphism \( \sigma_{(s; |\nu|)} \) is called
\emph{the automorphism of type \( (s; |\nu|) \)}.
Note that the conjugacy class of the coset
\( \sigma_{(s; |\nu|)} \textnormal{Inn}_{\CC - \textnormal{LieAlg}}(\g) \) is
represented by \( \nu \).

Now we describe the first isomorphism in the chain
\eqref{eq: L sigma = L sigma(s,k) = L nu}.
Define \emph{the \( s \)-height
\( \hgt_s (\alpha, k) \)}
of a root \( (\alpha, k) \in \Phi \)
in the following way: decompose \( (\alpha, k) \)
with respect to the simple root system
\( \Pi \), i.e. 
\( (\alpha,k)
= 
\sum_{i=0}^n c_i \widetilde{\alpha}_i \) and set
\begin{equation}
\hgt_s (\alpha, k) \coloneqq \sum_{i=0}^n c_i s_i. \end{equation}
We introduce a new \( \ZZ \)-grading on \( \Loop^\nu \),
called
\emph{\( \ZZ \)-grading of type \( s \)},
by declaring \( \deg(f) = 0 \)
for \( f \in \h \) and
\( \deg(f) = \hgt_s(\alpha, k) \) for
\( f \in \Loop^\nu_{(\alpha, k)} \).
The isomorphism
\( G^s \colon \Loop^\nu \longrightarrow \Loop^{\sigma_{(s; |\nu|)}} \),
called \emph{regrading},
is given by
\begin{equation}\label{eq: regrading to s}
  G^s(z^k x) \coloneqq z^{\hgt_s(\alpha, k)} x \qquad
  \forall z^k x \in \Loop^\nu_{(\alpha, k)}.
\end{equation}
If \( \Loop^\nu \) is equipped with the grading of
type \( s \) and
\( \Loop^{\sigma_{(s; |\nu|)}} \)
is equipped with the natural grading given by
the powers of \( z \), then \( G^s \) is
a graded isomorphism.
We write \( G^{s'}_s \) 
for the resulting regrading
\( G^{s'} \circ (G^s)^{-1} \colon \Loop^{\sigma_{(s; |\nu|)}}
\longrightarrow \Loop^{\sigma_{(s'; |\nu|)}} \).
\begin{remark}
The grading given by
\( s = \mathbf{1} = (1, 1, \ldots, 1) \)
is called
\emph{the principle grading}
and the corresponding
automorphism \( \sigma_{(\mathbf{1}; |\nu|)} \) is
\emph{the Coxeter automorphism}
of the pair \( (\g, \nu) \).
\end{remark}

\subsubsection{Structure theory (general case)}
We finish the discussion of loop algebras by pushing the
structure theory for 
\( \Loop^\nu \) to \( \Loop^\sigma \)
through the chain of isomorphisms
\eqref{eq: L sigma = L sigma(s,k) = L nu}.
We do it gradually, starting with the case
\( \sigma = \sigma_{(s; |\nu|)} \), \( s = (s_0, s_1, \ldots, s_n)\).
Let \( \Phi \) and \( \Pi \), as before, be the set of all roots and the simple root system of \( \Loop^\nu \).
From the definition of regrading
it is clear that \( G^s(\h) = \h \).
This allows us to define the joint
eigenspaces
\( \g^{\sigma}_{(\alpha, \ell)} \),
\( \alpha \in \dual{\h}\), \( \ell \in \ZZ \),
using the exact same formula
\eqref{eq: eqigenspace g nu alpha}
and call \( (\alpha, \ell) \) a root
of \( \Loop^{\sigma} \)
if \( \g^{\sigma}_{(\alpha, \ell)} \neq 0 \).
Using regrading we can describe the root spaces
\( \Loop^\sigma_{(\alpha,\ell)} \coloneqq z^{\ell}\g^{\sigma}_{(\alpha, \ell)}\)
of \( \Loop^\sigma \) in terms of the root spaces
of \( \Loop^\nu \), namely
\begin{equation}\label{eq: regraded root space}
    G^s \seq{\Loop^\nu_{(\alpha, k)}}
    =
    \Loop^\sigma_{(\alpha, \hgt_s(\alpha,k))}
    \qquad
    \forall (\alpha, k) \in \Phi.
\end{equation}
This gives a
bijection between roots of
\( \Loop^\nu \) and \( \Loop^\sigma \).
More precisely, let \( \Phi_\sigma \) be the set of all roots of \( \Loop^\sigma \), then
\begin{equation}
  \Phi_\sigma = \set{(\alpha, \hgt_s(\alpha, k)) \mid (\alpha, k) \in \Phi}.
\end{equation}
The subset \( \Pi^\sigma \coloneqq \{ (\alpha_0, s_0), (\alpha_1, s_1), \ldots, (\alpha_n, s_n) \} \subseteq \Phi_\sigma \) is said to be the simple root system of \( \Loop^\sigma \).
We again adopt the notation \( \widetilde{\alpha_i} \) for the simple root \( (\alpha_i, s_i) \).
By definition of
\( \hgt_s \) the root spaces
\( \Loop^\sigma_{\seq{\alpha, \ell_1}} \) and
\( \Loop^\sigma_{\seq{\beta, \ell_2}} \) 
are paired by the form
\( B \) non-degenratly if 
\(  (\alpha, \ell_1) + (\beta, \ell_2) = (0,0) \); otherwise
\( B(\Loop^\sigma_{\seq{\alpha, \ell_1}}, \Loop^\sigma_{\seq{\beta, \ell_2}}) = 0 \).
It is evident from \eqref{eq: regraded root space} that the subspaces
\( \mathfrak{N}_\pm, \B_\pm \subseteq \Loop^\nu \) are
fixed under regrading and thus we can
unambiguously use the same notations for them
considered as subspaces of \( \Loop^\sigma \).
Applying regrading to the set of generators \( \Lambda = \{ X_i^-, H^{~}_i, X_i^+ \}_{i=0}^n \) of \( \Loop^\nu \) we obtain the set
\begin{equation}\label{eq: generators of L sigma_sk}
  \Lambda^\sigma \coloneqq
\set{z^{-s_i}X^-_i(1), H^{~}_i, z^{s_i}X^+_i(1)}_{i=0}^{n}
\end{equation}
of generators of \( \Loop^\sigma \).
When \( S \subsetneq \Pi^\sigma \) we use the same notation
\( \mathfrak{S}^S \) to denote the semi-simple subalgebra of \( \Loop^\sigma \) generated by
\( \{ z^{-s_i}X^-_i(1), H^{~}_i, z^{s_i}X^+_i(1) \}_{\widetilde{\alpha}_i \in S}\)
with the induced triangular decomposition
\( \mathfrak{S}^S = \N^S_- \add \h^S \add \N^S_+  \).
The corresponding parabolic subalgebras of \( \Loop^\sigma \) are defined using the same formulas, namely
\( \p^S_\pm \coloneqq \B_\pm \add \mathfrak{N}^S_\mp \).
We also define
\begin{equation}
    \n^\sigma_\pm :=
    \bigoplus_{(\alpha, 0) \in \Phi_\sigma^\pm}
    \Loop^\sigma_{(\alpha, 0)}
    =
    \g^\sigma_0 \cap \mathfrak{N}_\pm.
\end{equation}
This gives the triangular decomposition
\( \g^\sigma_0 = \n^\sigma_- \add \h \add \n^\sigma_+ \).

Finally, we consider the case
\( \sigma = \rho \sigma_{(s; |\nu|)} \rho^{-1}  \) for some 
\( \rho \in \textnormal{Aut}_{\CC - \textnormal{LieAlg}}(\g) \).
We denote the natural isomorphism
\begin{equation}
\Loop^{\sigma_{(s; |\nu|)}} \longrightarrow
\Loop^\sigma,
\qquad
z^k x \longmapsto z^k \rho(x)
\end{equation}
with the same letter \( \rho \).
The roots of \( \Loop^\sigma \)
with respect to the action of the
Cartan subalgebra
\( \rho(\h) \) are of the form
\((\alpha \rho^{-1}, \ell)\),
where \( (\alpha, \ell) \) is
a root of
\( \Loop^{\sigma_{(s; |\nu|)}} \),
and the root spaces
are described by
\begin{equation}\label{eq: conjugated root space}
    \Loop^\sigma_{(\alpha \rho^{-1}, \ell)}
    =
    \rho\seq{\Loop^{\sigma_{(s; |\nu|)}}_{(\alpha, \ell)}}.
\end{equation}
%The set
%\(  \Lambda^\sigma \coloneqq
%\set{\rho\seq{z^{-s_i}X^-_i(1)}, %\rho\seq{H^{~}_i}, %\rho\seq{z^{s_i}X^+_i(1)}}_{i=0}%^{n} \) generates \( %\Loop^\sigma \).
The set of all roots is again denoted by \( \Phi_\sigma \),
and its subset
\begin{equation}
\Pi^\sigma \coloneqq \{ (\alpha_0 \rho^{-1}, s_0), (\alpha_1 \rho^{-1}, s_1), \ldots, (\alpha_n \rho^{-1}, s_n)\}
\end{equation}
is called the simple root system of \( \Loop^\sigma \).
Applying \( \rho \) to the generators 
\eqref{eq: generators of L sigma_sk} of \( \Loop^{\sigma_{(s; |\nu|)}} \) we get the set
\begin{equation}\label{eq: generators of L sigma}
    \Lambda^\sigma \coloneqq 
    \set{z^{-s_i}\rho(X_i^-(1)), \rho(H_i), z^{s_i}\rho(X_i^+(1))}_{i=0}^n
\end{equation}
of generators of \( \Loop^\sigma \).
Later, when there is no ambiguity,
the same notations \( X_i^\pm \) and \( H_i \)
are used to denote the elements of generating sets 
\eqref{eq: generators of L sigma} and 
\eqref{eq: generators of L sigma_sk}.
Combining \eqref{eq: conjugated root space} with \eqref{eq: regraded root space} we define
\begin{equation}
\begin{split}
    \n^\sigma_\pm
    &\coloneqq
    \rho(\n^{\sigma_{(s; |\nu|)}}_\pm) =
    \bigoplus_{(\alpha, 0) \in \Phi^\pm}
    \Loop^\sigma_{(\alpha \rho^{-1}, \hgt_s(\alpha, k))}, \\
    \mathfrak{N}^{\sigma}_\pm
    &\coloneqq
    \rho(\mathfrak{N}_\pm)
    =
    \bigoplus_{(\alpha, k) \in \Phi^\pm}
    \Loop^\sigma_{(\alpha \rho^{-1}, \hgt_s(\alpha, k))}, \\
    \B^\sigma_\pm
    &\coloneqq
    \rho(\B_\pm)
    =
    \mathfrak{N}^{\sigma}_\pm \add \rho(\h).
\end{split}
\end{equation}
where \( \Phi \), as before,
is the set of all roots of \( \Loop^\nu \).
Note that this notation
is in consistence with
the one defined earlier.

\begin{remark}
  Let \( \sigma = \rho \sigma_{(s; |\nu|)} \rho^{-1} \) and \( A \) be the
  affine matrix associated to \( \Loop^\sigma \),
  defined in a way similar to \eqref{eq: Cartan matrix of Lnu}. Then \(A\) coincides with the affine Cartan matrix of \(\Loop^\nu\) and
  so does the Dynkin diagram of \( \Loop^\sigma \).
\end{remark}

\subsubsection{Connection to Kac-Moody algebras}\label{rem: g(A) as L + Cc + Cd}
As the structure theory developed in the preceding subsections suggests, the notion of a loop algebra is closely related to the notion
of an affine Kac-Moody algebra. More precisely, let
\( A \) be an affine matrix of type
\( X_N^{(m)} \), \( \g \) be the simple finite-dimensional
Lie algebra of type \( X_N \) and \( \nu \) be an
automorphism of \( \g \) induced by an automorphism
of the corresponding Dynkin diagram with
\( |\nu| = m \).
Then \(A\) is the Cartan matrix of \(\Loop^\nu\) and the affine Kac-Moody algebra
\( \mathfrak{K}(A) \) is isomorphic to 
\begin{equation}
  \Loop^\nu \add \CC c \add \CC d,
\end{equation}
where \( \CC c \) is the one-dimensional center of
\( \mathfrak{K}(A) \), \( d \) is the additional derivation element
that acts on \( \Loop^\nu \) as \( z \frac{d}{dz} \) and
the Lie bracket is described by
\begin{equation}
[ z^k x, z^{\ell} y ] = z^{k+\ell}[x,y] + kB(z^k x, z^{\ell} y)c
\qquad \forall z^k x, z^\ell y \in \Loop.
\end{equation}
Consequently \( \Loop^\nu \cong [\mathfrak{K}(A), \mathfrak{K}(A)]/\CC c \) and the form 
\eqref{eq: form on L sigma}
on \( \Loop^\nu \)
extends to a
standard bilinear form on 
\( \mathfrak{K}(A) \) in the sense of \cite[Section 2]{kac_book}.

\section{The standard Lie bialgebra structure on \( \Loop^\sigma \) and its twists}\label{sec: standard bialgebra}
Let \( \mathfrak{K}(A) \) be a
symmetrizable Kac-Moody algebra
with a fixed invariant non-degenerate symmetric
bilinear form \( B \). Then it possesses a Lie bialgebra structure \( \delta_0 \), called the standard Lie bialgebra structure on \( \mathfrak{K}(A) \),
given by
\begin{equation}\label{eq: standard bialgebra on g(A)}
\delta_0(H_i) = 0,
\qquad
\delta_0(D_i) = 0,
\qquad
\delta_0(X_i^\pm) = \frac{B(\dual{\alpha_i}, \dual{\alpha_i})}{2} H_i \wedge X_i^\pm,
\end{equation}
where
\( \{X_i^-, H_i, X_i^+ \} \cup \{ D_i \} \)
is a set of standard generators for
\( \mathfrak{K}(A) \) (see \cite[Example 3.2]{drinfeld_quantum_groups} and \cite[Example 1.3.8]{chari_pressley}).
We can immediately see that \( \delta_0 \)
induces a Lie bialgebra structure on
\begin{equation} 
[\mathfrak{K}(A), \mathfrak{K}(A)]/ Z(\mathfrak{K}(A)),
\end{equation}
where \( Z(\mathfrak{K}(A)) \) is the center of \( \mathfrak{K}(A) \).
In particular, when \( A \) is an affine matrix
and \( B \) is the form mentioned in
Subsection \ref{rem: g(A) as L + Cc + Cd}
we get a Lie bialgebra structure \( \delta_0^\nu \) on
\( \Loop^\nu \).
Applying the methods described in Section
\ref{sec: Loop algebras}
we induce a Lie bialgebra structure
\( \delta_0^\sigma \),
called
\emph{the standard Lie bialgebra structure},
on \( \Loop^\sigma \) for any finite order
automorphism \( \sigma \).
Its twisted versions
\( \delta_t^\sigma \)
are called 
\emph{twisted standard structures}.

\subsection{Pseudoquasitriangular structure}\label{sec: pseudoquasitriangular structure}
We want to prove that \( \delta_t^\sigma \)
is a pseudoquasitriangular Lie bialgebra structure, i.e.
it is defined by an \rmatr.
We restrict our attention to a special case
\( \sigma = \sigma_{(s; |\nu|)} \).
The general result will then follow from
the natural isomorphism mentioned in
Remark \ref{rem: L^nu is determined by ord(nu)}.

Let \( C_k^\sigma \) be the projection
of the Casimir element
\( C = \sum_{k = 0}^{|\sigma| - 1} C_k^\sigma
\in \g \ot \g \)
on the eigenspace
\( \g^\sigma_{k} \ot \g^\sigma_{-k} \).
The triangular decomposition
\( \g_0^\sigma = \n^\sigma_- \add \h \add \n^\sigma_+ \) leads to the splitting
\( C_0^\sigma =
C_{-}^{\sigma} + C_{\h} + C_{+}^{\sigma} \),
where \( C_\pm^\sigma \in \n^\sigma_\pm \ot \n^\sigma_\mp \)
and \( C_\h \in \h \ot \h \).
We introduce a rational function
\( r_0^\sigma \colon \CC^2 \longrightarrow \g \ot \g \)
defined by
\begin{equation}\label{eq: r 0 sigma}
r_0^\sigma(x,y) \coloneqq
\frac{C_\h}{2} + C_-^\sigma +
\frac{1}{(x/y)^{|\sigma|} - 1} \sum_{k = 0}^{|\sigma|-1} 
\seq{\frac{x}{y}}^k C_k^\sigma.
\end{equation}

\begin{remark}
Formula \eqref{eq: r 0 sigma}
can be seen as a generalization of
well-known \rmatrs.
P. Kulish introduced \( r^{\sigma_{(\mathbf{1};1)}}_0 \) in \cite{kulish}. More generally \( r^{\sigma_{(\mathbf{1};|\nu|)}}_0 \) was introduced in \cite{belavin_drinfeld_solutions_of_CYBE_paper} by A. Belavin and V.Drinfeld, which they later, in \cite{belavin_dirnfeld_triangle}, called the simplest trigonometric solution. M. Jimbo used
\( r^{\nu}_0 \) in \cite{jimbo} and the formula for
\(r_0^{\id}\) appears in the recent works
\cite{KPSST, pop_stolin} and \cite{ burban_galinat_stolin} under the
name ``quasi-trigonometric \rmatr''.
\end{remark}

The statement in
\cite[Lemma 6.22]{belavin_drinfeld_solutions_of_CYBE_paper} suggests
the following holomorphic relations between functions defined by \eqref{eq: r 0 sigma}.

\iffalse
The functions in \eqref{eq: r 0 sigma} are related by holomorphic equivalences, as can be seen by the following generalisation of \cite[Lemma 6.22]{belavin_drinfeld_solutions_of_CYBE_paper}.
\fi

\begin{lemma}\label{thm: GxG r = r}
Let \( \sigma, \sigma' \in \textnormal{Aut}_{\CC - \textnormal{LieAlg}}(\g) \) be
two automorphisms of types
\( (s; |\nu|) \) and \( (s'; |\nu|) \)
respectively, where \( s = (s_0, s_1, \ldots, s_n) \) and
\( s' = (s_0', s_1', \ldots, s_n') \). Then
\begin{enumerate}
\item The equations
\( \alpha_i(\mu) = s'_i / |\sigma'| - s_i/|\sigma|\),
\( i \in \{0,1, \ldots, n\} \), 
define a unique element \( \mu \in \h \) such that
\begin{equation}\label{eq: Gs = exp}
e^{u \, \ad(\mu)} f\seq{e^{u / |\sigma|}}
=
\seq{G_s^{s'} f}\seq{e^{u / |\sigma'|}}
\qquad
\forall f \in \Loop^\sigma, \ \forall u \in \CC;
\end{equation}\label{eq: Gs ot Gs r = r'}
\item The functions \( r^\sigma_0 \)
and \( r^{\sigma'}_0 \) satisfy the relation
\begin{equation}\label{eq: (e ot e)r = r}
  \seq{e^{u \, \ad(\mu)} \ot e^{v \, \ad(\mu)} }
  r^\sigma_0 \seq{e^{u/|\sigma|}, e^{v/ |\sigma|}} = 
  r^{\sigma'}_0 \seq{e^{u/ |\sigma'|}, e^{v / |\sigma'|}}
  \qquad
  \forall u, v \in \CC, u - v \notin 2\pi i \ZZ.
\end{equation}
\end{enumerate}
\end{lemma}

\begin{proof}
Using the formulas
\begin{equation}
|\sigma|= |\nu| \sum_{i = 0}^n a_i s_i
\
\text{ and }
\
|\sigma'| = |\nu| \sum_{i = 0}^n a_i s'_i,
\end{equation}
we can easily
deduce that the equations
\( \alpha_i(\mu) = s'_i / |\sigma'| - s_i/|\sigma| \)
are consistent and define a unique element
\( \mu \in \h \).
Let \( f = X_i^\pm \), \( i \in \{ 0, 1, \ldots, n \} \). Then for all \( u \in \CC \) we have
\begin{equation}\label{eq: number 2}
\begin{split}
e^{u \, \ad(\mu)} X_i^\pm \seq{e^{u / |\sigma|}}
&=
e^{\pm u s_i / |\sigma|}
\sum_{k \geq 0} \frac{u^k}{k!} 
\seq{\pm \frac{s_i'}{|\sigma'|} \mp \frac{s_i}{|\sigma|}}^k
X_i^\pm(1) \\
&=
e^{\pm u s_i'/|\sigma'|} X_i^\pm(1) 
=
\seq{G_s^{s'} X_i^{\pm}}\seq{e^{u / |\sigma'|}}.
\end{split}
\end{equation}
Since \( \Loop^\sigma \) is generated by \( X_i^\pm \), identity \eqref{eq: number 2}
proves the first statement.
To verify the second statement we choose a basis
\( \{ b_{(\alpha,k)}^i \} \)
for each \( \Loop^\sigma_{(\alpha,k)} \)
such that
\begin{equation}
B\seq{b^i_{(\alpha, k)}, b^j_{(-\alpha,-k)}} = \delta_{ij}
\qquad
\forall (\alpha, k) \in \Phi_\sigma.
\end{equation}
Setting
\( n_{(\alpha, k)} \coloneqq \dim(\Loop^\sigma_{(\alpha, k)}) \)
we can write
\begin{equation}
  \seq{\frac{y}{x}}^{-k + \ell |\sigma|} C^\sigma_k
  =
  \sum_{\substack{(\alpha, k) \in \Phi_{\sigma}^+ \\ 1 \leq i \leq n_{(\alpha, k)}}}
  b^i_{(-\alpha, k - \ell |\sigma|)}(x) \ot b^i_{(\alpha, -k + \ell |\sigma|)}(y)
  \qquad
  \forall x,y \in \CC^*,
\end{equation}
where \( \Phi_{\sigma}^+ \) stands for the set of positive roots
of \( \Loop^\sigma \).
Then the Taylor series of \( r_0^\sigma \)
in \( y = 0 \) for a fixed \( x \) is
\begin{equation}\label{eq: tensor expr to r_0 sigma}
r_0^\sigma(x,y)
=
%\frac{C_\h}{2} + C_-^\sigma -
%\sum_{l = 0}^\infty \sum_{k=0}^{m-1}
%\sum_{\substack{(\alpha, k) \in \Phi^+ \\ 1 \leq i \leq n_{(\alpha, k)}}}
%  b^i_{(\alpha, k + lm)}(x) \ot b^i_{(-\alpha, -k - lm)}(y)
\frac{C_\h}{2} + 
\sum_{\substack{(\alpha, k) \in \Phi_{\sigma}^+ \\ 1 \leq i \leq n_{(\alpha, k)}}}
  b^i_{(-\alpha, -k)}(x) \ot b^i_{(\alpha, k)}(y).
\end{equation}
It converges absolutely in \( |y| < |x| \) allowing us to perform the 
following calculation
\begin{align*}
\seq{e^{u \, \ad(\mu)} \ot e^{v \, \ad(\mu)} }
r^\sigma_0 \seq{e^{u/|\sigma|}, e^{v/ |\sigma|}}
&= 
\frac{C_\h}{2} + 
\sum_{\substack{(\alpha, k) \in \Phi_{\sigma}^+ \\ 1 \leq i \leq n_{(\alpha, k)}}}
  e^{u \, \ad(\mu)} b^i_{(-\alpha, -k)}\seq{e^{u/|\sigma|}}
  \ot
  e^{v \, \ad(\mu)} b^i_{(\alpha, k)}\seq{e^{v/|\sigma|}} \\
&=
\frac{C_\h}{2} + 
\sum_{\substack{(\alpha, k) \in \Phi_{\sigma}^+ \\ 1 \leq i \leq n_{(\alpha, k)}}}
  \seq{G_s^{s'} b^i_{(-\alpha, -k)}}\seq{e^{u/|\sigma'|}}
  \ot
  \seq{G_s^{s'} b^i_{(\alpha, k)}}\seq{e^{v/|\sigma'|}} \\
&= 
r_0^{\sigma'} \seq{e^{u / |\sigma'|}, e^{v / |\sigma'|}},
\end{align*}
for \( |e^{v/|\sigma|}| < |e^{u/|\sigma|}| \) or, equivalently,
\( |e^{v/|\sigma'|}| < |e^{u/|\sigma'|}| \).
Equality
\eqref{eq: (e ot e)r = r} now
follows by the
identity theorem for holomorphic
functions of several variables (see \cite{gunning_rossi_1965}).
\end{proof}

Having this result at hand we can obtain
the desired pseudoquasitriangularity
for twisted standard structures \( \delta_t^\sigma \). Let us call a meromorphic function
\( r : \CC^2 \longrightarrow \g \ot \g \)
\emph{skew-symmetric}
if \( r(x,y) + \tau(r(y,x)) = 0 \).
\begin{theorem}\label{thm: delta = [, r]}
Let
\( \sigma \in \textnormal{Aut}_{\CC - \textnormal{LieAlg}}(\g) \)
be a finite order automorphism and \(t \in \Loop^\sigma \ot \Loop^\sigma\). 
Then
\( r_t^\sigma \coloneqq r_0^\sigma + t\) is a skew-symmetric solution of the CYBE if and only if \(t\) is a classical twist of \(\delta_0^\sigma\).
Moreover, if \( t \) is a classical twist of
\( \delta_0^\sigma \), then the
following relation holds:
\footnote{We define \((f \ot g)(x,y) \coloneqq f(x) \ot g(y)\) for any \(x,y \in \CC^*\) and \(f,g \in \Loop^\sigma\).}
\begin{equation}\label{eq: delta = [, r]}
\delta_t^\sigma(f)(x,y)
=
[f(x) \ot 1 + 1 \ot f(y), r_t^\sigma(x,y)]
\qquad
\forall f \in \Loop^\sigma, \ \forall x,y \in \CC^*.
\end{equation}
\end{theorem}

\begin{proof}
First, assume that \( \sigma = \sigma_{(\mathbf{1}; |\nu|)} \) and \( t = 0 \). In this case \cite[Proposition 6.1]{belavin_drinfeld_solutions_of_CYBE_paper} implies that \(r^\sigma_0\) is a skew-symmetric solution of the CYBE and \eqref{eq: delta = [, r]}
follows immediately from comparing
\cite[Equations (6.4) and (6.5)]{belavin_drinfeld_solutions_of_CYBE_paper} with the projections of the defining relations \eqref{eq: standard bialgebra on g(A)} to \(\Loop^\sigma\).
Secondly, applying Lemma \ref{thm: GxG r = r}
we get the statement for an arbitrary finite order automorphism \( \sigma \) and 
\( t = 0 \).
Finally, since \(r^\sigma_0\) is skew-symmetric, the skew-symmetry of \( t \) is equivalent to the skew-symmetry of \( r^\sigma_t \) and a straightforward computation gives the equality
%Finally, for any tensor \(t\), since \(r^\sigma_0\) is skew-symmetric, the skew-symmetry of \(r^\sigma_t\) and \(t\) is equivalent and 
%a straightforward computation
%gives the equality
\begin{equation}
    \textnormal{CYB}(r^\sigma_t) =  \textnormal{CYB}(r^\sigma_0) + \textnormal{CYB}(t) - \textnormal{Alt}((\delta_0^\sigma \ot 1)t) = \textnormal{CYB}(t) - \textnormal{Alt}((\delta_0^\sigma \ot 1)t),
\end{equation}
which completes the proof.
\end{proof}

We finish this subsection by relating \rmatrs
of the form \( r_t^\sigma \) to trigonometric
\rmatrs in the sense of the Belavin-Drinfeld classification
\cite{belavin_drinfeld_solutions_of_CYBE_paper}. 

\begin{theorem}\label{thm: r(x,y)=X(u-v)}
Let \( t \) be a classical twist of the standard
Lie bialgebra structure \( \delta_0^\sigma \)
on \( \Loop^\sigma \) and \( r_t^\sigma = r_0^\sigma + t \)
be the corresponding \rmatr .
Then there exists
a holomorphic function
\( \varphi \colon \CC \longrightarrow \textnormal{Inn}_{\CC - \textnormal{LieAlg}}(\g) \)
and a trigonometric \rmatr
\( X \colon \CC \longrightarrow \g \ot \g \) such that
\begin{equation}
X(u-v) = (\varphi(u)^{-1} \ot \varphi(v)^{-1}) r(e^{u/|\sigma|}, e^{v/|\sigma|}).
\end{equation}
\end{theorem}

\begin{proof}
Let
\begin{equation}
  r(x,y) \coloneqq r^\sigma_t(x,y) = \frac{1}{(x/y)^{|\sigma|} - 1}
  \widetilde{C} \seq{x/y}
  + g(x,y),
\end{equation}
where
\( \widetilde{C}(z) \coloneqq \sum_{k = 0}^{{|\sigma|}-1}
z^k C^\sigma_k\).
Following the arguments in \cite{belavin_drinfeld_diffrence_depending} and \cite[Theorem 11.3]{KPSST} 
we rewrite the CYBE for
\( r \) in the form
\begin{align*}
[r^{12}(x,y), r^{13}(x,z)] +
[r^{12}(x,y) + r^{13}(x,z), g^{23}(y,z)] +
\frac{1}{(y/z)^{|\sigma|} - 1}
[r^{12}(x,y) + r^{13}(x,z), \widetilde{C}^{23}(y/z)] = 0.
\end{align*}
Calculating the limit \( y \to z \) using L'Hospital's rule we obtain
\begin{align*}
[r^{12}(x,z), r^{13}(x,z)] +
[r^{12}(x,z) + r^{13}(x,z), g^{23}(z,z) + {|\sigma|}^{-1}
(\widetilde{C}'(1))^{23}] +
\frac{z}{{|\sigma|}}
[ \partial_z r^{12}(x,z),
\widetilde{C}^{23}(y/z)] = 0.
\end{align*}
Applying the function
\( 1 \ot L \colon \g \ot \g \ot \g \longrightarrow \g \ot \g \),
where \( L(a \ot b) \coloneqq [a,b] \), we get the
equality
\begin{equation}\label{eq: diff eq 1}
[r(x,z), r(x,z)] + [r(x,z), 1 \ot f(z)]
+ \frac{z}{{|\sigma|}} \partial_z r(x,z) = 0,
\end{equation}
where \( f(z) \coloneqq L(g(z,z) + {|\sigma|}^{-1}
\widetilde{C}'(1)) \) and
\( [a \ot b, c \ot d] \coloneqq [a,c] \ot [b,d] \).
Similarly, letting \( x \to y \) in the CYBE for
\( r \) and then applying \( L \ot 1 \) we obtain the identity
\begin{equation}\label{eq: diff eq 2}
[r(y,z), r(y,z)] - [r(y,z), f(y) \ot 1]
- \frac{y}{{|\sigma|}} \partial_y r(y,z) = 0.
\end{equation}
Subtracting \eqref{eq: diff eq 1} from
\eqref{eq: diff eq 2} and setting
\( x = y = e^{u/{|\sigma|}} \) and \( z = e^{v/{|\sigma|}} \) we get
\begin{equation}\label{eq: diff eq}
\partial_u r(e^{u/{|\sigma|}}, e^{v/{|\sigma|}})
+
\partial_v r(e^{u/{|\sigma|}}, e^{v/{|\sigma|}})
=
[h(u) \ot 1 + 1 \ot h(v), r(e^{u/{|\sigma|}}, e^{v/{|\sigma|}})],
\end{equation}
for \( h(u) \coloneqq f(e^{u/{|\sigma|}}) \).
Since \( h \) is holomorphic on \( \CC \), we can
find a holomorphic function
\( \varphi \colon \CC \longrightarrow
\textnormal{Aut}_{\CC - \textnormal{LieAlg}} (\g)  \)
such that
\( \varphi'(z) =  \ad(h(z))\varphi(z) \) and
\( \varphi(0) = \textnormal{id}_{\g} \)
(see \cite[Proof of Theorem 11.3]{KPSST}). 
The connected component of
\( \textnormal{id}_{\g} \)
in the group \( \textnormal{Aut}_{\CC - \textnormal{LieAlg}} (\g) \)
is exactly the inner
automorphisms of \( \g \) and thus
\( \varphi \colon \CC \longrightarrow 
\textnormal{Inn}_{\CC - \textnormal{LieAlg}} (\g)
\). Finally, the relation \eqref{eq: diff eq}
implies that the \rmatr
\begin{equation}
\widetilde{X}(u,v)
\coloneqq
(\varphi(u)^{-1} \ot \varphi(v)^{-1})r(e^{u/{|\sigma|}}, e^{v/{|\sigma|}})
\end{equation}
satisfies the equation
\(
\partial_u \widetilde{X}(u,v)
+
\partial_v \widetilde{X}(u,v) = 0 \).
Therefore, we can define
\( X(u-v) \coloneqq
\widetilde{X}(u-v, 0) = \widetilde{X}(u,v) \).
The set of poles of \( X \) is \( 2\pi i \ZZ \)
and hence it is a trigonometric \rmatr.
\end{proof}

From now on \rmatrs of the form
\(r^\sigma_t = r_0^\sigma + t \),
where \( \sigma \) is a finite order automorphism
of \( \g \) and
\( t \) is a classical
twists of \( \delta_0^\sigma \), are called
\emph{\( \sigma \)-trigonometric}.

\subsection{Manin triple structure}\label{sec: Manin triples structure}
The standard Lie bialgebra stucture on an
affine Kac-Moody algebra
\eqref{eq: standard bialgebra on g(A)}
can be defined using the standard
Manin triple
(see \cite[Example 3.2]{drinfeld_quantum_groups} and \cite[Example 1.3.8]{chari_pressley}).
Restricting that triple to
\( \Loop^\sigma \) we get a Manin triple
defining the standard Lie bialgebra structure \(\delta^\sigma_0\)
on \( \Loop^\sigma \).
More precisely, \( \delta_0^\sigma \)
is defined by the Manin triple
\begin{equation}\label{eq: standard Manin triple}
\seq{\Loop^\sigma \times \Loop^\sigma, \Delta, W_0},
\end{equation}
where \( \Delta \) is the
image of the diagonal embedding
of \( \Loop^\sigma \) into \( \Loop^\sigma \times \Loop^\sigma \) and \( W_0 \)
is defined by
\begin{equation}
W_0
\coloneqq
\set{(f, g) \in \B^\sigma_+ \times \B^\sigma_- \mid
f + g \in \N^\sigma_+ \dotplus \N^\sigma_-}.
\end{equation}
The form \( \mathcal{B} \)
on \( \Loop^\sigma \times \Loop^\sigma \)
is given by
\begin{equation}\label{eq: form on LxL}
  \mathcal{B} \seq{\seq{f_1, f_2}, \seq{g_1, g_2}}
  \coloneqq
  B(f_1, g_1)
  -
  B(f_2, g_2)
  \qquad
  \forall f_1, f_2, g_1, g_2 \in \Loop^\sigma,
\end{equation}
where \( B \) is the form
\eqref{eq: form on L sigma}.
From Theorem
\ref{thm: W <-> t <-> T correspondence} we know
that classical twists \( t \) 
of \( \delta_0^\sigma \)
are in one-to-one correspondence with
Lagrangian subalgebras
\( W_t \subseteq \Loop^\sigma \times \Loop^\sigma \) complementary to \( \Delta \) and
commensurable with \( W_0 \).
We now describe the construction of such
subalgebras using \( \sigma \)-trigonometric
\rmatrs.

Let
\( \psi \colon \g \ot \g \longrightarrow \textnormal{End}_{\CC - \textnormal{Vect}} (\g) \)
and
\( \Psi \colon \Loop^\sigma \ot \Loop^\sigma \longrightarrow \textnormal{End}_{\CC - \textnormal{Vect}} (\Loop^\sigma) \)
be the natural maps given by 
\( a \ot b \longmapsto \K(b, -)a \) and
\( a \ot b \longmapsto B(b, -)a \) respectively.
Then we have the following useful identity
\begin{equation}\label{eq: useful for R_t}
\textnormal{res}_{y=0}
\left[
\frac{1}{y}
\psi(P(z,y))(f(y))
\right]
=
\Psi(P)(f)(z)
\qquad
\forall P \in \Loop^\sigma \ot \Loop^\sigma, \
\forall f \in \Loop^\sigma, \
\forall z \in \CC^*.
\end{equation}

\begin{theorem}
Let \( t \) be a classical twist of the standard
Lie bialgebra structure \( \delta_0^\sigma \)
on \( \Loop^\sigma \) and 
\( r_t = r_0^\sigma + t \)
be the corresponding
\( \sigma \)-trigonometric \rmatr.
Denote by \( \pi_\h \) and \( \pi_\pm \) the
projections of \( \Loop^\sigma \) onto
\( \h \) and \( \N^\sigma_\pm \) respectively.
Then the linear map 
\( R_t \coloneqq  \pi_\h / 2 + \pi_- + \Psi(t) \) satisfies
the relation
\begin{equation}\label{eq: def eq for r_t}
\textnormal{res}_{y=0}
\left[
\frac{1}{y}
\psi(r_t(z,y))(f(y))
\right]
=
R_t(f)(z)
\qquad
\forall f \in \Loop^\sigma, \
\forall z \in \CC^*,
\end{equation}
and the Lagrangian subalgebra \( W_t \),
corresponding to \( t \), can be described
in the following way
\begin{equation}\label{eq: Wt using Rt}
W_t = \set{\seq{\seq{R_t - 1}f, R_tf} \mid
f \in \Loop^\sigma}.
\end{equation} 
\end{theorem}
\begin{proof}
We prove the theorem for
\( \sigma = \sigma_{(s;|\nu|)} \) and \( t=0 \).
The general result then follows by linearity and equation \eqref{eq: useful for R_t}.
Writing \( r_0^\sigma(z,y) \) as series
\eqref{eq: tensor expr to r_0 sigma} and applying
\( \psi \) we get
\begin{equation}
\psi(r_0^\sigma(z,y))(f(y))
=
\frac{\psi(C_\h)(f(y))}{2}
+
\sum_{\substack{(\alpha, k) \in \Phi_\sigma^+ \\ 1 \leq i \leq n_{(\alpha, k)}}}
\psi\seq{b^i_{(-\alpha, -k)}(z) \ot b^i_{(\alpha, k)}(y)}\Big(f(y)\Big).
\end{equation}
The absolute convergence of the series in the
annulus
\( \epsilon < |y| < |z| \)
for any \( \epsilon \in \RR_+ \)
allows the componentwise
calculation of the residue, i.e.
\begin{equation}
\textnormal{res}_{y=0}
\left[
\frac{1}{y}
\psi(r_t(z,y))(f(y))
\right]
=
\frac{\pi_\h (f(y))}{2} + 
\sum_{(\alpha, k) \in \Phi_\sigma^+}
\pi_{(-\alpha,-k)}(f(y))
=
R_0(f)(z),
\end{equation}
where \( \pi_{(\alpha,k)} \) is the projection of 
\( \Loop^\sigma \) onto \( \Loop^\sigma_{(\alpha,k)} \).

For the second statement let us take an arbitrary
\( (w_1, w_2) \in W_0 \). The relation
\( \pi_\h(w_1) = - \pi_\h(w_2) \) implies 
\( (w_1, w_2) = ((R_0 - 1)(w_2 - w_1), R_0 (w_2 - w_1)) \). The desired result now follows from the
fact that \( (w_1, w_2) \longmapsto w_2 - w_1 \) is
an isomorphism between \( W_0 \) and \( \Loop^\sigma \).
\end{proof}

\begin{remark}\label{rem: large double}
For later sections it is convenient to define 
another, more geometric,
Manin triple defining the standard Lie
bialgebra structure \( \delta_0^\sigma \) on
\( \Loop^\sigma \).
Define \( m \coloneqq |\sigma| \), 
\( O^\sigma \coloneqq \CC[z^m, z^{-m}] \) and \( \widehat{O}^\sigma_{\pm} \coloneqq \CC(\!(z^{\pm m})\!) \).%
\footnote{The notation \( \CC(\!(u)\!) \) is used to
  denote the ring of Laurent series of the form
  \( \sum_{k = N}^{\infty} a_k u^k \), where
  \( a_k \in \CC \) and \( N \in \ZZ \).}
The Lie algebra
\( \Loop^\sigma \) is naturally an \( O^\sigma \)-module
and hence we can extend it to
\( \widehat{\Loop}^\sigma_\pm = \Loop^\sigma \ot_{O^\sigma} \widehat{O}^\sigma_\pm \).
Equip the product Lie algebra
\( \widehat{\Loop}^\sigma_+ \times \widehat{\Loop}^\sigma_- \)
with the following bilinear form
\begin{equation}
\mathcal{B}((f_1, f_2), (g_1, g_2))
\coloneqq
\textnormal{res}_{z=0}
\left[
\frac{1}{z}
\K(f_1, g_1)
\right]
-
\textnormal{res}_{z=0}
\left[
\frac{1}{z}
\K(f_2, g_2)
\right],
\end{equation}
where 
\( \K(\sum_i a_i z^i, \sum_j b_j z^j)
\coloneqq
\sum_{i,j} \K(a_i, b_j) z^{i+j}   \) and
\( \textnormal{res}_{z = 0} \) reads off the
coefficient of \( z^{-1} \). The restriction
of this form to \( \Loop^\sigma \times \Loop^\sigma \)
is the form \eqref{eq: form on LxL} defined earlier.
Consider the subset
\begin{equation}
  \widehat{W}_0
  =
  \set{(f, g) \in
  \widehat{\B}^\sigma_+ \times \widehat{\B}^\sigma_- \mid
  \widehat\pi^+_\h(f) =  -\widehat\pi^-_\h(g)}
     \subseteq \widehat{\Loop}^\sigma_+ \times
     \widehat{\Loop}^\sigma_- ,
  \end{equation}
where 
\( \widehat{\N}^\sigma_\pm \) stands for the
completion of \( \N^\sigma_\pm \)
with respect to the ideal
\( (z^{\pm m}) \subseteq \CC[z^{\pm m}] \), \( \widehat{\B}^\sigma_\pm \coloneqq
\h \add \widehat{\N}^\sigma_\pm \) and \(\widehat\pi_\h^\pm \colon \widehat\Loop^\sigma_\pm \longrightarrow \h\) are the canonical projections
. Then \(\widehat{W}_0\)
  is a Lagrangian subalgebra
  complementary to the diagonal embedding
  \( \Delta \) of \( \Loop^\sigma \) into
  \( \widehat{\Loop}^\sigma_+ \times
     \widehat{\Loop}^\sigma_- \).
Since the Lie bracket on \( W_0 \)
is the restriction of the Lie bracket on 
\( \widehat{W}_0 \), the Manin triple
\begin{equation}\label{eq: geometric manin triple}
  (\widehat{\Loop}^\sigma_+ \times
     \widehat{\Loop}^\sigma_-, \Delta, 
     \widehat{W}_0)
\end{equation}
also defines the standard Lie bialgebra structure 
\( \delta_0^\sigma \) on \( \Loop^\sigma \).

The geometric nature of this Manin triple is revealed in
\cite{TBA}: the sheaves used for construction of \( \sigma \)-trigonometric \rmatrs can be viewed as formal gluing of twisted versions of \( \widehat{W}_0 \) with
\( \Loop^\sigma \cong \Delta \) over the nodal Weierstraß cubic.
\iffalse
In Section \ref{sec: survey on geometry} we will explain the the results of \cite{burban_galinat}, which permits the construction of \rmatrs from certain sheaves of Lie algebras on Weierstraß cubic curves. The result in \cite{TBA}, that all \(\sigma\)-trigonometric \rmatrs arise in this way, will turn out to be vital to the proof of our main classification theorem \ref{thm: classification}. We want to remark here, that the sheaves, which are used to construct the \(\sigma\)-trigonometric \rmatrs, can be viewed as a formal gluing of the twisted versions of \(\widehat{W}_0\) with \(\Loop^\sigma \cong \Delta\) over the nodal Weierstraß cubic. Hence we can think of these Manin triples to be geometric in nature.
\fi
\end{remark}

\subsection{Regular equivalence}\label{subsec: Gauge equivalence}
Let us fix a finite order
automorphism \( \sigma \) of \( \g \).
We now turn to defining the notion of equivalence for twisted standard bialgebra structures on \( \Loop^\sigma \) which is compatible with
the corresponding pseudoquasitriangular and Manin triple structures. In other words, we want equivalences of Lie bialgebras to induce equivalences of the corresponding Manin triples and trigonometric \rmatrs and vice versa.
We stress that the notion of holomorphic equivalence used in the
Belavin-Drinfeld classification \cite{belavin_drinfeld_solutions_of_CYBE_paper} is unsuitable for our purpose, because in general it does not provide isomorphisms of loop algebras.

In the spirit of
\cite{stolin_sln, stolin_maximal_orders} we define a
\emph{regular equivalence} on the loop algebra \( \Loop^\sigma \) to be a regular function \(\phi \colon \CC^* \longrightarrow \textnormal{Aut}_{\CC-\textnormal{LieAlg}}(\g)\) preserving the quasi-periodicity of \( \Loop^\sigma \), i.e.
\begin{equation}
\phi(\varepsilon_\sigma z) = \sigma \phi(z) \sigma^{-1},
\end{equation}
where \( \varepsilon_\sigma = e^{2 \pi i / |\sigma|} \).
Recalling that \( O^\sigma = \CC[z^{|\sigma|}, z^{-|\sigma|}] \),
we can equivalently define a regular equivalence on \( \Loop^\sigma \)
to be an element of
\( \textnormal{Aut}_{O^\sigma - \textnormal{LieAlg}} (\Loop^\sigma) \). The equivalence between these two definitions is given by
\( \phi(f)(z) \coloneqq \phi(z)f(z) \)
\iffalse
The objective of this subsection is to define the notion of regular equivalence for twisted standard bialgebra structures, that is isomorphisms of bialgebras between different twisted standard structures, which are compatible with the pseudoquasicoboundary structure. 
Note that the holomorphic equivalences
used in \cite{belavin_drinfeld_solutions_of_CYBE_paper} are
unsuitable since in general they are not
isomorphisms of loop algebras. Rather we stick closer in spirit to the equivalences given in \cite{stolin_sln} \& \cite{stolin_maximal_orders}. Namely we define a \emph{regular equivalence} on the loop algebra \(\Loop^\sigma\) to be a regular function \(\phi \colon \CC^* \longrightarrow \textnormal{Aut}_{\CC-\textnormal{LieAlg}}(\g)\) s.t. \(\phi\) preserves the quasi periodicity of \(\Loop^\sigma\), i.e. \(\phi(\varepsilon_\sigma z) = \sigma \phi(z) \sigma^{-1}\), where we recall that \(\varepsilon_\sigma = e^{2\pi i/|\sigma|}\). The rule \(\phi(f)(z) := \phi(z)f(z)\) then defines an automorphism of \(\Loop^\sigma\). In fact, exactly all automorphisms of \(\Loop^\sigma\) which are linear with respect to the \(O^\sigma\)-module structure, arise in this way, where we recall that \(O^\sigma\) are the Laurent polynomials in \(z^{|\sigma|}\).
\fi

By definition
\eqref{eq: Loop sigma regular funcs}
the space \( \Loop^\sigma \ot \Loop^\sigma \) can be viewed as the space of regular functions 
\( T \colon \CC^* \times \CC^* \longrightarrow \g \ot \g \) such that
\( (1 \ot \sigma)T(x,y) = T(x,\varepsilon_\sigma y) \) and
\( (\sigma \ot 1)T(x,y) = T(\varepsilon_\sigma x,y) \).
It is straightforward to check that if such a function \( T \) vanishes
along the diagonal, i.e.
\( T(z,z) = 0 \) for all \( z \in \CC^* \), then it is divisible by
\( (x/y)^{|\sigma|} -1 \).
Applying this observation to the function
\begin{equation}
(\phi(x) \ot \phi(y)) \sum_{k = 0}^{|\sigma|-1} \seq{\frac{x}{y}}^k C^\sigma_k - 
\sum_{k = 0}^{|\sigma|-1} \seq{\frac{x}{y}}^k C^\sigma_k,
\end{equation}
where \( \phi \in \textnormal{Aut}_{O^\sigma - \textnormal{LieAlg}} (\Loop^\sigma) \), we see that
\( (\phi(x) \ot \phi(y))r_t^\sigma(x,y) = r_0^\sigma(x,y) + s(x,y) \) for some classical twist \( s \), i.e. it is again
a \( \sigma \)-trigonometric \rmatr.

\iffalse
We now explain what we mean by compatibility with the pseudoquasitriangular structure.
Let \( r_t^\sigma \) be a \( \sigma \)-trigonometric
\rmatr. The identities 
\eqref{eq: Loop regular funcs} and
\eqref{eq: Loop sigma regular funcs}
imply that 
\( \Loop^\sigma \otimes \Loop^\sigma \) can be viewed
as the space of regular functions
\( T : \CC^* \times \CC^* \longrightarrow \g \ot \g \) such that
\( (1 \ot \sigma)T(x,y) = T(x, \varepsilon_\sigma y) \)
and
\( (\sigma \ot 1)T(x,y) = T(\varepsilon_\sigma x, y) \).
If such a function vanishes along the diagonal,
i.e. \( T(z,z) = 0 \) for all \( z \in \CC^* \), then
it is divisible by \( (x/y)^m - 1 \).
In particular, since the function
\begin{equation}
(\phi(x) \ot \phi(y)) \sum_{k = 0}^{m-1} (x/y)^k C^\sigma_k - \sum_{k = 0}^{m-1} (x/y)^k C^\sigma_k
\end{equation}
vanishes along the diagonal, we can see that
\( (\phi(x) \ot \phi(y))r_t^\sigma(x,y) \) is again
a \( \sigma \)-trigonometric \rmatr. Let us refine the compatibility between the different descriptions of of twisted standard structures under this regular equivalence notion.
\fi

The following theorem demonstrates that the notion of a regular equivalence meets all our needs.
\begin{theorem}\label{thm: Gauge equivalence}
Let \( \phi \) be a regular equivalence on \(\Loop^\sigma\) and 
\( s,t \in \Loop^\sigma \ot \Loop^\sigma \) be
two classical twists of the standard Lie bialgebra
structure \( \delta_0^\sigma \) on \( \Loop^\sigma \).
The following are equivalent:
\begin{enumerate}
\item \( r^\sigma_t(x,y) = (\phi(x) \ot \phi(y))r^\sigma_s(x,y) \)
  for all \( x,y \in \CC^*, \ x^{|\sigma|} \ne y^{|\sigma|} \);
\item \( \delta^\sigma_t \phi
            =
            (\phi \ot \phi) \delta^\sigma_s \);
\item \( W_t = (\phi \times \phi) W_s \).
\end{enumerate}
\end{theorem}

\begin{proof}

\emph{''1. \( \implies \) 3.''\,:}
If \( r_t(x,y) = (\phi(x) \ot \phi(y))r_s(x,y) \)
for all \( x, y \in  \CC^* \), \( x^{|\sigma|} \neq y^{|\sigma|} \),
then \eqref{eq: def eq for r_t} implies  \( R_t = \phi R_s \phi^* \). 
Since the adjoint of \( \phi(z) \) with respect to
the Killing form is \( \phi(z)^{-1} \), we have
\( \phi^* = \phi^{-1} \).
The formula \eqref{eq: Wt using Rt} applied to both
\( W_t \) and \( W_s \) gives
\begin{equation}
\begin{split}
  (\phi \times \phi)W_s
  &=
  \set{\seq{\phi (R_s - 1) \phi^{-1} (\phi f), 
             \phi R_s \phi^{-1} (\phi f)}
             \mid f \in \Loop^\sigma} = W_t.
\end{split}
\end{equation}

\emph{''3. \( \implies \) 2.''\,:}
Assuming \( W_t = (\phi \times \phi) W_s \), we can easily see that \(\phi \times \phi\)
is an isomorphism of Manin tiples
\((\Loop^\sigma \times \Loop^\sigma, \Delta,W_t)\)
and 
\((\Loop^\sigma \times \Loop^\sigma, \Delta,W_s)\).
Identifying \( \Delta \) with \( \Loop^\sigma \) and applying Remark \ref{rem: isomorphic manin tiples define isomorphic bialgebras}
we immediately get the desired isomorphism
\(\phi  \colon (\Loop^\sigma,\delta_t) \longrightarrow (\Loop^\sigma,\delta_s)\).

\emph{''2. \( \implies \) 1.''\,:}
Since \( \Loop^\sigma \) has no
non-trivial finite-dimensional ideals (see \cite[Lemma 8.6]{kac_book}),
the only element in 
\( \Loop^\sigma \ot \Loop^\sigma \) invariant under
the adjoint action of
\( \Loop^\sigma \) is
\( 0 \).
\iffalse
the space of all invariant elements in
\( \Loop^\sigma \ot \Loop^\sigma \)
under the adjoint action of
\( \Loop^\sigma \) is trivial.
\fi
Applying this result to the equality
\begin{equation}
\begin{split}
  \left[
    \phi(f)(x) \ot 1 + 1 \ot \phi(f)(y), r_t(x,y)
  \right]
  &=
  \delta_t \phi (f)(x,y) 
  =
  (\phi(x) \ot \phi(y)) \delta_s (f)(x,y) \\
  &=
  (\phi(x) \ot \phi(y))
  \left[
    f(x) \ot 1 + 1 \ot f(y), r_s(x,y)
  \right] \\
  &=
  \left[
    \phi(f)(x) \ot 1 + 1 \ot \phi(f)(y),
    (\phi(x) \ot \phi(y)) r_s(x,y)
  \right],
\end{split}
\end{equation}
where \( f \in \Loop^\sigma \) and
\( x,y \in \CC^*, x^{|\sigma|} \neq y^{|\sigma|}\), we get the last implication.
\end{proof}
We say that two twisted standard bialgebra structures
or \(\sigma\)-trigonometric \rmatrs
are \emph{regularly equivalent} if one of the equivalent
conditions in Theorem 
\ref{thm: Gauge equivalence} 
holds.
\iffalse
Note that the Lie bialgebra structures
\( \delta^\sigma_0 \)
and
\( \delta^{\sigma'}_0 \) are automatically
isomorphic if the cosets \( \sigma \textnormal{Inn}_{\CC - \textnormal{LieAlg}}(\g)  \)
and
\( \sigma' \textnormal{Inn}_{\CC - \textnormal{LieAlg}}(\g)  \)
have the same conjugacy class.
The corresponding \rmatrs are
always equivalent in the sense of
Belavin-Drinfeld by Lemma
\ref{thm: GxG r = r}, but not in general
regularly equivalent in our sense.
\fi
% and \( \sigma \)-trigonometric \( r \)-matrices

\section{The main classification theorem and its consequences}
Before stating the main classification theorem
we recall the notion of a Belavin-Drinfeld
quadruple for an arbitrary finite order automorphism
\( \sigma \), defined in \cite{belavin_drinfeld_solutions_of_CYBE_paper},
and then associate it
with a classical twist of the standard
Lie bialgebra structure \( \delta_0^\sigma \).

We start with the case
\( \sigma = \sigma_{(s; |\nu|)} \).
Let \( \Pi^\sigma, \Lambda^\sigma\) and \( \Phi_\sigma \) be as at the end of Section \ref{sec: Loop algebras}.
A Belavin-Drinfeld (BD) quadruple
is a quadruple
\( Q = \seq{\Gamma_1, \Gamma_2, \gamma, t_\h} \),
where \( \Gamma_1 \) and \( \Gamma_2 \)
are proper subsets
of the simple root system \( \Pi^\sigma \),
\( \gamma \colon \Gamma_1 \longrightarrow \Gamma_2 \) is a bijection and
\( t_\h \in \h \wedge \h \) such that
\begin{enumerate}
\item \( B\seq{\dual{\alpha_{\gamma(i)}}, \dual{\alpha_{\gamma(j)}}} = B\seq{\dual{\alpha_i}, \dual{\alpha_j}} \)
for all 
\( \widetilde{\alpha}_i, \widetilde{\alpha}_j \in \Gamma_1 \),
where \(
\widetilde{\alpha}_{\gamma(i)} \coloneqq \gamma(\widetilde{\alpha}_i) \);
\item For any \( \widetilde{\alpha}_i \in \Gamma_1 \)
there is a positive integer \( k \) such that
\( \gamma^k(\widetilde{\alpha}_i) \not \in \Gamma_1 \);
\item \( ( \alpha_{\gamma(i)} \ot 1
+ 1 \ot \alpha_i)(t_{\h} + C_{\h}/2) = 0 \)
for all
\( \widetilde{\alpha}_i \in \Gamma_1 \).
\end{enumerate}
The bijection \( \gamma \) induces
an isomorphism
\( \theta_\gamma \colon \mathfrak{S}^{\Gamma_1} \longrightarrow
\mathfrak{S}^{\Gamma_2}\),
\( \theta_\gamma(z^{\pm s_i}X_i^\pm(1)) \coloneqq z^{\pm s_{\gamma(i)}}X_{\gamma(i)}^\pm(1) \), which we extend by \( 0 \) to the whole
\( \Loop^\sigma \).
Let
\( \Phi_{1} \subseteq \Phi_\sigma \)
be the subset of roots that can be written as linear
combinations of elements in \( \Gamma_1 \).
For each \( \widetilde{\alpha} \in \Phi_1 \)
we choose an element
\( b_{\widetilde{\alpha}} \in \Loop^\sigma_{\widetilde{\alpha}}  \)
such that
\( B(b_{\widetilde{\alpha}}, b_{- \widetilde{\alpha}}) = \nolinebreak 1 \)
and construct the following skew-symmetric
tensor
\begin{equation}\label{eq: tQ}
t^\sigma_Q \coloneqq t_\h + \sum_{\widetilde{\alpha}
\in \Phi^+_{1}}
\sum_{j = 1}^{\infty} b_{-\widetilde{\alpha}} \wedge
\theta^j_\gamma \seq{b_{\widetilde{\alpha}}}
\in \Loop^\sigma \ot \Loop^\sigma,
\end{equation}
where \( \Phi^+_1 = \Phi_1 \cap \Phi^+_\sigma \) and
the second sum has only finitely many non-zero
terms since
\( \theta_\gamma \) is nilpotent
by condition 2.

We write
\( r_Q^\sigma, \delta_Q^\sigma, R_Q\)
and \(  W_Q \) instead of
\( r_{t^\sigma_Q}^\sigma, \delta^\sigma_{t^\sigma_Q},
R_{t_Q^\sigma} \) and \( W_{t_Q^\sigma} \)
respectively. 
In the case \( s = (1, \ldots, 1) \) the functions
\(r_Q^\sigma \) and \( R_Q\) as well as the Cayley transform of \( R_Q \) were studied in details in \cite{belavin_drinfeld_solutions_of_CYBE_paper}.
Using regrading and Lemma \ref{thm: GxG r = r} we derive the following statements:
\begin{itemize}
    \item \(r_t^\sigma\) is a skew-symmetric solution of CYBE.
    Hence Theorem \ref{thm: delta = [, r]} implies that
    \(t_Q^\sigma\) is a classical twist of \(\delta_0^\sigma\);
    \item The inhomogeneous system of linear equations constraining \(t_\h\) is consistent. The dimension of its solution space %the affine space of solutions of this equation has dimension
    is \(\ell(\ell - 1)/2\), where \(\ell = |\Pi^\sigma \setminus \Gamma_1|\);
    \item Setting 
    \( \theta^{\pm}_{\gamma^{\pm 1}} \coloneqq
    \theta_{\gamma^{\pm 1}} |_{\N_{\pm}}  \), we have
    \begin{equation}\label{eq: R_Q}
    R_Q = \theta^+_\gamma (\theta^+_\gamma - \pi_+)^{-1}
    + (\psi(t_\h) + \id_\h/2)
    + (\pi_- - \theta^{-}_{\gamma^{-1}})^{-1};
    \end{equation}
    \item Let \( \h_1 \coloneqq \im(\psi(t_\h) - \id_\h/2) \) and
    \( \h_2 \coloneqq \im(\psi(t_\h) + \id_\h/2) \).
    The Cayley transform of \( R_Q \) is 
    the triple
    \( (C^1_Q, C^2_Q, \theta_Q) \), where
    \begin{equation}
    \begin{split}
    C^1_Q &\coloneqq \im(R_Q - \id) = 
    \N_+ \add \h_1 \add \N_-^{\Gamma_1}, \\
    C^2_Q &\coloneqq \im(R_Q) = \N_+^{\Gamma_2} \add \h_2 \add \N_-,
    \end{split}
    \end{equation}
    and
    \(\theta_Q\) is the unique gluing of \(\theta_\gamma\) with the natural isomorphism 
    \begin{equation}
    \begin{split}
    \phi \colon \frac{\im(\psi(t_\h) - \id_\h /2 )}{\ker(\psi(t_\h) + \id_\h /2 )}
    &\longrightarrow
    \frac{\im(\psi(t_\h) + \id_\h /2 )}{\ker(\psi(t_\h) - \id_\h /2 )}, \\
    [(\psi(t_\h) - \id_\h /2 )(h)]
    &\longmapsto
    [(\psi(t_\h) + \id_\h /2 )(h)],
    \end{split}
    \end{equation}
    which coincides with \( \theta_\gamma \)
    on the intersection of the domains.
    The subalgebra \( W_Q \) is then given by
    \begin{equation}\label{eq: W_Q}
    W_Q = \set{(x,y) \in C^1_Q \times C^2_Q \mid 
    \theta_Q ([x]) = [y]}.
\end{equation}
\end{itemize}
Conjugating  \( \sigma \) by \( \rho \in \textnormal{Aut}_{\CC - \textnormal{LieAlg}} (\g) \) we extend all statements and constructions given above to an arbitrary finite order automorphism of \( \g \). 

\begin{theorem}[The main classification theorem]\label{thm: classification}
For any classical twist \( t \) of the standard
Lie bialgebra structure \( \delta_0^\sigma \) on
\( \Loop^\sigma \) there is a 
regular equivalence 
\( \phi \) of \( \Loop^\sigma \) and
a BD quadruple
%\( \Loop^\sigma \) there is a 
%\( \phi \in \textnormal{Aut}_{\CC[z^m,z^{-m}] - \textnormal{LieAlg}} (\Loop^\sigma) \) and
%a BD quadruple
\( Q = \seq{\Gamma_1, \Gamma_2, \gamma, t_\h} \)
such that
\begin{equation}
\delta^\sigma_t \phi 
= 
(\phi \ot \phi) \delta^\sigma_Q.
\end{equation}
Furthermore, if
\( Q' = (\Gamma'_1, \Gamma'_2, \gamma', t'_\h) \)
is another BD
quadruple,
the twisted bialgebra structures
\( \delta^\sigma_Q \) and
\( \delta^\sigma_{Q'} \)
are regularly equivalent if and only if
there is
an automorphism \( \vartheta \)
of the Dynkin diagram
of \( \Loop^\sigma \) 
such that
\( \vartheta(\Gamma_i) = \Gamma'_i \)
for \( i = 1,2 \),
\( \vartheta \gamma \vartheta^{-1} = \gamma' \) and
\( (\vartheta \ot \vartheta)t_\h = t'_\h \),
which we denote by 
\( \vartheta(Q) = Q' \).
\end{theorem}

We put off the proof of the theorem to Section
\ref{sec: Proof}.
The rest of this section is devoted to various
consequences of
Theorem \ref{thm: classification}
and to the proof of its first part in the
special case \( \g = \mathfrak{sl}(n, \CC) \)
and \( \sigma = \id \).

\subsection{Classification of twists for parabolic subalgebras}
To simplify the notation we again assume \( \sigma = \sigma_{(s; |\nu|)}\). 
The following results can be stated for an arbitrary finite order automorphism by applying conjugation.

Let \( S \subsetneq \Lambda^\sigma \) be a proper
subset of standard generators of \( \Loop^\sigma \).
It is easy to see that
the standard Lie bialgebra structure
\( \delta_0^\sigma \) restricts to both
\( \mathfrak{S}^S \) and
\( \p_{\pm}^S \).
Such induced Lie bialgebra
structures can be defined using modifications of the Manin triple
\eqref{eq: standard Manin triple}.
For example, the Lie bialgebra structure
\( (\p_+^S, \,
\delta_0^\sigma |_{\p^S_+}) \)
is defined by the Manin triple
\begin{equation}
\seq{\seq{\mathfrak{S}^S + \h} \times \Loop^\sigma, \Delta^S, W^S_0},
\end{equation}
where
\( W^S_0 = W_0 \cap ((\mathfrak{S}^S + \h) \times \Loop^\sigma) \)
and
\( \Delta^S = \{ (\pi_S(f), f) \mid f \in \p^S_+ \} \)
for the canonical projection 
\( \pi_S \colon \nobreak \p^S_+ \longrightarrow \nobreak (\mathfrak{S}^S + \h)
=
\p^S_+ / (\p^S_+)^\bot \).
%Let \( D^\sigma \) denote the Dynkin diagram
%of \( \Loop^\sigma \) and \( D^S \) be
%the subgraph of \( D^\sigma \) associated
%to \( \mathfrak{S}^S \subseteq \Loop^\sigma \).
The following theorem gives a classification of
classical twists of the restricted Lie
bialgebra structure
\( \delta_0^\sigma |_{\p_+^S} \)
or, equivalently, classical twists of
\( \delta_0^\sigma \) contained in
\( \p_{+}^S \ot \p_{+}^S  \).

\begin{theorem}[The classification theorem for parabolic subalgebras]\label{lem: classification for parabolics}
For any classical twist
\( t \in \p_{+}^S \ot \, \p_{+}^S \)
of the standard Lie bialgebra structure
\( \delta_0^\sigma \) on \( \Loop^\sigma \)
there exists a regular equivalence \( \phi \)
that restricts to an automorphism of
\( \p_+^S \)
and a BD quadruple \( Q = \seq{\Gamma_1, \Gamma_2, \gamma, t_\h} \) such that
\begin{equation}
\Gamma_1 \subseteq S
\ \text{ and } \
\delta_t^\sigma\phi = (\phi \ot \phi) \delta^\sigma_{Q}.
\end{equation}
Let \( Q' = (\Gamma'_1, \Gamma'_2, \gamma', t'_\h) \),
\( \Gamma'_1 \subseteq S \), be another BD
quadruple.
A regular equivalence between twisted standard structures
\( \delta_Q^\sigma \) 
and \( \delta_{Q'}^\sigma \) restricts to an automorphism of \( \p^S_+ \)
if and only if the induced Dynkin diagram automorphism \( \vartheta \)
preserves \( S \), i.e.
\( \vartheta(S) = S \).
\end{theorem}

\begin{remark}
For a BD quadruple \( Q = \seq{\Gamma_1, \Gamma_2, \gamma, t_\h} \)  with \(\Gamma_1 \subseteq S\), formula \eqref{eq: tQ} directly implies that \(t_Q  \in \p_{+}^S \ot \p_{+}^S \). In particular \(t_Q\) is a twist of \(\delta_0^\sigma |_{\p_+^S}\).
\end{remark}
%To visualize the assumption of the lemma
%we can consider
%\( S = \{ \widetilde{\alpha}_1, \widetilde{\alpha}_2 \} \)
%and
%\( S' = \{ \widetilde{\alpha}_1, \widetilde{\alpha}_3 \} \).
%Then all embeddings of 
%\( D^S = \bullet \!\! \! - \!\!\! - \! \bullet \) into \( A^{(1)}_5 \)
%can be mapped into each other using
%appropriate Dynkin
%diagram automorphisms, but 
%this is not true for the embeddings of 
%\( D^{S'} = \bullet \phantom{-} \bullet \).

%\begin{figure}[h]
%\centering
%\includegraphics[scale=0.4]{Graphs.png}
%\end{figure}

The proof of Theorem \ref{lem: classification for parabolics}
is based on the following three structural results for \( \Loop^\sigma \).

\begin{lemma}\label{lem: structure parabolic}~
\begin{enumerate}
    \item A subalgebra \(\mathfrak{a}\) of \(\Loop^\sigma \) containing a coisotropic subalgebra \(\mathfrak{h}_1\) of \(\h\) satisfies \([\h,\mathfrak{a}] \subseteq \mathfrak{a}\);
    \item A subalgebra \(\mathfrak{p}\) of \(\Loop^\sigma \) containing \(\B_\pm\) is of the form \(\mathfrak{p}_\pm^{S'}\) for some \(S' \subseteq \Pi^\sigma\);
    \item A mapping
    \(\phi \in \textnormal{Aut}_{\CC - \textnormal{LieAlg}} (\Loop^\sigma)\) fixing \(\B_+\) or \(\B_-\) induces an automorphism of the Dynkin diagram of \(\Loop^\sigma\).
\end{enumerate}
\end{lemma}
\begin{proof}
\emph{1.\,:}
We can write
\begin{equation}
\Loop^\sigma
=
\bigoplus_{\alpha' \in \dual{\h_1}} 
\Loop^\sigma_{\alpha'}
=
\bigoplus_{\alpha' \in \dual{\h_1}} 
\bigoplus_{\substack{\alpha \in \dual{\h} \\ 
\alpha |_{\h_1} = \, \alpha'}}
\Loop^\sigma_{\alpha},
\end{equation}
where \(\Loop^\sigma_\alpha = \{ f \in \Loop^\sigma \mid [h, f] = \alpha(h)f \ \forall h \in \h \} = \bigoplus_{k \in \ZZ} \Loop^\sigma_{(\alpha,k)}\) 
for any \(\alpha \in \dual{\h}\) and similarly \(\Loop^\sigma_{\alpha'} = \{ f \in \Loop^\sigma \mid [h, f] = \alpha'(h)f \ \forall h \in \h_1 \}\)
for any \(\alpha' \in \dual{\h_1}\).
Assume there are distinct
\( \alpha_1, \alpha_2 \in \dual{\h}\) such that \(\Loop^\sigma_{\alpha_1},\Loop^\sigma_{\alpha_2} \ne 0\)
and 
\( (\alpha_1 - \alpha_2) |_{\h_1} = 0 \).
Since \( \h_1 \) is coisotropic
inside \( \h \), we have
\( \dual{(\alpha_1 - \alpha_2)} \in \h_1^\bot \subseteq \h_1 \). 
From \cite[Lemma X.5.6]{helgason}
it follows that
\( \alpha_1 - \alpha_2 = 0 \)
which, in its turn, implies that
for any \( \alpha' \in \dual{\h_1}\)
there exists a unique weight
\(\alpha \in \dual{\h} \) such that
\( \Loop^\sigma_{\alpha'} = \Loop^\sigma_{\alpha}\).
This observation combined with \([\h_1,\mathfrak{a}] \subseteq \mathfrak{a}\) allows us to write (see \cite[Proposition 1.5]{kac_book})
\begin{equation}
    \mathfrak{a}
    =
    \bigoplus_{\alpha' \in \dual{\h_1}} \Loop^\sigma_{\alpha'} \cap \mathfrak{a}
    =
    \bigoplus_{\alpha \in \dual{\h}} \Loop^\sigma_{\alpha} \cap \mathfrak{a},
\end{equation}
implying \([\h,\mathfrak{a}] \subseteq \mathfrak{a}\).

\emph{2.\,:}
Without loss of generality assume that \( \p \) contains \( \B_+ \). The inclusion
\(\h \subseteq \mathfrak{p}\) 
and
\cite[Proposition 1.5]{kac_book}
imply that
\begin{equation}
\p = \bigoplus_{\alpha \in \dual{\h}} \Loop^\sigma_{\alpha} \cap \p.
\end{equation}
Take \(X \in \Loop^\sigma_{-\alpha} \cap \mathfrak{p} \cap \N_-\) for some \(\alpha \ne 0\). Let \(j\) be the maximal non-negative integer such that the \(\Loop_{(-\alpha,-j)}\)-component of \(X\) is non-zero.
The structure theory of \( \Loop^\sigma \) implies
\( \dim(\Loop_{(-\alpha,-j)}) = 1 \).
Assume that
\( (-\alpha,-j+k) \) is a root
for some positive integer \( k \).
Decomposing the difference of \( (-\alpha,-j) \) and
\( (-\alpha,-j+k) \) into the sum of simple roots we get a relation of the form
\( \sum_{i=0}^n c_i \widetilde{\alpha}_i = (0, k) \).
Then the identity 
\( \sum_{i=0}^n c_i \alpha_i = 0 \) and \eqref{eq:defa_i} imply that \(k\) is an integer multiple of \(\sum_{i = 0}^n a_is_i\). Using \cite[Theorem 5.6.b)]{kac_book} we see that \( (0,k) \) is a root.
Applying \cite[Lemma X.5.5'.(iii)]{helgason} iteratively we see that
\begin{equation}\label{eq: parabolic is graded in real roots}
    \bigoplus_{k \geq 0 }\Loop_{(-\alpha,-j+k)} \subseteq \mathfrak{p}.
\end{equation}

Following the proof of \cite[Lemma 1.5]{kac_wang} we now show that \(\mathfrak{p} = \mathfrak{p}_+^{S'}\), where
\begin{equation}
S' = \{(\alpha_i,s_i)\in\Pi^\sigma \mid \Loop_{(-\alpha_i,-s_i)} \subseteq \mathfrak{p}\}.
\end{equation}
Assume the claim is false.
Let \((-\gamma,-\ell) \notin \Span_{\ZZ}(S') \) be a negative root of maximal height such that there exists an element \(Y \in
\Loop^\sigma_{-\gamma} \cap \mathfrak{p} \cap \N_- \)
with a non-zero \(\Loop_{(-\gamma,-\ell)}\)-component \(Y_{-\ell}\).
Then there exists \((\alpha_j,s_j)\in \Pi^\sigma \setminus S'\)
such that \([X^+_j, Y_{-\ell}] \ne 0\)
and \((-\gamma+\alpha_j,-\ell + s_j) \in \Span_{\ZZ}(S') \),
where \( X_j^+ \) is the standard
generator of \( \Loop^\sigma \). Note that equation \eqref{eq: parabolic is graded in real roots} implies \(\gamma \ne \alpha_j\).
By the structure theory of loop algebras
we can find
\( Z \in \Loop^\sigma_{(\gamma-\alpha_j,\ell - s_j)} \subseteq \B_+ \subseteq \p\)
such that
\begin{equation}
B([X_j^+, Y_{-\ell}], Z) \neq 0.
\end{equation}
The invariance of the form \( B \) then gives
\( 0 \neq [Y_{-\ell}, Z] \in  \Loop^\sigma_{(-\alpha_j, -s_j)} \).
Applying formula \eqref{eq: parabolic is graded in real roots}
to \( X = [Y,Z] \in \p\) we get
\( (\alpha_j, s_j) \in S' \) contradicting
our choice of \( (\alpha_j, s_j) \).

\emph{3.\,:} Assume that \( \phi(\B_+) = \B_+ \). Since \( \N_+ = [\B_+, \B_+] \), we see that \(\phi\) also fixes \( \N_+ \) and \( \h \).
By \cite[Lemma 1.29]{kac_wang} the automorphism \( \phi \) maps \(\Pi^\sigma\) to a root basis. But since \(\phi\) fixes \(\N_+\), this root basis cosists of positive roots and the only root basis in the set of positive roots is \(\Pi^\sigma\). Hence \(\phi(\Pi^\sigma) = \Pi^\sigma\) and thus \(\phi\)
induces an automorphism \( \vartheta \) of the Dynkin diagram of \( \Loop^\sigma \).
\end{proof}

\begin{proof}[Proof of Theorem \ref{lem: classification for parabolics}]
We prove the statement for \( \sigma = \sigma_{(s;|\nu|)} \). The general result it obtained using conjugation.
By Theorem \ref{thm: classification}
there is a regular equivalence 
\( \phi_1 \) on \(\Loop^\sigma \) and
a BD quadruple 
\( Q' = (\Gamma'_1, \Gamma'_2, \gamma', t'_\h) \)
such that
\( (\phi_1 \times \phi_1) W_t = W_{Q'} \).
Since \( t \in \p_{+}^S \ot \, \p_{+}^S  \) we have
\( W_t \subseteq \p^S_+ \times \Loop^\sigma \).
Let \( \h_1 \subseteq \h \)
be the image of \( \psi(t_\h) - \id_\h /2 \). Since \(t_\h\) is skew-symmetric, this is easily seen to be a coisotropic subspace of \(\h\).
Then
\begin{equation}\label{eq: C1 subset phi(p)}
C^1_{Q} = \N_+ \add \h_1 \add \N^{\Gamma_1}_- 
\subseteq \phi_1(\p^{S}_+)
\end{equation}
and, in particular,
we have the inclusion
\( \h_1
\subseteq  \phi_1(\p^S_+) \).
By the first part of Lemma \ref{lem: structure parabolic}
we have
\begin{equation}\label{eq: inclusion h}
[\h, \phi_1(\p^S_+)] \subseteq \phi_1(\p^S_+).
\end{equation}
Since \(\p^S_+\) is self-normalizing, \( \phi_1(\p^S_+) \) is self-normalizing as well. Therefore we get \( \h \subseteq \phi_1(\p^S_+) \) and consequently
\( \B_+ \subseteq \phi_1(\p^S_+) \).
Then the second statement of Lemma
\ref{lem: structure parabolic}
shows that 
\( \phi_1(\p^S_+) = \p^{S'}_+ \)
for some \( S' \subsetneq \Pi^\sigma \).
The inclusion
\ref{eq: C1 subset phi(p)} implies that
\( \Gamma_1 \subseteq S' \).

Define \( \B' \coloneqq \phi_1^{-1} (\B_+) \).
The subalgebra \( \B' / \p^{S, \bot}_+ \), being the preimage of the Borel subalgebra \( \B_+/\p^{S', \bot}_+ \) of \( \mathfrak{S}^{S'} + \h = \p^{S'}_+/\p^{S', \bot}_+\) under \( \phi_1 \), is a Borel subalgebra of
\( \mathfrak{S}^S + \h  = \p^{S}_+/\p^{S, \bot}_+\). Therefore,
by the conjugacy theorem for Borel subalgebras, there exists an inner automorphism \( \phi_2 \) of \( \mathfrak{S}^S + \h \) mapping \( \B' / \p^{S, \bot}_+ \) to \( \B_+ / \p^{S, \bot}_+ \).
It can be seen from \cite[Lemma X.5.5]{helgason} that \(\textnormal{ad}_x\) is nilpotent on \(\Loop^\sigma\) for any \(x\in \Loop_{(\alpha,k)}^\sigma\) and \(\alpha \ne 0\).
Combining this result with the equality
\begin{align*}
\textnormal{Inn}_{\CC - \textnormal{LieAlg}}(\mathfrak{S}^S + \h) 
&=
\langle
e^{\ad_x} \mid 
x \in \Loop^\sigma_{(\alpha, k)},
(\alpha, k) \in \Phi_\sigma \cap \textnormal{span}_{\ZZ}(S),
\alpha \neq 0
\rangle 
\end{align*}
(see \cite[§3.2]{bourbaki_lieIII}),
we can view \( \phi_2 \) as a regular equivalence
on \( \Loop^\sigma \) that restricts
to an automorphism of
\( \p^S_+ \) and maps
\( \B' \) to \( \B_+ \).
The composition \( \phi^{~}_2 \phi_1^{-1 } \) is then an automorphism of \( \Loop^\sigma \) mapping \( \p^{S'}_+ \) to \( \p^S_+ \) and fixing the Borel subalgebra \( \B_+ \).
The third part of Lemma
\ref{lem: structure parabolic}
implies that \( \phi^{~}_2 \phi_1^{-1 } \)
induces an automorphism \( \vartheta \) of the Dynkin diagram of \( \Loop^\sigma \) such that
\( \vartheta(S') = S \).
Applying the second part of Theorem \ref{thm: classification} to \( \vartheta \) we obtain a regular equivalence \( \phi_3 \) such that 
\( (\phi_3 \times \phi_3)W_{Q'} = W_{Q := \vartheta(Q')} \).
The composition \( \phi := \phi_3 \phi_1 \) and the
quadruple \( Q \) satisfy all the requirements of the
theorem.
\end{proof}

\subsection{Quasi-trigonometric solutions of CYBE}
Letting \( \sigma = \id \) and
\( S = \Pi \setminus \{ (\alpha_0, 1) \} \)
the corresponding parabolic subalgebra
\( \p^S_+\) becomes \( \g[z] \).
The solutions to CYBE of the form \( r_t = r_0 + t \),
where \( t \in \g[z]^{\ot2} \),
are called \emph{quasi-trigonometric}.
Two quasi-trigonometric solutions \( r_t \) and
\( r_s \) are called \emph{polynomially equivalent} if
there exists a \( \phi \in \textnormal{Aut}_{\CC[z] - \textnormal{LieAlg}}(\g[z]) \)
such that 
\begin{equation}
    r_s(x,y) = (\phi(x) \ot \phi(y))r_t(x,y)
    \qquad
    \forall x, y \in \CC^*, \ x \neq y.
\end{equation}
Therefore, a polynomial equivalence is a regular equivalence that restricts to an automorphism of \(\g[z]\).
Quasi-trigonometric \rmatrs
were introduced and classified up to polynomial equivalence and choice of a maximal order in \cite{KPSST, pop_stolin}.
More precisely, it was shown that
quasi-trigonometric solutions are in one-to-one correspondence with certain Lagrangian subalgebras of
\( \g \times \g(\!(z^{-1})\!) \). 
Embedding the Lagrangian subalgebra, corresponding to
a quasi-trigonometric solution \( r \),
into some maximal order of
\( \g \times \g(\!(z^{-1})\!) \), the
authors of \cite{KPSST, pop_stolin} obtained a unique quasi-trigonometric solution \( r_Q \) (given by a
BD quadruple \( Q \))
polynomially equivalent to \( r \).
In this setting we get the following results:
\begin{itemize}
\item The classification theorem for parabolic subalgebras
\ref{lem: classification for parabolics} together with
Theorem \ref{thm: delta = [, r]} gives a new
proof of the above-mentioned classification of quasi-trigonometric \rmatrs;

\item In general, a maximal order in which one can embed the Lagrangian
subalgebra correspoding to a quasi-trigonometric
solution \( r \) is not unique. Choosing two
different maximal orders we get
two different BD quadruples \( Q \) and \( Q' \) and
two polynomially equivalent quasi-trigonometric
\rmatrs \( r_Q \) and \( r_{Q'} \).
By Theorem \ref{lem: classification for parabolics} this
equivalence induces an automorphism \( \vartheta \)
of the Dynking diagram of \( \Loop \) 
that fixes the minimal root, i.e.
\( \vartheta(\widetilde{\alpha}_0) = \widetilde{\alpha}_0 \). Therefore, any quasi-trigonometric
solution is polynomially equivalent to exactly one
quasi-trigonometric \rmatr \( r_Q \), for
some BD quadruple \( Q \), if and only if
\( \g \) is of type \( A_1, B_n, C_n, F_4, G_2\) or \( E_8 \).

We note that there exist regularly equivalent quasi-trigonometric solutions 
\( r_Q \) and \( r_{Q'} \) which are not polynomially equivalent (see Figure \ref{fig: regular_butnot_poly}). Therefore regular equivalence
is strictly weaker than polynomial one;

\item It was shown in
\cite{KPSST}
that for any quasi-trigonometric \rmatr \( r \)
there exists a holomorphic
function \( \phi \colon \CC \longrightarrow \textnormal{Inn}_{\CC - \textnormal{LieAlg}}(\g)\) such that 
\begin{align*}
(\phi(x)^{-1} \ot \phi(y)^{-1})r(x,y) = X(x/y),
\end{align*}
where \( X \) is a trigonometric solution
in the Belavin-Drinfeld classification \cite{belavin_drinfeld_solutions_of_CYBE_paper}.
Combining Lemma \ref{thm: GxG r = r} and
Theorem \ref{lem: classification for parabolics}
we get a general version of this statement with more control over the holomorphic equivalence.
Precisely,
the trigonometric \rmatr
\( r^{\sigma_{(\mathbf{1}; |\nu|)}}_Q \),
by definition,
always depends on the quotient of its parameters; in order to obtain it from a \( \sigma \)-trigonometric
\rmatr \( r^\sigma_t \), where the coset of \(\sigma\) is conjugate to \( \nu \textnormal{Inn}_{\CC - \textnormal{LieAlg}}(\g)  \),
it is enough to apply
a regular equivalence composed with the regrading
to the principal grading: %i.e. grading corresponding to the Coxeter automorphism
%\( \sigma_{(\mathbf{1}; |\nu|)} \);
\begin{equation*}
      r^\sigma_t(x,y) \xmapsto{\text{regular eq.}} r^\sigma_Q(x,y) \xmapsto{\text{regrading}} r_Q^{\sigma_{(\mathbf{1};|\nu|)}}(x,y) = X(x/y);
  \end{equation*}

\item Conjecture 1 in \cite{burban_galinat_stolin} is justified: Combining \eqref{eq: r 0 sigma} with \eqref{eq: tQ}
we get the explicit formula for a quasi-trigonmetric
solution \( r_Q \) given by a BD quadruple \(Q\),
namely
\begin{equation}\label{eq: quasi-trig explicit formula}
\begin{split}
    r_Q(x,y) &= \frac{yC}{x-y} +
    \frac{C_\h}{2} +
    C_- +
    t_\h +
\sum_{\widetilde{\alpha} \in \Phi^+_{1}}
\sum_{j = 1}^{\infty} b_{-\widetilde{\alpha}} \wedge
\theta^j_\gamma \seq{b_{\widetilde{\alpha}}} \\
%&= - \frac{xC}{y-x} -
%    \frac{C_\h}{2} -
%    C_+ -
%    \tau(t_\h) -
%\sum_{\widetilde{\alpha} \in \Phi^+_{1}}
%\sum_{j = 1}^{\infty} \theta^j_\gamma %\seq{b_{\widetilde{\alpha}}} \wedge
%b_{-\widetilde{\alpha}} \\
&= - \frac{1}{2} \seq{\frac{y+x}{y-x}C +
    \sum_{\widetilde{\alpha} \in \Phi^+} b_{\widetilde{\alpha}} \wedge b_{-\widetilde{\alpha}} -
    t_\h +
\sum_{\widetilde{\alpha} \in \Phi^+_{1}}
\sum_{j = 1}^{\infty} \theta^j_\gamma \seq{b_{\widetilde{\alpha}}} \wedge
b_{-\widetilde{\alpha}}}.
\end{split}
\end{equation}
%where \( \tau(a \ot b) \coloneqq b \ot a \).
This formula (up to a sign) coincides with the one
conjectured by Burban, Galinat and Stolin in
\cite{burban_galinat_stolin};

\item Question 2 in \cite{burban_galinat_stolin} is answered:
Let \( Q \) be a BD quadruple and 
\( h \coloneqq |\sigma_{(\mathbf{1}; 1)}| \).
The relation between the quasi-trigonmetric
solution \eqref{eq: quasi-trig explicit formula} and the trigonometric solution
\begin{equation}
\begin{split}
  X(u-v) = t_\h + \frac{C_\h}{2} + 
  \frac{1}{e^{u-v} - 1}
  \sum_{k=0}^{h - 1}
  e^{\frac{k(u-v)}{h}}C^{\sigma_{(\mathbf{1}; 1)}}_k 
  &+ \sum_{\substack{\widetilde{\alpha} \in \Phi_1^+ \\ \widetilde{\alpha}=(\alpha,k)}} \sum_{j = 1}^{\infty}
  e^{\frac{k(u-v)}{h}} \theta^j_\gamma \seq{b_{\widetilde{\alpha}}}(1) \ot
  b_{-\widetilde{\alpha}}(1)
 \\
&- \sum_{\substack{\widetilde{\alpha} \in \Phi_1^+ \\ \widetilde{\alpha}=(\alpha,k)}} \sum_{j = 1}^{\infty}
e^{\frac{k(v-u)}{h}}
  b_{-\widetilde{\alpha}}(1) \ot
\theta^j_\gamma \seq{b_{\widetilde{\alpha}}}(1),
\end{split}
\end{equation}
given by the same quadruple Q (see \cite{belavin_drinfeld_solutions_of_CYBE_paper}), is described by regrading
from \( \id \) to the Coxeter automorphism \( \sigma_{(\mathbf{1}; 1)} \) using Lemma \ref{thm: GxG r = r}. More precisely,
\begin{equation}
  \seq{e^{u \, \ad(\mu)} \ot e^{v \, \ad(\mu)} }
  r_Q \seq{e^{u}, e^{v}} = r^{\sigma_{(\mathbf{1}; 1)}}_Q(e^{u/h}, e^{v/h}) =
  X(u-v).
\end{equation}
\end{itemize}

\begin{figure}[h]
    \centering
    \includegraphics[scale = 0.12]{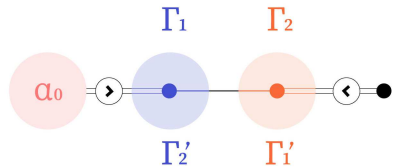}
    \caption{\( \Gamma_1, \Gamma_1', \Gamma_2  \) and 
    \( \Gamma_2' \) leading to regularly but not polynomially equivalent \rmatrs \( r_Q \) and \( r_{Q'} \).}
    \label{fig: regular_butnot_poly}
\end{figure}

\subsection{Special case \( \g = \mathfrak{sl}(n,\CC) \) and \( \sigma = \id \)}
\label{sec: sl(n,C) case}
The classification of classical twists
of the standard Lie bialgebra structure
\( \delta_0 := \delta^{\id}_0 \) on
\( \Loop = \g[z,z^{-1}] \) with
\( \g = \mathfrak{sl}(n, \CC) \)
can be done without heavy geometric
machinery. More precisely, using the theory
of maximal orders developed in
\cite{stolin_sln},
we can show that the equivalence classes
of twisted standard bialgebra structures on \( \Loop \) are in one-to-one
correspondence with the equivalence classes of quasi-trigonometric
\rmatrs, which were classified in
\cite{KPSST}
in terms of BD quadruples.
 
The following lemma explains the way in which orders
emerge in our work.

\begin{lemma}\label{lem: trigonom solutions <-> Lagrangian sub}
  Classical twists \( t \) of the standard
  Lie bialgebra structure \( \delta_0 \)
  are in one-to-one correspondence
  with Lagrangian Lie subalgebras 
  \( \widehat{W}_t \subseteq \g(\!(z)\!) \times \g(\!(z^{-1})\!) \)
  satisfying the conditions:
  \begin{enumerate}
    \item \( \Delta \add \widehat{W}_t = \g(\!(z)\!) \times \g(\!(z^{-1})\!) \);
    \item There are non-negative integers \( N \)
    and \( M \) such that
    \begin{align*}
      z^N \g[\![z]\!]
      &\subseteq \pi_{1} \widehat{W}_t \subseteq 
      z^{-N}\g[\![z]\!], \\
      z^{-M}\g[\![z^{-1}]\!]
      &\subseteq \pi_{2} \widehat{W}_t \subseteq 
      z^M \g[\![z^{-1}]\!],
    \end{align*}
    where \( \pi_{1} \) and \( \pi_{2} \)
    are the projections of
    \( \g(\!(z)\!) \times \g(\!(z^{-1})\!) \) onto its components
    \( \g(\!(z)\!) \) and \( \g(\!(z^{-1})\!) \) respectively.
  \end{enumerate}
\end{lemma}

\begin{proof}
By Remark
\ref{rem: large double}
the standard Lie bialgebra structure
\( \delta_0 \) on \( \Loop \)
is defined by the Manin triple
\begin{equation}
(\g(\!(z)\!) \times  \g(\!(z^{-1})\!),\Delta,
\widehat{W}_0).
\end{equation}
Therefore, in view of
Theorem \ref{thm: W <-> t <-> T correspondence} and its proof,
it is enough to show that condition \emph{2.} corresponds to
the commensurability condition on
\( \widehat{W}_t \) and
\( \widehat{W}_\textnormal{0} \),
or equivalently,
to finite dimensionality of the image of the map
\( T = \psi(t) \colon \widehat{W}_\textnormal{0} 
\longrightarrow \Delta \).
The latter correspondence is justified by the following chain of arguments:
the condition \( \dim(\im(T)) < \infty \)
is equivalent to the inclusion
\begin{equation}\label{eq: im(T) subset Delta}
  \im(T)
  \subseteq
  \set{\seq{\sum_{k = -N}^{M} a_k z^k, \sum_{k = -N}^{M} a_k z^k} \bigg| \,
  a_k \in \g},
\end{equation}
for some non-negative integers \( N \) and \( M \),
which in its turn
is equivalent to
\begin{equation}\label{eq: projections < t^-N g[[t]]}
\begin{split}
  \pi_1 \widehat{W}_t 
  &\subseteq 
  z^{-N}\g[\![z]\!], \\
  \pi_{2} \widehat{W}_t
  &\subseteq 
  z^M \g[\![z^{-1}]\!];
\end{split}
\end{equation}
Since \( \widehat{W}_t \) is Lagrangian, inclusions
\eqref{eq: projections < t^-N g[[t]]} are equivalent to
condition \emph{2.} of the theorem. 
\end{proof}

A subalgebra
\( W \subseteq \g(\!(u)\!) \)
is called an \emph{order} if
there is a non-negative integer \( N \) such that
\begin{equation}
u^N \g[\![u]\!] \subseteq W \subseteq u^{-N}\g[\![u]\!].
\end{equation}
Therefore, condition \emph{2.} of Lemma
\ref{lem: trigonom solutions <-> Lagrangian sub}
means that the projections \( \pi_{1}\widehat{W}_t \)
and \( \pi_2 \widehat{W}_t \) are orders.

The following two results from
\cite{stolin_sln}
play the key role in the classification
of classical twists of
\( \delta_0 \).

\begin{theorem}\label{thm: orders in sl(n,C((t^-1)))}
For any order \( W \)
in \( \mathfrak{sl}(n, \CC(\!(u^{-1})\!)) \)
there is a matrix
\( A \in \textnormal{GL}(n, \CC(\!(u^{-1})\!)) \)
such that
\( W \subseteq A^{-1} \mathfrak{sl}(n, \CC[\![u^{-1}]\!]) A \).
In particular, any maximal order must be of the form
\( A^{-1} \mathfrak{sl}(n, \CC[\![u^{-1}]\!]) A \)
for some \( A \in \textnormal{GL}(n, \CC(\!(u^{-1})\!)) \).
\end{theorem}

\begin{lemma}[Sauvage Lemma]\label{lem: Sauvage Lemma}
  The diagonal matrices
  \( \diag(u^{m_1}, \ldots, u^{m_n}) \), 
  where \( m_k \in \ZZ \) and 
  \( m_1 \leq  \ldots \leq m_n \),
  represent all double cosets in
  \( 
  \textnormal{GL}(n, \CC[\![u^{-1}]\!])
  \backslash
  \textnormal{GL}(n, \CC(\!(u^{-1})\!))
  /
  \textnormal{GL}(n, \CC[u])
  \).
\end{lemma}

Let
\( W \subseteq \g(\!(z)\!) \times \g(\!(z^{-1})\!) \) be
a Lagrangian subalgebra satisfying the conditions of
Lemma \ref{lem: trigonom solutions <-> Lagrangian sub}.
Since the projections
\( \pi_1 W \subseteq \g(\!(z)\!) \) and
\( \pi_2 W \subseteq \g(\!(z^{-1})\!) \) are orders,
by Theorem
\ref{thm: orders in sl(n,C((t^-1)))} there are
matrices
\( A_\pm \in GL(n, \CC(\!(z^{\pm 1})\!)) \)
such that
\begin{equation}
\begin{split}
  \pi_1 W &\subseteq A_+^{-1} \mathfrak{sl}(n, \CC[\![z]\!]) A_+, \\
  \pi_2 W &\subseteq A_-^{-1} \mathfrak{sl}(n, \CC[\![z^{-1}]\!]) A_-.
\end{split}
\end{equation}
By Sauvage Lemma
\ref{lem: Sauvage Lemma}
we can find matrices
\begin{equation}
\begin{split}
    P_\pm
    &\in
    \textnormal{GL}(n,\CC[\![z^{\pm 1}]\!]), \\
    d_\pm
    &\in
    \textnormal{GL}(n,\CC[z, z^{-1}] ), \\
    Q_\pm
    &\in
    \textnormal{GL}(n,\CC[z^{\mp 1}]),
\end{split}
\end{equation}
where \( d_\pm \) are diagonal, such that
\begin{equation}
\begin{split}
  \pi_1 W
  &\subseteq
  Q_+^{-1} d_+^{-1} P_+^{-1} \mathfrak{sl}(n, \CC[\![z]\!])
  P_+ d_+ Q_+, \\
  \pi_2 W
  &\subseteq
  Q_-^{-1} d_-^{-1} P_-^{-1} \mathfrak{sl}(n, \CC[\![z^{-1}]\!])
  P_- d_- Q_-.
\end{split}
\end{equation}
Taking the product and using the fact that
\( P_\pm^{-1} \mathfrak{sl}(n, \CC[\![z^{\pm 1}]\!])
  P_\pm = \mathfrak{sl}(n, \CC[\![z^{\pm 1}]\!]) \) we obtain
the inclusion
\begin{equation}
W
\subseteq
\seq{Q_+^{-1} d_+^{-1} \mathfrak{sl}(n, \CC[\![z]\!]) d_+ Q_+}
\times
\seq{Q_-^{-1} d_-^{-1} \mathfrak{sl}(n, \CC[\![z^{-1}]\!]) d_- Q_-}.
\end{equation}
Note that the componentwise conjugation by \( Q_+ \) or
\( d_+ \) is a regular equivalence.
Applying these conjugations we get
\begin{equation}\label{eq: dQWQd}
\begin{split}
\widetilde{W} := d_+ Q_+ W Q_+^{-1} d_+^{-1}
&\subseteq
\mathfrak{sl}(n, \CC[\![z]\!])
\times
\seq{d_+ Q_+ Q_-^{-1} d_-^{-1} \mathfrak{sl}(n, \CC[\![z^{-1}]\!]) d_- Q_- Q_+^{-1} d_+^{-1}}
\\&\subseteq \mathfrak{sl}(n, \CC[\![z]\!])
\times
 \mathfrak{sl}(n, \CC(\!(z^{-1})\!))
\end{split}
\end{equation}

%%%%%%%%%%%%%%%%%%%%%%%%%%%%%%%%%%%%%%%%%%%%%%%%%%
%%%%%%%%%%%%%%%%%%%%%%%%%%%%%%%%%%%%%%%%%%%%%%%%%%

\iffalse
Since the product 
\( d_- Q_- Q_+^{-1} d_+^{-1} \) lies inside
\( \textnormal{GL}(n,\CC[z, z^{-1}])\),
we can use Sauvage Lemma \ref{lem: Sauvage Lemma} again
and write
\( d_- Q_- Q_+^{-1} d_+^{-1} = P d Q \) for
matrices
\( P \in \textnormal{GL}(n,\CC[\![z^{-1}]\!])\), 
\( Q \in \textnormal{GL}(n,\CC[z]) \) and
a diagonal matrix
\( d \in \textnormal{GL}(n,\CC[z,z^{-1}]) \).
Conjugating the relation
\eqref{eq: dQWQd}
with \( Q \) and absorbing
the actions of \( P \) and \( Q \) on the second and first components respectively we conclude that
\begin{equation}
\widetilde{W}
:=
Q d_+ Q_+ W Q_+^{-1} d_+^{-1} Q^{-1}
\subseteq
\mathfrak{sl}(n, \CC[\![z]\!])
\times
\seq{d^{-1} \mathfrak{sl}(n, \CC[\![z^{-1}]\!]) d}.
\end{equation}
\fi

%%%%%%%%%%%%%%%%%%%%%%%%%%%%%%%%%%%%%%%%%%%%%%%%%%%%%%%
%%%%%%%%%%%%%%%%%%%%%%%%%%%%%%%%%%%%%%%%%%%%%%%%%%%%%%

By Theorem \ref{thm: delta = [, r]}
the classification problem of
classical twists of the standard Lie bialgebra
structure \( \delta_0 \) is equivalent to the
classification problem of \( \id \)-trigonometric
\rmatrs.
The following lemma reduces the question even further
to quasi-trigonometric \rmatrs.

\begin{lemma}\label{lem: quasi-trigonom <-> trigonom}
Any \( \id \)-trigonometric \rmatr is regularly equivalent to
a quasi-trigonometric one.
\end{lemma}

\begin{proof}
Let \( r_t = r_0 + t \)
be an \( \id \)-trigonometric \rmatr, where \( t \) is a classical
twist of \( \delta_0 \),
and \( W_t \) be the corresponding Lagrangian
subalgebra of \( \g(\!(z)\!) \times \g(\!(z^{-1})\!) \).
By the argument preceding the lemma there is a regular equivalence
\( \phi \in \textnormal{Aut}_{\CC[z,z^{-1}]-\textnormal{LieAlg}}(\g[z,z^{-1}]) \)
such that
\begin{equation}\label{eq: Ws in sl times d sl d}
  W_s \coloneqq (\phi \times \phi) W_t
  \subseteq
  \mathfrak{sl}(n, \CC[\![z]\!])
  \times
  \mathfrak{sl}(n, \CC(\!(z^{-1})\!)).
\end{equation}
for some classical twist \( s \) of \( \delta_0 \).
We now show that \( r_s \) is quasi-trigonometric,
or equivalently, that
\( s \in \g[z]^{\ot2} \cong \g^{\ot2}[x,y]\).
Let \( \set{I_\alpha}_{\alpha=1}^n \)
be an orthonormal basis
for \( \mathfrak{sl}(n,\CC) \).
Then we can write
\begin{equation}
  s
  =
  s^{\alpha \beta}_{ij} I_\alpha x^i \ot I_\beta y^j.
\end{equation}
Assume that \( s^{\alpha' \beta'}_{k \ell} \neq 0 \)
for some \( \alpha', \beta' \in \set{1, \ldots, n} \)
and \( k, \ell \in \ZZ \)
such that at least one of the indices 
\( k \) or \( \ell \) is strictly negative,
i.e. the tensor \( s \) contains a negative power
of \( z \) in one of its components.
Since \( s \) is skew-symmetric we may assume without
loss of generality that \( k < 0 \).
Then 
\begin{equation}
  \pi_1 \seq{s^{\alpha \beta}_{ij}
  B\seq{
    (I_\beta z^j,I_\beta z^j),
    (I_{\beta'} z^{-\ell}, 0)
  }
  (I_\alpha z^i, I_\alpha z^i)}
  =
  s^{\alpha \beta'}_{i \ell}
  I_{\alpha}z^i
  \end{equation}
where the sum in the right-hand side contains
\( z^{k} \), \( k < 0 \).
However, by \eqref{eq: Ws in sl times d sl d}
the projection \( \pi_1(W_s) \) is contained in
\( \mathfrak{sl}(n, \CC[\![z]\!]) \) and
hence cannot contain negative powers of \( z \).
This contradiction shows that both components of \( s \) 
are polynomials in \( z \).
\end{proof}

Quasi-trigonometric \rmatrs over \(\mathfrak{sl}(n,\CC)\) were classified (up to regular equivalence) in \cite{KPSST}
using BD qudruples we
introduced at the beginning of this section.
%(the statement for an arbitrary simple Lie algebra is given in \cite{pop_stolin})
%is given by a BD quadruple we
%introduced at the beginning of this section.
One can show that if we lift
the Lagrangian subalgebra
\( W \subseteq \g \times \g(\!(z^{-1})\!) \),
constructed from a BD quadruple \( Q \) in 
\cite{KPSST}, to
\( \g(\!(z)\!) \times \g(\!(z^{-1})\!) \)
we get precisely 
the Lagrangian subalgebra
\( \widehat{W}_Q \) 
determined by the relation
\begin{equation}
\seq{ \Loop \times \Loop}
\cap
\widehat{W}_Q = W_Q,
\end{equation}
where \( W_Q \)
is given by \eqref{eq: W_Q}.
\iffalse
One can show that the Lagrangian subalgebra
\( W \subseteq \g \times \g(\!(z^{-1})\!) \),
constructed from the BD quadruple \( Q \) in 
\cite{KPSST},
lifted to
\( \g(\!(z)\!) \times \g(\!(z^{-1})\!) \)
is precisely the Lagrangian subalgebra
\( \widehat{W}_Q \) 
determined by the relation
\begin{equation}
\seq{ \Loop \times \Loop}
\cap
\widehat{W}_Q = W_Q,
\end{equation}
where \( W_Q \)
is given by \eqref{eq: W_Q}.
\fi
By Lemma 
\ref{lem: trigonom solutions <-> Lagrangian sub}
the Lagrangian subalgebra \( \widehat{W}_Q \)
uniquely determines the classical twist
\( t_Q \).
This gives the classification of classical twists and proves the first part of Theorem \nolinebreak \ref{thm: classification} in the special case \( \g = \mathfrak{sl}(n, \CC) \).

\begin{remark}
The statement of Lemma \ref{lem: quasi-trigonom <-> trigonom} is not surprising. Its general version can be deduced from Theorems \ref{thm: classification} and \ref{lem: classification for parabolics}.
Precisely, for any finite-dimensional simple Lie algebra \( \g \) an \( \id \)-trigonometric solution
\( r_Q = r_0 + t_Q \), given by a BD qudruple \( Q=(\Gamma_1, \Gamma_2, \gamma, t_\h ) \), is regularly equivalent
to a quasi-trigonometric one if and only if there is an automorphism \( \vartheta \) of the Dynkin diagram of \( \Loop \) such that
\( \widetilde{\alpha}_0 \not\in \vartheta(\Gamma_1) \).
%By definition of a BD quadruple, \( \Gamma_1 \) is a proper subset of the simple root system \( \Pi \).
%Let \( \widetilde{\alpha}_i \) be a simple root not included in \( \Gamma_1 \).
It is easy to check that this condition is always satisfied for Dynkin diagrams
of types \( A^{(1)}_n, C^{(1)}_n, B^{(1)}_{2-4} \) and \( D^{(1)}_{4-10} \).
Therefore, in these cases any \( \id \)-trigonmetric solution is regularly equivalent to a quasi-trigonometric one.
In other cases it is always possible to find a BD quadruple \( Q \),
such that \( r_Q \) is not equivalent to a quasi-trigonometric \rmatr (see Figure \ref{fig: dynkin diag}). 
%The Dynkin diagram of type \( A^{(1)}_n, n \geq 2, \) can always be rotated in such a way, that a simple root that is not included in \( \Gamma_1 \) becomes \( \alpha_0 \).
\end{remark}
\begin{figure}[H]
    \centering
    \includegraphics[scale=0.26]{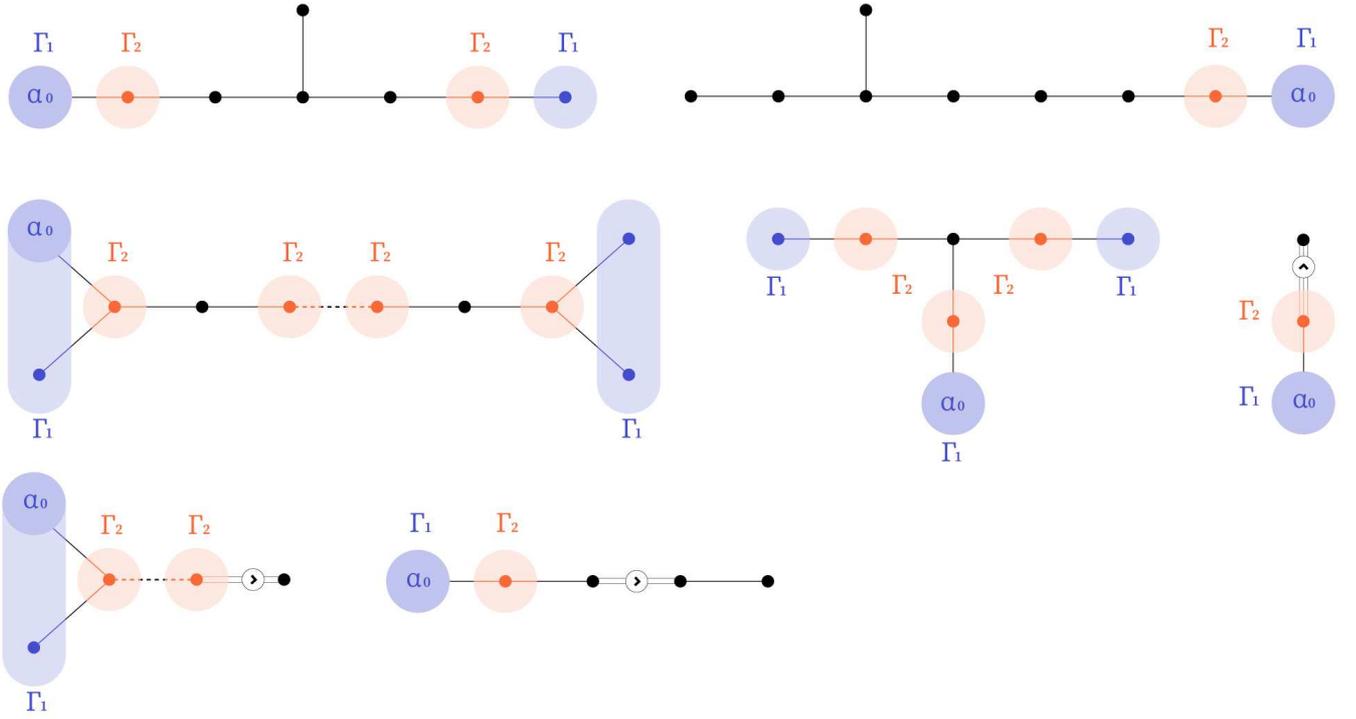}
    \caption{Examples of \( \Gamma_1 \) and \( \Gamma_2 \) giving rise to \( \id \)-trigonometric solutions not equivalent to quasi-trigonometric ones. The dashed lines mean any number \( m \geq 1 \)
    of vertices.}
    \label{fig: dynkin diag}
\end{figure}

\section{Algebro-geometric proof of the main classification theorem}\label{sec: Proof}
In this section we 
give a brief summary of the results in 
\cite{burban_galinat},
prove the extension property for formal equivalences between geometric \rmatrs
(see Theorem \ref{thm: lifting equivalences}) and, finally, combining this property with the results in
\cite{TBA} on geometrization of \( \sigma \)-trigonometric \rmatrs
we verify
Theorem \ref{thm: classification}.

\subsection{Survey on the geometric theory of the CYBE}\label{sec: survey on geometry}
Let \(E\) be an irreducible projective curve of arithmetic genus 1. Then \(E\) is a \emph{Weierstraß cubic}, i.e. there are parameters \(g_2,g_3 \in \CC\) such that \(E\) is the projective closure of \(E_\circ = \textnormal{V}(y^2-4x^3+g_2x+g_3)\subseteq \mathbb{P}^2_{(w:x:y)} \) by a smooth point \(p\) at infinity. \(E\) is singular if and only if \(g_2^3=27g_3^2\) and an elliptic curve otherwise. In the singular case it has a unique singular point \(s\), which is a simple cusp if \(g_2 = 0 = g_3\) and a simple node otherwise. Let \(\Breve{E}\) be the set of smooth points of \(E\). Fix a non-zero section \(\omega \in \Gamma(E,\Omega_E) \cong \CC\), where \(\Omega_E\) is the dualising sheaf. We view \(\omega\) as a global regular 1-form in the Rosenlicht sense (see e.g. \cite[Section II.6]{barth_et_al}).

We consider now a coherent sheaf \(\mathcal{A}\) of Lie algebras on \(E\) such that
\begin{itemize}
    \item[(i)] \(\textnormal{H}^0(E,\mathcal{A}) = 0 = \textnormal{H}^1(E,\mathcal{A})\) and
    \item[(ii)] \(\Breve{\mathcal{A}} = \mathcal{A}|_{\Breve{E}}\) is \emph{weakly \(\g\)-locally free}, i.e. \(\mathcal{A}|_p \cong \g\) as Lie algebras for all \(p \in \Breve{E}\).
\end{itemize}
Property (i) gives that \(\mathcal{A}\) is torsion free and property (ii) ensures that the rational envelope \(A_K\) of the sheaf \(\mathcal{A}\) is a simple Lie algebra over the field \(K\) of rational functions on \(E\). Together these properties give the existence of a distinguished section, called \emph{geometric \(r\)-matrix}, \(\rho \in \Gamma(\Breve{E} \times \Breve{E},\Breve{\mathcal{A}} \boxtimes \Breve{\mathcal{A}}(D))\), where \(D = \{(x,x) \in \Breve{E} \times \Breve{E}\colon x \in \Breve{E}\}\) is the \emph{diagonal divisor}. This section satisfies a geometric version of a generalised CYBE, although, if \(E\) is singular, it lacks skew-symmetry in general, which prevents it to solve the CYBE. Thus we demand one more property of \(\mathcal{A}\) in this case, which ensures skew-symmetry.

If \(E\) is singular with a singularity \(s\), we can consider the invariant non-degenerate \(\CC\)-bilinear form
\begin{align}
    B^\omega_s \colon A_K \times A_K \longrightarrow K \stackrel{\textnormal{res}_s^\omega}{\longrightarrow} \CC,
\end{align}
where the first map is the Killing form of \(A_K\) over \(K\) and \(\textnormal{res}_s^\omega(f) = \textnormal{res}_s(f\omega)\) is the residue taken in the Rosenlicht sense. 
\begin{itemize}\setcounter{enumi}{2}
    \item[(iii)] \(\mathcal{A}_s \subset A_K\) is isotropic, i.e. \(B^\omega_s(\mathcal{A}_s,\mathcal{A}_s) = 0\).
\end{itemize}
Now the main statement of the geometric approach to the CYBE is the following.
\begin{theorem}[{\cite[Theorem 4.3]{burban_galinat}}]
The geometric \(r\)-matrix \(\rho\) is a non-degenerate and skew symmetric (both meant in an appropriate geometric manner) solution of a geometric version of the CYBE.\qed
\end{theorem}

We want to describe \(\rho\) as a series, which can be thought of as a Taylor expansion in the second coordinate at the smooth point \(p\) at infinity. To do so, let us switch from the sheaf theoretic setting to one localised at the formal neighbourhood of \(p\).
There is a unique element \(u\) inside the \(\mathfrak{m}_p\)-adic completion \(\widehat{O}_p\) of the local ring \((\mathcal{O}_{E,p},\mathfrak{m}_p)\), such that \(u(p) = 0\) and \(\widehat{\omega}_p = du\). We can identify \(\widehat{O}_p\) with \(\CC[\![u]\!]\). Thus the field of fractions \(\widehat{Q}_p\) can be identified with \(\CC(\!(u)\!)\). Consequently, we may view \(O = \Gamma(E_\circ,\mathcal{O}_E)\) as a subalgebra of \(\widehat{Q}_p =\CC(\!(u)\!)\). 

Since \( \g \) is simple, Whitehead's lemma implies that \(\textnormal{H}^2(\g,\g) = 0\) and hence
%By simplicity of \(\g\), we have that \(\textnormal{H}^2(\g,\g) = 0\) by Whiteheads Lemma and hence 
all formal deformations of \(\g\) are trivial (see e.g. \cite[Section A.8]{hazewinkel_gerstenhaber}). 
Thus \(\widehat{\mathcal{A}}_p\), which can be understood as a formal deformation of \(\g\) by Property (ii) of \(\mathcal{A}\), is trivial as a formal deformation, i.e. there exists an \(\widehat{O}_p=\CC[\![u]\!]\)-linear isomorphism \(\xi\colon\widehat{\mathcal{A}}_p \longrightarrow \g[\![u]\!]\), called \emph{formal trivialisation}, of Lie algebras such that the induced isomorphism 
\begin{equation}\label{eq: formal trivialisation}
    \g \cong \widehat{\mathcal{A}}_p/\mathfrak{m}_p\widehat{\mathcal{A}}_p \longrightarrow \mathcal{A}|_p  \cong \g
\end{equation}
is the identity. We obtain an induced Lie algebra isomorphism \(Q(\widehat{
\mathcal{A}}_p) = \widehat{ 
\mathcal{A}
}_p \ot_{\widehat{Q}_p} \widehat{Q}_p \longrightarrow \g(\!(u)\!)\) via the \(\CC(\!(u)\!)\)-linear extension of \(\xi\), which we denote by the same symbol. We write the image of \(\Gamma(E_\circ,\mathcal{A}) \subseteq Q(\widehat{
\mathcal{A}}_p)\) under \(\xi\) by \(\g(\rho) \subseteq \g(\!(u)\!)\). 

Note that \(\g(\!(u)\!)\) is equipped with the invariant non-degenerate \(\CC\)-bilinear form
\begin{equation}
B_p(f,g)
\coloneqq
\textnormal{res}_0 \left[ \K(f,g) du \right]
\qquad
f,g \in \g(\!(u)\!).
\end{equation}

\begin{theorem}[{\cite[Proposition 6.1 \& Theorem 6.4]{burban_galinat}}]\label{thm: geomrmat in p} \((\g(\!(u)\!),\g(\rho),\g[\![u]\!])\) is a Manin triple and the Taylor expansion in the second coordinate with respect to \(u\) in \(p\) gives an injection 
\begin{equation}
\Gamma(\Breve{E}\times \Breve{E},\Breve{\mathcal{A}} \boxtimes \Breve{\mathcal{A}}(D)) \longrightarrow (\g \ot \g)(\!(x)\!)[\![y]\!]
\end{equation}
which maps \(\rho\) to \(\sum_{k = 0}^\infty \sum_{\ell = 1}^n f_{k\ell} \ot y^kb_\ell\),
where \(\{b_\ell\}\) is a basis of \(\g\) and \(\{f_{k\ell}\}\) is the basis of \(\g(\rho) \subseteq \g(\!(u)\!)\), uniquely determined by \(B_p(f_{k\ell},u^{k'}b_{\ell'}) = \delta_{k\ell}\delta_{k'\ell'}\).\qed
\end{theorem}

\begin{remark}\label{rem: taylor series in p}
Let us clarify what we mean by Taylor expansion in the second coordinate. Let \(P_k = \textnormal{Spec}(\widehat{O}_p/\mathfrak{m}^k_p\widehat{O}_p)\) and \(\iota_k \colon P_k \longrightarrow E\) be the injection, mapping the closed point of \(P_k\) to \(p\). Then we can consider the pull-back with respect to \(\textnormal{id}_{\Breve{E}\setminus\{p\}}\times \iota_k\) to obtain the morphism
\begin{align}
    \Gamma(\Breve{E}\times \Breve{E},\Breve{\mathcal{A}} \boxtimes \Breve{\mathcal{A}}(D)) \longrightarrow \Gamma(\Breve{E}\setminus \{p\} \times P_k,\Breve{\mathcal{A}} \boxtimes \Breve{\mathcal{A}}(D)) \cong  \Gamma(\Breve{E}\setminus \{p\},\Breve{\mathcal{A}}) \otimes\widehat{\mathcal{A}}_p/\mathfrak{m}^k_p\widehat{\mathcal{A}}_p ,
\end{align}
where we have used that \((\Breve{E}\setminus \{p\} \times P_k)\cap D) = \emptyset\) and applied the Künneth isomorphism. Mapping \(\Gamma(\Breve{E}\setminus \{p\},\Breve{\mathcal{A}})\) via \(\xi\) to \(\g(\!(u)\!)\), using \(\widehat{\mathcal{A}}_p/\mathfrak{m}^k_p\widehat{\mathcal{A}}_p \cong \g[u]/u^k\g[u]\) and applying the projective limit with respect to \(k\), yields the desired injection  
\(\Gamma(\Breve{E}\times \Breve{E},\Breve{\mathcal{A}} \boxtimes \Breve{\mathcal{A}}(D)) \longrightarrow (\g \ot \g)(\!(x)\!)[\![y]\!]\).
\end{remark}

The theorem suggests, that the geometric \(r\)-matrix \(\rho\) actually determines \(\mathcal{A}\) completely. Our next goal is to formalise this idea. The construction we present is known in other situations, see e.g. \cite{mumford}. 
The algebras \(O\) and \(\g(\rho)\) inherit
the ascending filtrations from the natural filtrations of \(\CC(\!(u)\!)\) and \(\g(\!(u)\!)\), namely we have 
\(O_j \coloneqq O \cap u^{-j}\CC[\![u]\!]\), \(\g(\rho)_j \coloneqq \g(\rho) \cap u^{-j}\g[\![u]\!]\)
and
\begin{align}
    \ldots = 0 = O_{-1} \subseteq \CC = O_0 \subseteq O_1 \subseteq \ldots,
    \qquad
    \ldots = 0 = \g(\rho)_0 \subseteq \g(\rho)_1 \subseteq \g(\rho)_2 \subseteq \ldots,
\end{align}
such that \(O_jO_k \subseteq O_{k+j}, O_j\g(\rho)_k \subseteq \g(\rho)_{j+k}\) and \([\g(\rho)_j,\g(\rho)_k] \subseteq \g(\rho)_{j+k}\).
Therefore, we can consider the associated graded objects%
\footnote{These should not be confused with the graded objects associated to modules with a descending filtration.} \(\textnormal{gr}(O)\) and \(\textnormal{gr}(\g(\rho))\), given by
\begin{equation}
\textnormal{gr}(O) \coloneqq \bigoplus^\infty_{j = 0}O_j
\ \text{ and } \ 
\textnormal{gr}(\g(\rho)) \coloneqq \bigoplus_{j = 0}^\infty\g(\rho)_j.
\end{equation}
Note that \(\textnormal{gr}(\g(\rho))\) is a graded Lie algebra over the graded \(\CC\)-algebra \(\textnormal{gr}(O)\). Let us denote by \(\textnormal{gr}(\g(\rho))^{\sim}\) the associated quasi-coherent sheaf of Lie algebras on \(\textnormal{Proj}(\textnormal{gr}(O))\) (see e.g. \cite[Section II.5]{hartshorne}).
\begin{lemma} \label{lemma: rho_determines_A}
  We have \(E = \textnormal{Proj}(\textnormal{gr}(O))\) and the formal trivialisation \(\xi\) induces an isomorphism \( \mathcal{A} \longrightarrow \textnormal{gr}(\g(\rho))^{\sim}\) of sheaves of Lie algebras, which we again denote by \(\xi\).
\end{lemma}
\begin{proof}
We can view \(E_\circ = \textnormal{Spec}(O)\) as an affine open subscheme of \(\textnormal{Proj}(\textnormal{gr}(O))\) by identifying any prime ideal \(\mathfrak{p}\) of \(O\) (which inherits a natural filtration) with \(\textnormal{gr}(\mathfrak{p})\). Under this identification we have
\(\Gamma(E_\circ,\textnormal{gr}(\g(\rho))^{\sim}) = \g(\rho)\), which is most easily seen by using the definitions in \cite[Section II.5]{hartshorne}. In paticular, \(\xi \colon \Gamma(E_\circ,\mathcal{A}) \longrightarrow \g(\rho) = \Gamma(E_\circ,\textnormal{gr}(\g(\rho))^{\sim})\) is an isomorphism of Lie algebras over \(O\).

By \cite[Proposition 3]{parshin} for any coherent sheaf \(\mathcal{F}\) on an open neighbourhood \(U\) of \(p\) the sequence
\begin{align}
    0 \longrightarrow \Gamma(U,\mathcal{F}) \longrightarrow \Gamma(U\setminus\{p\},\mathcal{F})\oplus \hat{\mathcal{F}}_p \longrightarrow \hat{\mathcal{F}}_p \ot_{\widehat{O}_p}\widehat{Q}_p
\end{align}
is exact. This implies \(\Gamma(U,\mathcal{F}) = \Gamma(U\setminus\{p\},\mathcal{F})\cap \hat{\mathcal{F}}_p\).

Thus, for any \(a\in O_j \setminus O_{j-1}\) we have that \(\textnormal{D}(a)\cup \{p\}\) is an affine (see \cite[Proposition 5]{goodman}) open neighbourhood of \(p\) in \(E\) with the coordinate ring \(O[a^{-1}] \cap \CC[\![u]\!] = (\textnormal{gr}(O)[a^{-1}])_0\), where \(a \in \textnormal{gr}(O)\) is taken to have degree \(j\).
Thus we have a natural identification of affine schemes \(E \supseteq \textnormal{D}(a)\cup \{p\} = \textnormal{D}_+(a) \subseteq \textnormal{Proj}(\textnormal{gr}(O))\)
and we see that \(E \subseteq \textnormal{Proj}(\textnormal{gr}(O))\).
Since now \(E\) is a projective curve identified with an open subset in the projective curve \(\textnormal{Proj}(\textnormal{gr}(O))\),
\cite[Proposition 1]{goodman} implies
the equality \(E = \textnormal{Proj}(\textnormal{gr}(O))\).
\iffalse
we have equality
\(E = \textnormal{Proj}(\textnormal{gr}(O))\)
by e.g. \cite[Proposition 1]{goodman}.
\fi

Finally, by definition we have \(\Gamma(\textnormal{D}_+(a), \textnormal{gr}(\g(\rho))^{\sim}) = \Gamma( \textnormal{D}(a)\cup \{p\},\textnormal{gr}(\g(\rho))^{\sim}) = (\g(\rho)[a^{-1}])_0 = \g(\rho)[a^{-1}] \cap \g[\![u]\!]\)
and hence we obtain the isomorphism
\begin{equation*}
    \xi\colon \Gamma(\textnormal{D}(a)\cup \{p\},\mathcal{A}) = \Gamma(\textnormal{D}(a),\mathcal{A})\cap \hat{\mathcal{A}}_p \longrightarrow \g(\rho)[a^{-1}] \cap \g[\![u]\!]=\Gamma(\textnormal{D}(a)\cup \{p\},\textnormal{gr}(\g(\rho))^{\sim}).
\end{equation*}
This ends the proof, because \(E = \textnormal{Proj}(\textnormal{gr}(O)) = E_\circ \cup \textnormal{D}_+(a)\).
\end{proof}

\subsection{Extension property of formal local equivalences}

Now let us consider two coherent sheaves of Lie algebras \(\mathcal{A}_1\) and \(\mathcal{A}_2\) on \(E\) satisfying the conditions (i) - (iii) of Section \ref{sec: survey on geometry} and denote by \(\rho_1\) and \(\rho_2\) the corresponding geometric \(r\)-matrices. Fix formal trivialisations \(\xi_i\) of \(\mathcal{A}_i\) at \(p\) and consider the corresponding isomorphisms \(\xi_i \colon \mathcal{A}_i \longrightarrow \textnormal{gr}(\g(\rho_i))^\sim\) for \(i = 1,2\), where \(\g(\rho_i)= \textnormal{Span}_{\CC}(\{f_{k\ell}^{(i)}\}) \subseteq \g(\!(u)\!)\) is the image of \(\Gamma(E_\circ,\mathcal{A}_i)\) and
\begin{equation}
    \sum_{k = 0}^\infty \sum_{\ell = 1}^n f^{(i)}_{k\ell} \ot y^kb_\ell \in (\g \otimes \g)(\!(x)\!)[\![y]\!]
\end{equation}
are the Taylor expansions of \(\rho_i\) described in Theorem \ref{thm: geomrmat in p}.
We are now in a position to show that any formal
equivalence of \(\rho_1\) and \(\rho_2\) at \(p\)
extends to a global isomorphism of the
corresponding sheaves.

\begin{theorem}\label{thm: lifting equivalences}
Let \(\phi \colon \g[\![u]\!] \longrightarrow \g[\![u]\!]\) be a \(\CC[\![u]\!]\)-linear automorphism of Lie algebras such that 
\begin{equation}\label{eq: local equivalence}
    \sum_{k = 0}^\infty\sum_{\ell = 1}^n \phi(f^{(1)}_{k\ell}) \ot \phi(y^kb_\ell) = \sum_{k = 0}^\infty\sum_{\ell = 1}^n f^{(2)}_{k\ell} \ot y^kb_\ell,
\end{equation}
where we consider the \(\CC(\!(u)\!)\)-linear expansion of \(\phi\) in the first tensor factor.
Then there is an isomorphism \(\psi \colon \mathcal{A}_1 \longrightarrow \mathcal{A}_2\) of coherent sheaves of Lie algebras such that
\(\xi_2\hat{\psi}_p\xi_1^{-1} = \phi\) and \((\psi \boxtimes \psi) \rho_1 = \rho_2\), where we consider the linear extension with respect to the rational functions on \(\Breve{E} \times \Breve{E}\).
\end{theorem}
\begin{proof}
Write \(\phi = \sum_{j = 0}^\infty u^j \phi_j \in \textnormal{End}(\g)[\![u]\!]\). Then
\begin{equation}
\begin{split}
    \sum_{k = 0}^\infty\sum_{\ell = 1}^n \phi(f^{(1)}_{k\ell}) \ot \phi(y^kb_\ell) &= \sum_{k = 0}^\infty\sum_{\ell = 1}^n\sum_{j = 0}^k \phi(f^{(1)}_{(k-j)\ell}) \ot y^k \phi_j(b_\ell)\\&=\sum_{k = 0}^\infty\sum_{\ell' = 1}^n\left(\sum_{\ell = 1}^n\sum_{j = 0}^k a^j_{\ell'\ell}\phi(f^{(1)}_{(k-j)\ell})\right) \ot y^k b_{\ell'},
\end{split}
\end{equation}
where \(\phi_j(b_\ell) = \sum^n_{\ell' = 1} a^j_{\ell'\ell}b_{\ell'}\). 
This shows that \(\g(\rho_2) \subseteq \phi(\g(\rho_1))\) by comparing coefficients in \eqref{eq: local equivalence}. Since \(\g[\![u]\!] \add \g(\rho_2) = \g[\![u]\!] \add \phi(\g(\rho_1))\), we see that
\(\g(\rho_2) = \phi(\g(\rho_1))\).

Clearly the \(\CC(\!(u)\!)\)-linear extension of \(\phi\) preserves the filtration of \(\g(\!(u)\!)\) and hence induces a graded isomorphism of Lie algebras \(\phi \colon \textnormal{gr}(\g(\rho_1)) \longrightarrow \textnormal{gr}(\g(\rho_2))\). Since the procedure \((\cdot)^\sim\) of associating a quasi-coherent sheaf to a graded module on \(E = \textnormal{Proj}(\textnormal{gr}(O))\) is functorial, we get an isomorphism \(
\phi^\sim \colon \textnormal{gr}(\g(\rho_1))^\sim\longrightarrow \textnormal{gr}(\g(\rho_2))^\sim\). Thus we can define \(\psi \coloneqq \xi_2^{-1}\phi^\sim\xi_1 \colon \mathcal{A}_1 \longrightarrow \mathcal{A}_2\). Applying \cite[Lemma 1.10]{galinat_thesis} we see that \((\psi \boxtimes \psi) \rho_1 = \rho_2\). 
\end{proof}

\begin{remark}\label{rem: psi value at p}
Since the induced isomorphisms in equation \eqref{eq: formal trivialisation} of the formal trivialisations \(\xi_i\) (\(i = 1,2\)) give the identity,
we have \(\phi_0 = \psi|_p\).
\end{remark}

\subsection{Proof of the main classification theorem}
Fix an automorphism \( \sigma \in \textnormal{Aut}_{\CC - \textnormal{LieAlg}}(\g) \) of finite order \( m \) and an outer automorphism \( \nu \) from the coset \( \sigma \textnormal{Inn}_{\CC - \textnormal{LieAlg}} (\g) \).
Let \(t \in \Loop^\sigma \ot \Loop^\sigma\)
be a classical twist of the standard
Lie bialgebra structure \(\delta^\sigma_0\)
on \( \Loop^\sigma \).
%where \( \sigma \) is inside the coset of
%the outer automorphism \( \nu \).
In view of Theorem \ref{thm: Gauge equivalence}, to prove the first part of Theorem \ref{thm: classification} we need to show that there exists a regular equivalence \(\phi \in\textnormal{Aut}_{O^\sigma-\textnormal{LieAlg}}(\Loop^\sigma)\), taking values in \(\textnormal{Inn}_{\CC - \textnormal{LieAlg}} (\g) \),
and a BD quadruple \( Q \)
such that
\begin{equation}
(\psi(x) \ot \psi(y))r^\sigma_t(x,y)
=
r^\sigma_Q(x,y)
\end{equation}
for all \(x,y\in\CC^*, x^{m} \ne y^{m}\).
Combining the results of Section \ref{sec: pseudoquasitriangular structure} and \cite{belavin_drinfeld_solutions_of_CYBE_paper}, we get the following statement.

\begin{lemma}\label{lem: holomorphic equivalence t,Q}
  There exists a holomorphic function \(\phi \colon \CC \longrightarrow \textnormal{Inn}_{\CC - \textnormal{LieAlg}}(\g)\) such that
  \begin{equation}\label{eq: holomorphic equivalence t,Q}
      (\phi(u) \ot \phi(v))r^\sigma_t(e^{u/{m}}, e^{v/{m}}) 
      =
      r^\sigma_Q(e^{u/{m}}, e^{v/{m}}).
  \end{equation}
\end{lemma}

\begin{proof}
  By Theorem \ref{thm: r(x,y)=X(u-v)} and
its proof there exists a holomorphic function 
\(\phi_1 \colon \CC \longrightarrow \textnormal{Inn}_{\CC - \textnormal{LieAlg}}(\g)\)
and a trigonometric (in the sense of the Belavin-Drinfeld classification) \rmatr \( X \)
such that
\begin{equation}
\phi_1(0) = \id_\g
\
\text{ and }
\
X(u-v) = (\phi_1(u) \ot \phi_2(y))
r^\sigma_t(e^{u/m},e^{v/m}).
\end{equation}
Furthermore,
it is shown in \cite{belavin_drinfeld_solutions_of_CYBE_paper}
that there is a holomorphic function
\( \phi_2 \colon \CC \longrightarrow
\textnormal{Inn}_{\CC - \textnormal{LieAlg}}(\g)
\)
such that
\begin{equation}
\phi_2(0) = \id_\g
\
\text{ and }
\
(\phi_2(u) \ot \phi_2(v)) X(u-v)
=
r^{\sigma_{(\mathbf{1}, \ord(\nu))}}_Q(e^{u/h}, e^{v/h}),
\end{equation}
where \(h := |\sigma_{(\mathbf{1}, \ord(\nu))}|\).
Combining these results and applying the regrading
scheme from Lemma \ref{thm: GxG r = r} we get the desired holomorphic function.
\end{proof}

Our next goal is to apply Theorem \ref{thm: lifting equivalences} to this holomorphic equivalence to obtain a regular one. Therefore, we need sheaves which give rise to the \(\sigma\)-trigonometric \rmatrs from \eqref{eq: holomorphic equivalence t,Q}. These were constructed in \cite{TBA}. 

\begin{theorem}[{\cite[Theorem 6.9]{TBA}}]
Let \(t\) be a classical twist of \(\delta_0^\sigma\), \(E\) be a nodal Weierstraß cubic with global nonvanishing 1-form \(\omega = d(z^{m})/z^{m}\) under the identification \(\Breve{E} = \textnormal{Spec}(\CC[z^m,z^{-m}])\). Then there exists a coherent sheaf of Lie algebras \(\mathcal{A}_t\) on \(E\), satisfying properties (i)-(iii) of Section \ref{sec: survey on geometry}, such that
\begin{enumerate}
    \item \(\Gamma(\Breve{E},\mathcal{A}_t) = \Loop^\sigma\) and
    \item the isomorphism 
    \begin{equation}
    \Gamma(\Breve{E}\otimes\Breve{E},\Breve{\mathcal{A}}_t \boxtimes \Breve{\mathcal{A}}_t(D)) \cong \left(\frac{1}{(x/y)^m-1}\right) \Loop^\sigma \ot \Loop^\sigma.
    \end{equation}
    maps the geometric \rmatr \(\rho_t\) of \(\mathcal{A}\) to \(r_t^\sigma\).\qed
\end{enumerate}
\end{theorem}

We may identify the smooth point at infinity with \(1 \in \CC^* = \Breve{E} = \textnormal{Spec}(\CC[z^m,z^{-m}])\). The algebra homomorphism \(\CC[z^{m},z^{-m}] \longrightarrow \CC[\![u]\!]\) given by
\begin{align}
    z^m \longmapsto e^u = \sum_{k = 0}^\infty \frac{u^n}{n!}
\end{align}
induces an identification \(\widehat{\mathcal{O}}_{E,p} = \CC[\![u]\!]\) in such a way that \(u(1) = \textnormal{log}(1) = 0\) and \(\widehat{\omega}_p = e^{-u}d(e^u) = du\). Thus \(u\) is the formal local parameter used in the setting of Theorem \ref{thm: geomrmat in p} and onwards. Now the \(\CC[z^{m},z^{-m}]\)-\(\CC[\![u]\!]\)-equivariant Lie algebra morphism \( \Loop^\sigma \longrightarrow \g[\![u]\!]\) given by \(f(z) \longmapsto f(e^{u/m})\) 
induces a formal trivialization \(\xi\colon \widehat{\mathcal{A}}_p \longrightarrow \g[\![u]\!]\). With this choice,
we can interpret the series expansion of \(\rho_t\), described in Remark \ref{rem: taylor series in p}, as the Taylor series of \(r_t^\sigma(e^{u/m},e^{v/m})\) at \(v = 0\). 
Therefore, we can apply Theorem \ref{thm: lifting equivalences} to the Taylor expansion of \eqref{eq: holomorphic equivalence t,Q} at \(v = 0\) to obtain an isomorphism \(\psi \colon \mathcal{A}_t \longrightarrow \mathcal{A}_Q,\) satisfying \((\psi \boxtimes \psi)\rho_t = \rho_Q\), where \(\mathcal{A}_Q\ \coloneqq \mathcal{A}_{t_Q} \) and \(\rho_{Q} \coloneqq \rho_{t_Q}\).

We obtain a regular equivalence
\(\psi \colon \Loop^\sigma = \Gamma(\Breve{E},\mathcal{A}_t) \longrightarrow\Gamma(\Breve{E},\mathcal{A}_Q) = \Loop^\sigma\) by applying \(\Gamma(\Breve{E},-)\).
Note that by Remark \ref{rem: psi value at p} and Lemma \ref{lem: holomorphic equivalence t,Q} the function \(\psi\) actually takes values in 
\(
\textnormal{Inn}_{\CC - \textnormal{LieAlg}}(\g)\). Using the commutative diagram 
\begin{equation}
    \xymatrix{\Gamma(\Breve{E}\times \Breve{E},\Breve{\mathcal{A}} \boxtimes \Breve{\mathcal{A}}) \ar[d]_\cong \ar[r]^{\psi \boxtimes \psi}& \Gamma(\Breve{E}\times \Breve{E},\Breve{\mathcal{A}} \boxtimes \Breve{\mathcal{A}}) \ar[d]^\cong \\ \Loop^\sigma \ot \Loop^\sigma \ar[r]_{\psi \ot \psi} & \Loop^\sigma \ot \Loop^\sigma},
\end{equation}
where the vertical isomorphisms are given by the Künneth theorem, we obtain the desired identity
\begin{align}
    (\phi(x)\ot \phi(y))r^\sigma_t(x,y) = r^\sigma_Q(x,y).
\end{align}

We finish the proof of the main theorem by explaining when two BD quadruples give rise to equivalent twisted standard structures.

\begin{lemma}
  Let \(Q = (\Gamma_1,\Gamma_2,\gamma,t_\h)\) and \(Q'=(\Gamma'_1,\Gamma'_2,\gamma',t'_\h)\) be two BD quadruples. Then \(\delta_Q^\sigma\) and \(\delta^\sigma_{Q'}\) are regularly equivalent if and only if there exists
  an automorphism \(\vartheta\) of the Dynkin diagram of \(\Loop^\sigma\) such that \(\vartheta(Q) = Q'\).
\end{lemma}
\begin{proof}
 To simplify the notations we assume
\(\sigma = \sigma_{(s,|\nu|)}\) for
\(s = (s_0,\dots,s_n)\). The general result follows from Remark \ref{rem: L^nu is determined by ord(nu)}.

\emph{''\(\implies\)''\,:} Let \(\phi\) be a regular equivalence between \(\delta^\sigma_Q\) and \(\delta^\sigma_{Q'}\). The identity \begin{equation}
(\phi(x) \ot \phi(y))r_Q^\sigma(x,y) = r_{Q'}^\sigma(x,y),
\end{equation}
given by Theorem \ref{thm: Gauge equivalence},
and the equation \eqref{eq: def eq for r_t} imply \(\phi R_Q \phi^{-1} = R_{Q'}\).
%, where we used \(\phi^* = \phi^{-1}\).
If \( GE_0 \) and \( GE'_0 \) are the generalized eigenspaces of \( R_Q \) and \( R_{Q'} \) respectively, corresponding to the common eigenvalue 0, then \( \phi(GE_0) = GE'_0 \).
Since both \(\theta^+_\gamma \) and \(\theta^+_{\gamma'} \) are nilpotent, the identity \eqref{eq: R_Q} implies the equality of normalizers \( \textnormal{N}_{\Loop^\sigma}(GE_0) =\textnormal{N}_{\Loop^\sigma}(GE'_0) = \B_+ \). Therefore, \( \phi \) is an automorphism of \( \Loop^\sigma \) fixing the Borel subalgebra \( \B_+ \).
From the third part of Lemma \ref{lem: structure parabolic}, we see that \( \phi \) induces an automorphism \( \vartheta \)  of the Dynkin diagram of \( \Loop^\sigma \).
The identity \( \vartheta(Q) = Q' \) follows from Theorem \ref{thm: Gauge equivalence} and formula \eqref{eq: W_Q}.
  
\emph{''\(\impliedby\)''\,:} Let \(\vartheta\) be a Dynkin diagram automorphism, such that
\(\vartheta(Q) = Q'\).
We want to show that \(\vartheta\) defines a regular equivalence \(\phi\) on \(\Loop^\sigma\), such that \(\phi(\Loop^\sigma_{(\alpha,k)}) = \Loop^\sigma_{\vartheta(\alpha,k)}\) for any root \((\alpha,k)\) of \(\Loop^\sigma\).
Let \( \phi' \) be the automorphism defined by \(\phi'(z^{\pm s_i}X^\pm_i(1)) \coloneqq z^{\pm s_{\vartheta(i)}}X^\pm_{\vartheta(i)}(1)\).
Then it maps the root spaces onto each other in the desired way.
We now adjust \(\phi'\) to be \(\CC[z^m,z^{-m}]\)-linear. By \cite[Lemma 8.6]{kac_book} we have the equality \(\phi' ((z^{m}-1)\Loop^\sigma) = ((z/a)^{m} - 1) \Loop^\sigma\) for some \(a \in \CC^*\). 
Let \(\mu_a\)  be the automorphism of \(\Loop^\sigma\) given by \(\mu_a(f(z)) \coloneqq f(az)\). Note that it preserves the root spaces of \(\Loop^\sigma\). Define \( \phi \coloneqq \mu_a \phi' \), then
\(\phi((z^{m}-1)\Loop^\sigma) = (z^{m} - 1) \Loop^\sigma\) and thus
\begin{equation}
      \phi(z^{\pm s_i +m}X_i^\pm(1)) = z^m\phi(z^{\pm s_i}X_{i}^\pm(1))
\end{equation}
implying the \(\CC[z^m,z^{-m}]\)-linearity.
\end{proof}

\newpage
\printbibliography

@article{drinfeld_hopf_algebras,
	author = {Drinfeld, V.},
	title = {{Hopf algebras and the quantum Yang-Baxter equation}},
	journal = {Sov.Math.Dokl},
	year = {1985}
}

@article{montaner_stolin_zelmanov,
	author = {Montaner, F. and Stolin, A. and Zelmanov, E.},
	title = {{Classification of Lie bialgebras over current algebras}},
	journal = {Sel. Math. New. Ser.},
	year = {2010},
}

@phdthesis{galinat_thesis,
          school = {Universit{\"a}t zu K{\"o}ln},
           title = {Algebro-Geometric Aspects of the Classical Yang-Baxter Equation},
            year = {2015},
          author = {Galinat, L.},
}

@Article{kac_finite_order_automorphisms,
 Author = {Kac, V.},
 Title = {Automorphisms of finite order of semisimple Lie algebras},
 Journal = {Funct. Anal. Appl.},
 Year = {1969},
}

@book{gunning_rossi_1965, place={Englewood-Cliffs, NJ}, title={Analytic functions of several complex variables}, publisher={Prentice-Hall}, author={Gunning, R. and Rossi, H.}, year={1965}}

@book{hazewinkel_gerstenhaber,
	author = {Hazewinkel, M. and Gerstenhaber, M.},
	title = {Deformation Theory of Algebras and Structures and Applications},
	year = {1988},
	publisher = {Springer Netherlands},
}

@article{parshin,
author = {Parshin, A.},
title = {Integrable systems and local fields},
journal = {Communications in Algebra},
year  = {2001},
}

@book{bourbaki_lieIII,
  title={Lie Groups and Lie Algebras: Chapters 7-9},
  author={Bourbaki, N.},
  series={Elements of mathematics},
  year={2005},
  publisher={Springer-Verlag Berlin Heidelberg}
}

@article{jimbo,
	author = {Jimbo, M.},
	title = {{Quantum \(R\) matrix for the generalized Toda system}},
	journal = {Commun. Math. Phys.},
	year = {1986},
	publisher = {Springer-Verlag},
}

@article{kulish,
	author = {Kulish, P.},
	title = {{Quantum difference nonlinear Schr{\"o}dinger equation}},
	journal = {Lett. Math. Phys.},
	year = {1981},
	publisher = {Kluwer Academic Publishers},
}

@book{carter, 
    series={Cambridge Studies in Advanced Mathematics}, title={Lie Algebras of Finite and Affine Type}, publisher={Cambridge University Press}, 
    author={Carter, R.}, 
    year={2005}
}

@inproceedings{kosmann-schwarzbach,
author="Kosmann-Schwarzbach, Y.",
title="Lie bialgebras, Poisson Lie groups and dressing transformations",
booktitle="Integrability of Nonlinear Systems",
year="1997",
publisher="Springer Berlin Heidelberg",
SERIES = {Lecture Notes in Phys., Vol. 495, 104 -- 170},
YEAR = {1997}
}

@article{karolinsky_Stolin,
	author = {Karolinsky, E. and Stolin, A.},
	title = {{Classical Dynamical r-Matrices, Poisson Homogeneous Spaces, and Lagrangian Subalgebras}},
	journal = {Lett. Math. Phys.},
	year = {2002},
}

@inproceedings {mumford,
    AUTHOR = {Mumford, D.},
     TITLE = {An algebro-geometric construction of commuting operators and
              of solutions to the {T}oda lattice equation, {K}orteweg
              de{V}ries equation and related nonlinear equation},
 publisher = {Proc. of the Intl. Symp. Algebraic Geometry, Kyoto},
      YEAR = {1978}
}

@article{burban_galinat_stolin,
author = {Burban, I. and Galinat, L. and Stolin, A},
year = {2017},
title = {Simple vector bundles on a nodal Weierstrass cubic and quasi-trigonometric solutions of CYBE},
journal = {Journal of Physics A: Mathematical and Theoretical}
}

@article{drinfeld_quantum_groups,
    title={Quantum groups},
    author={Drinfeld, V.},
    year={1983},
    journal={Journal of Soviet Mathematics 41, 898–915}
}

@unpublished{TBA,
author={Abedin, R. and Burban, I.},
title={Algebraic geometry of Lie bialgebras defined by solutions of the classical Yang-Baxter equation}
}

@article{drinfeld_hamiltonian_structures,
    title={Hamiltonian structures on Lie groups, Lie bialgebras and the geometric meaning of the classical Yang–Baxter equations},
    author={Drinfeld, V.},
    year={1983},
    journal={Dokl. Akad. Nauk SSSR, 268:2}
}

@article{pop_stolin,
	author = {Pop, I. and Stolin, A.},
	title = {{Lagrangian Subalgebras and Quasi-trigonometric r-Matrices}},
	journal = {Lett. Math. Phys., Vol. 85, No. 2, 249 -- 262},
	year = {2008},
	publisher = {Springer Netherlands}
}

@article{goodman,
    title={Affine open subsets of algebraic varieties and ample divisors},
    author = {Goodman, J.},
    year = {1967},
    journal = {Annals of Mathematics, Jan. 1969, Second Series, Vol. 89, No 1, 160-183}
}

@book{etingof_schiffmann,
    title = {Lectures on Quantum Groups (Second Edition)},
    author = {Etingof, P. and Schiffmann, O.},
    year = {2002},
    publisher = {Lectures in Mathematical Physics}
}

@book{kac_book,
    title = {Infinite dimensional Lie algebras (3rd Edition)},
    author = {Kac, V.},
    year = {1990},
    publisher = {Cambridge University Press}
}

@article{kac_wang,
    
    title="On Automorphisms of Kac-Moody Algebras and Groups",
    
    author="Kac, V. and Wang, S.",
    
    journal="Advances in Mathematics 92, 129-195",
    
    year="1992",
}

@ARTICLE{KPSST,
       author = {Khoroshkin, S. and Pop, I. and Samsonov, M. and
         Stolin, A. and Tolstoy, V.},
        title = "{On Some Lie Bialgebra Structures on Polynomial Algebras and their Quantization}",
      journal = {Communications in Mathematical Physics},
         year = "2008",
}

@article{belavin_drinfeld_solutions_of_CYBE_paper,

  	title="Solutions of the classical Yang-Baxter equation for simple Lie algebras",

  	author="Belavin, A. and Drinfeld, V.",
  
	journal = "Funct. Anal. Appl. 16, No. 3",
  
	year="1983",

}

@article{belavin_drinfeld_diffrence_depending,

  	title="The classical Yang-Baxter equation for simple Lie algebras",

  	author="Belavin, A. and Drinfeld, V.",
  
	journal = "Funct. Anal. Appl. 17, No. 3",
  
	year="1983",

}

@book{barth_et_al,
  
	title="Compact Complex Surfaces",
  
	author="Barth, W. and Hulek, K. and Peters, C. and van de Ven, A.",
	series="Ergebnisse der Mathematik und ihrer Grenzgebiete. 3. Folge",
  

  	year="2015",
  
	publisher="Springer Berlin Heidelberg"

}

@article{stolin_sln,

 	author = "Stolin, A.",
 
	journal = "Mathematica Scandinavica",

	title = "On rational solutions of Yang-Baxter equation for $\mathfrak{sl}(n)$",
	year = "1991",

}

@article{stolin_maximal_orders,
	author = "Stolin, A.",
	journal = "Comm. Math. Phys.",
	title = "On rational solutions of Yang-Baxter equations. Maximal orders in loop algebra",
	year = "1991"
}

@article{burban_galinat,
	author = {Burban, I. and Galinat, L.},
	title = {{Torsion Free Sheaves on Weierstrass Cubic Curves and the Classical Yang{\textendash}Baxter Equation}},
	journal = {Commun. Math. Phys.},
	year = {2018},
}

@book{hartshorne,
  
	title="Algebraic Geometry",
  
	author="Hartshorne, R.",
   
	series="Graduate Texts in Mathematics",
 
	year="1977",
  
	publisher="Springer",

}

@book{belavin_dirnfeld_triangle,
  	title="Triangle Equations and Simple Lie Algebras",
  	author="Belavin, A. and Drinfeld, V.",
	series="Soviet scientific reviews: Mathematical physics reviews",
	year="1984",
	publisher="Harwood Academic Publishers"
}

@book{chari_pressley,
	title="A Guide to Quantum Groups",
	author="Chari, V. and Pressley, A.",
	year="1995",
	publisher="Cambridge University Press"
}

@book{helgason,
  title={Differential Geometry, Lie Groups, and Symmetric Spaces},
  author={Helgason, S.},
  series={Pure and Applied Mathematics},
  year={1978},
  publisher={Elsevier Science}
}

\end{document}